\documentclass[11pt,reqno]{amsart}
\usepackage[utf8]{inputenc}
\usepackage[T1]{fontenc}
\usepackage{amsmath,amssymb,amsthm,amsfonts,mathrsfs}
\usepackage{hyperref}
\usepackage{geometry}
\usepackage{tikz}
\usepackage{pgfplots}
\usepackage{caption}
\usepackage{mathtools}
\pgfplotsset{compat=1.17}
\usetikzlibrary{calc}
\usepackage{pstricks}
\usepackage{pst-plot}
\usepackage{pst-node}
\usepackage{bm}
\usepackage{enumitem}
\usepackage{array}
\usepackage{makecell}
\usepackage{color}
\allowdisplaybreaks[4]
\def\blue{\textcolor{blue}}
\def\red{\textcolor{red}}
\def\green{\textcolor{green}}
\numberwithin{equation}{section}
\newtheorem{theorem}{Theorem}[section]
\newtheorem{lemma}[theorem]{Lemma}
\newtheorem{proposition}[theorem]{Proposition}
\newtheorem{corollary}[theorem]{Corollary}

\newtheorem{definition}{Definition}[section]
\newtheorem{example}[]{Example}

\newtheorem{remark}{Remark}[section]

\newcommand{\R}{\mathbb{R}}

\newcommand{\N}{\mathbb{N}}
\renewcommand{\S}{\mathfrak{S}}
\renewcommand{\L}{\mathcal{L}}
\mathchardef\mhyphen="2D

\def\lmival{\mathrm{lminval}}
\def\rmival{\mathrm{rminval}}
\def\lrmival{\mathrm{lrminval}}
\def\s{\sigma}
\def\a{\alpha}
\def\dis{\mathrm{D}}
\def\drop{\mathrm{drop}}
\def\asc{\mathrm{asc}}
\def\des{\mathrm{des}}
\def\val{\mathrm{val}}
\def\pk{\mathrm{pk}}
\def\da{\mathrm{da}}
\def\dd{\mathrm{dd}}
\def\dd{\mathrm{dd}}

\def\pmax{\mathrm{lmaxpk}}
\def\bmax{\mathrm{rmaxpk}}
\def\lmaxpk{\mathrm{lmaxpk}}
\def\lrmaxpk{\mathrm{lmaxpk}}
\def\lrmaxda{\mathrm{lmaxda}}
\def\rlmaxdd{\mathrm{rmaxdd}}
\def\lmax{\mathrm{lmax}}
\def\rmax{\mathrm{rmax}}
\def\cros{\mathrm{cros}}
\def\nest{\mathrm{nest}}
\def\fmax{\mathrm{fmax}}
\def\ndes{\mathrm{ndes}}
\def\PRW{\mathrm{PRW}}
\def\MAD{\mathrm{MAD}}
\def\inv{\mathrm{inv}}
\def\wex{\mathrm{wex}}
\def\exc{\mathrm{exc}}
\def\fix{\mathrm{fix}}
\def\NDdif{\mathrm{NDdif}}
\def\NDtop{\mathrm{NDtop}}
\def\NDbot{\mathrm{NDbot}}
\def\depth{\mathrm{depth}}
\def\cpk{\mathrm{cpk}}
\def\cval{\mathrm{cval}}
\def\lbr{(31\!-\!2)}
\def\rbr{(2\!-\!31)}
\def\A{\mathcal{A}}
\def\Web{\mathcal{CA}}

\def\Cda{\mathrm{Cda}}
\def\rlmin{\mathrm{rmin}}
\def\lrmin{\mathrm{lmin}}
\def\rlminda{\mathrm{rminda}}

\def\area{\mathrm{area}}
\def\Val{\mathrm{Val}}
\def\Pk{\mathrm{Pk}}
\def\Cdd{\mathrm{Cdd}}
\def\Cpk{\mathrm{Cpk}}
\def\Cval{\mathrm{Cval}}

\def\Fix{\mathrm{Fix}}
\def\cda{\mathrm{cda}}
\def\cdd{\mathrm{cdd}}

\def\rasc{\mathrm{2\!-\!{13}}}
\def\ldes{\mathrm{{31}\!-\!2}}
\def\wt{\mathbf{wt}}
\def\lrmax{\mathrm{lrmax}}

\usepackage{color}

\title{The $(p,q)$-Analogues of Bi-Stirling Eulerian Polynomials via continued fractions}
\author{Chao Xu}
\address{Universite Lyon 1,
UMR 5208 du CNRS, Institut Camille Jordan\\
F-69622, Villeurbanne Cedex, France}
\email{xu@math.univ-lyon1.fr}
\author{Jiang Zeng}
\address{College of Mathematics and Physics, Wenzhou University, Wenzhou 325035, China;\\
Universit\'e Claude Bernard Lyon 1, CNRS UMR 5208, Institut Camille Jordan, France}
\email{zeng@math.univ-lyon1.fr}

\date{\today}

\begin{document}
\begin{abstract}
We establish Jacobi-type continued fraction expansions for
\((p,q)\)-analogues of the bi-Stirling Eulerian polynomials,
which refine classical Eulerian statistics on permutations.
These polynomials admit equivalent definitions in terms of descents or excedances.
The parameters \(p\) and \(q\) encode refined permutation statistics,
including \((\ldes,\rasc)\) and \((\cros,\nest)\).
Our approach is based on weighted Motzkin path models arising from
variants of the bijections of Françon--Viennot and Foata--Zeilberger.
As applications, we obtain new results on total positivity and
\(\gamma\)-positivity for several families of enumerative polynomials.
Our results unify and extend a number of recent works in the literature.
\end{abstract}

\maketitle

\section{Introduction}

The Eulerian polynomials are fundamental objects in enumerative combinatorics, as
they encode the distribution of ascents and descents over permutations. Over the
years, numerous refinements and generalizations have been introduced, notably the
bi-Stirling Eulerian polynomials, which incorporate record-type statistics,
and the \((p,q)\)-Eulerian polynomials, which refine Eulerian statistics via
generalized permutation patterns. The aim of the present paper is to combine these
two directions by introducing \((p,q)\)-analogues of the
bi-Stirling-Eulerian polynomials.

We now recall the necessary definitions and notation concerning permutation
statistics and Eulerian-type polynomials. For a finite subset \(E\) of \(\N\), let
\(\mathfrak{S}_E\) denote the set of all permutations of \(E\), and let
\(\mathfrak{S}_n\) be the symmetric group on \([n]:=\{1,2,\ldots,n\}\). For
\(\sigma\in\mathfrak{S}_n\), regarded as the word
\(\sigma=\sigma_1\sigma_2\cdots\sigma_n\), where \(\sigma_i=\sigma(i)\) for
\(i\in[n]\), an index \(1\le i\le n-1\) is called
\begin{itemize}
  \item a \emph{descent} (\textbf{des}) if \(\sigma_i>\sigma_{i+1}\);
  \item an \emph{ascent} (\textbf{asc}) if \(\sigma_i<\sigma_{i+1}\).
\end{itemize}
Furthermore, a letter \(\sigma_i\) is said to be
\begin{itemize}
  \item a \emph{left-to-right maximum} (\textbf{lmax}) if
        \(\sigma_i>\sigma_j\) for every \(j<i\);
  \item a \emph{right-to-left maximum} (\textbf{rmax}) if
        \(\sigma_i>\sigma_j\) for every \(j>i\).
\end{itemize}
Carlitz and Scoville~\cite{CS74} introduced the generalized Eulerian polynomials
as the enumerative polynomials of the symmetric group \(\mathfrak{S}_n\)
with respect to ascent, descent, and record statistics, namely
\begin{equation}
\label{eq:alpha-beta-eulerian}
A_n(x,y \mid \alpha,\beta)
= \sum_{\sigma \in \mathfrak{S}_{n+1}}
x^{\asc(\sigma)} y^{\des(\sigma)}
\alpha^{\lmax(\sigma)-1}
\beta^{\rmax(\sigma)-1}.
\end{equation}
Note that $\asc(\sigma)+\des(\sigma)=n$ for $\sigma\in \S_{n+1}$. 
Motivated by their combinatorial interpretation~\cite{XZ26}, we refer to these polynomials as the \emph{bi-Stirling Eulerian polynomials}, although they are termed the $(\alpha,\beta)$-Eulerian polynomials in \cite{Ji25}.

Recently, Ji~\cite{Ji25} introduced a refinement of
\(A_n(x,y \mid \alpha,\beta)\), defined by
\begin{equation}
\label{refine-CS-JI1}
A_n(u_1,u_2,u_3,u_4 \mid \alpha,\beta)
=
\sum_{\sigma\in \mathfrak{S}_{n+1}}
  (u_1u_2)^{\val(\sigma)}
  u_3^{\da(\sigma)}
  u_4^{\dd(\sigma)}
  \alpha^{\lmax(\sigma)-1}
  \beta^{\rmax(\sigma)-1}.
\end{equation}
In the refinement \eqref{refine-CS-JI1}, the four parameters
\(u_1,u_2,u_3,u_4\) distinguish valleys, peaks, double ascents, and double
descents, while the parameters \(\alpha\) and \(\beta\) record left-to-right
and right-to-left maxima.
This refined description will later be incorporated into a single weight
function \(\wt(\sigma)\), together with additional statistics, including a
D-statistic \(\dis(\sigma)\) naturally arising from the
Fran\c{c}on--Viennot bijection.

For \(\sigma=
\sigma_1\cdots\sigma_n\in\mathfrak{S}_n\), with the convention
\(\sigma_0=\sigma_{n+1}=0\), a letter \(\sigma_i\) with \(i\in[n]\) is called a
\begin{itemize}
  \item \emph{valley} (\textbf{val}) if \(\sigma_{i-1}>\sigma_i<\sigma_{i+1}\);
  \item \emph{peak} (\textbf{pk}) if \(\sigma_{i-1}<\sigma_i>\sigma_{i+1}\);
  \item \emph{double ascent} (\textbf{da}) if
        \(\sigma_{i-1}<\sigma_i<\sigma_{i+1}\);
  \item \emph{double descent} (\textbf{dd}) if
        \(\sigma_{i-1}>\sigma_i>\sigma_{i+1}\).
\end{itemize}
These statistics satisfy
\begin{equation}
\label{eq:refinement of asc-des}
\asc=\val+\da,\qquad
\des=\val+\dd,\qquad
\val=\pk-1.
\end{equation}
The exponential generating function (EGF) of $A_n(u_1,u_2,u_3,u_4 \mid \alpha,\beta)$ was obtained by Ji~\cite[Theorem~1.4]{Ji25}. As observed in \cite{XZ26},  Ji's EGF formula  can be written in the following form
\begin{align}
\sum_{n\geq 0} A_n(u_1,u_2,u_3,u_4\,|\,\alpha, \beta)\frac{z^n}{n!}&=e^{(\alpha u_3+\beta u_4)z}\left(\frac{x-y}{x e^{yz}-ye^{xz}}\right)^{\alpha+\beta}\label{Ji-eq1'},
\end{align}
where the parameters $x$ and $y$ are related by $x+y=u_3+u_4$ and $xy=u_1u_2$.
 The sequel~\cite{XZ24} established
a further generalization \eqref{Ji-eq1'} and derived a \(J\)-fraction expansion for the ordinary
generating function.  In particular, building on earlier work of Zeng~\cite{Zen93}, we obtained  the continued fraction formula
\begin{subequations}
\begin{equation}\label{equ:continued-fraction-XZ}
  \sum_{n=0}^\infty A_n(u_1,u_2,u_3, u_4\, |\, \alpha,\beta)\,z^n
 = \cfrac{1}{1-b_0 z-\cfrac{\lambda_1 z^2}{1-b_1 z-\cfrac{\lambda_2 z^2}{1-b_2 z-\cdots}}},
\end{equation}
where  $b_0=\alpha u_3+\beta u_4$, 
\begin{equation}
b_k = k(u_3+u_4) + (\alpha u_3+\beta u_4),
\qquad
\lambda_{k} = k(k-1+\alpha+\beta)u_1u_2
\quad (k\ge 1).
\end{equation}
\end{subequations}


In this paper, we will generalize the above continued fraction formula by 
using the
combinatorial theory of continued fractions developed by Flajolet~\cite{Fla80} and Viennot~\cite{Vie83},
together with variants of the bijections of Fran\c{c}on--Viennot~\cite{FV79} and
Foata--Zeilberger~\cite{FZ90} between permutations and weighted Motzkin paths.
This approach allows us to derive explicit Jacobi-type continued fraction
expansions and to obtain new consequences concerning  positivity properties of the
associated enumerative polynomials.

Our approach is based on continued fractions arising from weighted Motzkin path
models. These models provide transparent combinatorial interpretations and lead
naturally to explicit Jacobi-type continued fraction expansions for the ordinary
generating functions under consideration.

More generally, our work fits naturally into the broad framework connecting
permutation statistics, moment sequences, orthogonal polynomials, and continued
fractions. In many situations, continued fractions arising from permutation
statistics admit an interpretation as moment generating functions of suitable
measures, thereby providing a bridge between enumerative combinatorics and the
analytic theory of moments. This perspective has been developed, among others,
by Blitvi\'c and Steingr\'{\i}msson~\cite{BS21} in their study of permutations,
moments, and measures, and more recently by Sokal and Zeng~\cite{SoZ22} through
the introduction of multivariate master polynomials for permutations, set
partitions, and perfect matchings together with their continued-fraction
expansions. Building on the classical Flajolet--Viennot theory, P\'etr\'eolle,
Sokal, and Zhu~\cite{PSZ23} extended Jacobi-type continued fractions to lattice
paths and branched continued fractions, with applications to coefficientwise
Hankel-total positivity.

In this context, weighted Motzkin paths and Laguerre histories provide a
combinatorial realization of moment sequences, where the step weights encode
refined permutation statistics. The Foata--Zeilberger correspondence and its
variants therefore furnish a natural mechanism for translating permutation
statistics into continued-fraction expansions, and will play a central role in
the bijections and $J$-fraction formulas developed in this paper.
The present work offers further evidence for the relevance of this framework by
introducing new families of Eulerian-type polynomials whose generating
functions admit explicit Jacobi-type continued-fraction expansions.

\section{Main results}
\medskip
   
For  $\sigma=\sigma_1\ldots \sigma_n\in \S_n$ with 
\(\sigma_0=\sigma_{n+1}=0\), a letter $\s_i\in[n]$ is said to be a
\begin{itemize}
 \item \emph{left-to-right-maximum-peak} (\textbf{lmaxpk})  if 
 $\s_i$ is a left-to-right maximum and also a peak;
 \item  \emph{right-to-left-maximum-peak} (\textbf{rmaxpk})  if $\s_i$ is a  right-to-left maximum and also a peak;
 \item  \emph{left-to-right-maximum-double-ascent} (\textbf{lmaxda}) if $\s_i$ is a left-to-right  maximum and also  a double ascent;
 \item   \emph{right-to-left-maximum-double-descent} (\textbf{rmaxdd}) if $\s_i$ is a right-to-left maximum and also a double descent.  
\end{itemize}
\begin{example}
Let $\sigma = 2\,7\,1\,5\,9\,10\,8\,4\,3\,6 \in \mathfrak{S}_{10}$, we have  
\[
\pk(\sigma) = |\{6,7,10\}| = 3,\; 
\val(\sigma) = |\{1,3\}| = 2,\; 
\da(\sigma) = |\{2,5,9\}| = 3,\; 
\dd(\sigma) = |\{4,8\}| = 2,
\]
\[
\lmax(\sigma) = |\{2,7,9,10\}| = 4,\quad 
\rmax(\sigma) = |\{6,8,10\}| = 3,
\]
\[
\lrmaxda(\sigma) = |\{2,9\}| = 2,\quad 
\rlmaxdd(\sigma) = |\{8\}| = 1,
\]
\[
\pmax(\sigma) = |\{7,10\}| = 2, \text{ and }
\bmax(\sigma) = |\{6,10\}| = 2.
\]
\end{example}
 As a follow-up to~\cite{Ji25}
the authors~\cite{XZ24} generalized \eqref{refine-CS-JI1} with three additional variables as follows:
\begin{subequations}\label{def:An}
    \begin{equation}\label{XZ-polynomial}
    A_n(\mathbf{u},f,g,t\,|\, \alpha,\beta)
=\sum_{\sigma\in\S_{n+1}}\wt(\sigma),
\end{equation}
where $\mathbf{u}=(u_1,u_2,u_3,u_4)$ and
\begin{multline}\label{weight-functions}
\wt(\sigma):=(u_1u_2)^{{\val}(\sigma)}u_3^{\da(\sigma)}u_4^{\dd(\sigma)}
\alpha^{{\lmax}(\sigma)-1}{\beta}^{{\rmax}(\sigma)-1}\\
\times f^{\mathrm{lmaxpk(\sigma)-1}}g^{\bmax(\s)-1} t^{\lrmaxda(\s)+\rlmaxdd(\s)}.
\end{multline}
\end{subequations}

Let $\sigma=\sigma_1\ldots \sigma_n\in\S_n$ with $\sigma_0=\sigma_{n+1}=0$.  
Denote by $\Val(\sigma)$ and $\Pk(\sigma)$ the sets of valleys and peaks of $\sigma$, respectively.
We define the $\dis$ statistic  of $\sigma$ by
\begin{equation}\label{eq:distance}
\dis(\sigma)
:=\sum_{i\in \Pk(\sigma)\setminus\{n\}} i \;-\;\sum_{j\in \Val(\sigma)} j.
\end{equation}
This statistic corresponds  to the area under the associated Motzkin
path via the Fran\c{c}on--Viennot bijection, see Lemma~\ref{lem:dis}. 
Note  that    a shifted version of $\dis$ statistic appeared in~\cite{Fra92}.
 
\begin{example}
 Let $\sigma=42157368\in\S_8$. Then $\Val(\sigma)=\{1,3\}$, $\Pk(\sigma)=\{4,7,8\}$, and hence
    $\dis(\sigma)=4+7-1-3=7.$
\end{example}

The notion of generalized (dashed) permutation patterns was introduced by 
Babson and Steingr\'{\i}msson~\cite{BS00}. Recall that the pattern \(31\!-\!2\) (resp. \(2\!-\!13\)) occurs in a permutation \(\sigma\)
if there exist indices \(i<j\) such that the entries in positions \(i\) and
\(i+1\) (resp. \(j\) and \(j+1\)) are consecutive in \(\sigma\), and
\[
\sigma(i)>\sigma(j)>\sigma(i+1)
\quad
(\text{resp. }\sigma(j)<\sigma(i)<\sigma(j+1)).
\]
We denote by \((31\!-\!2)(\sigma)\) (resp. \((2\!-\!13)(\sigma)\)) the number of
occurrences of the pattern \(31\!-\!2\) (resp. \(2\!-\!13\)) in \(\sigma\); see
\cite{SZ12}. Accordingly, we set
\[
\ldes(\sigma):=(31\!-\!2)(\sigma)
\qquad\text{and}\qquad
\rasc(\sigma):=(2\!-\!13)(\sigma).
\]
Similarly we define the pattern  $2\!-\!31$ and the number $(2\!-\!31)(\sigma)$.
\begin{example}
Let \(\sigma=4736215\in\S_7\). Then
\begin{align*}
(31\!-\!2)(\sigma)&=|\{73\!-\!6,\,73\!-\!5,\, 62\!-\!5\}|=3,\\
(2\!-\!31)(\sigma)&=|\{4\!-\!73,\,4\!-\!62,\,3\!-\!62\}|=3,\\
(2\!-\!13)(\sigma)&=|\{4\!-\!36,\,4\!-\!15,\, 3\!-\!15,\,2\!-\!15\}|=4.
\end{align*}
\end{example}

Shin and one of us~\cite{SZ12} proved a continued fraction formula for the ordinary generating function of the following $(p,q)$-analogue of generalized Eulerian polynomials:
 \begin{equation}\label{(p,q)-eulerian}
     A_n(u_1, u_2, u_3, u_4;\,p,q) = \sum_{\sigma \in \S_n}(u_1u_2)^{\val(\sigma)}
  u_3^{\da(\sigma)}
  u_4^{\dd(\sigma)}p^{2\!-\!13( \sigma)} q^{31\!-\!2 (\sigma)}
 \end{equation} 
and derived the $\gamma$-positivity formula using combinatoris of continued fractions.
Combining \eqref{refine-CS-JI1} and \eqref{(p,q)-eulerian}, we define the 
the
$(p,q)$-analogue of the generalized Eulerian polynomials as follows: 
    \begin{equation}
    \label{(p,q)-Eulerian-poly}
A_n(\mathbf{u},f,g,h,t\,|\,\alpha, \beta\,;\, p,q):=\sum_{\sigma \in \mathfrak{S}_{n+1}}\wt(\sigma)\,h^{\dis(\sigma)}p^{\ldes(\sigma)}q^{\rasc(\sigma)}.
    \end{equation}
    
We define the generalized $(p,q)$-analogue of a positive integer $k\in \mathbb{N}^*$, 
with auxiliary parameters $\alpha$ and $\beta$, as follows:
\begin{subequations}
\begin{equation}
[\,k;\alpha,\beta\,]_{p,q}
:= (\alpha-1) p^{k-1} + (\beta-1) q^{k-1}
+ \frac{p^k - q^k}{p - q}.
\end{equation}
Note that
\[
[\,1;\alpha,\beta\,]_{p,q} = \alpha + \beta - 1,
\]
and
\begin{equation}
[\,k\,]_{p,q} := [\,k;1,1\,]_{p,q} = \frac{p^k - q^k}{p - q}.
\end{equation}
\end{subequations}

 Let $\mathbf{b}:=(b_k)_{k\ge0}$ and $\boldsymbol{\lambda}:=(\lambda_k)_{k\ge1}$ be elements of a commutative ring. We define the Jacobi-type continued fraction $\mathcal{J}(t;\mathbf{b}, \boldsymbol{\lambda})$ as a formal power series by 
\begin{equation}\label{J-frac}    
\mathcal{J}(t;\mathbf{b}, \boldsymbol{\lambda})
:= \cfrac{1}{1-b_0 t \;-\;
   \cfrac{\lambda_1 t^2}{1-b_1 t \;-\;
   \cfrac{\lambda_2 t^2}{1-b_2 t \;-\; \ddots}}}.
\end{equation}

   \begin{subequations}\label{equ;con-generalized-eulerian} 
\begin{theorem}\label{thm1}
We have 
    \begin{equation}
  \sum_{n=0}^\infty A_n(\textbf{u},f,g,h,t\,|\,\alpha, \beta\,;\,p,q)z^n
 = \mathcal{J}(z;\mathbf{b}, \boldsymbol{\lambda}),
\end{equation}
where
\begin{align}
b_k&=h^k\biggl(u_3\bigl[\,k+1;1, \,\alpha t\,\bigr]_{p,q}+u_4\bigl[\,k+1;\beta t,\,1 \,\bigr]_{p,q}\biggr),\\
\lambda_{k}&=u_1u_2h^{2k-1}\,[\,k\,]_{p,q}\,\bigl[\,k+1;\beta g,\,\alpha f\bigr]_{p,q}.
\end{align}
\end{theorem}
\end{subequations}
\begin{remark}
The special case $f=g=h=t=p=q=1$ of \eqref{equ;con-generalized-eulerian} is Eq. \eqref{equ:continued-fraction-XZ}.
    When $\alpha=\beta=h=1$, other generalizations of the above continued fraction were considered in \cite{SiS96,BS21,SoZ22}. 
\end{remark}

Let $\sigma = \sigma_1 \sigma_2 \cdots \sigma_n \in \S_n$.
An \emph{inversion} of $\sigma$ is a pair of indices $(i, j)$ such that 
$1 \le i < j \le n$ and $\sigma_i > \sigma_j$. 
The \emph{inversion number} of $\sigma$, denoted by $\operatorname{inv}(\sigma)$, 
is the total number of inversions in $\sigma$. 
Petersen and Tenner~\cite{PT15} introduced a statistic called \emph{depth} 
for Coxeter groups, defined in terms of factorizations of group elements into reflections. 
In the case of the symmetric group $\mathfrak{S}_n$, 
the reflections are the transpositions $(i\,j)$ with $1 \le i < j \le n$. 
The \emph{depth} of a permutation $\sigma \in \mathfrak{S}_n$, 
denoted by $\operatorname{depth}(\sigma)$, is defined as
\begin{equation}
  \operatorname{depth}(\sigma)
  = \min \left\{
      \sum_{r=1}^{k} (j_r - i_r) :
      \sigma = (i_1\,j_1)(i_2\,j_2)\cdots(i_k\,j_k)
    \right\}.
\end{equation}
They obtained an alternative  formula
\begin{equation}\label{def:depth}
  \operatorname{depth}(\sigma)
  = \sum_{\sigma(i) > i} \bigl(\sigma(i) - i\bigr).
\end{equation}
Thus, this statistic coincides with the statistic \emph{Edif} over permutations in \cite{CSZ97} and 
corresponds to the area under the Motzkin path via the bijection of Foata-Zeilberger, 
see Lemma~\ref{lem:dep}.

Corteel \cite{Cor07} defined
the \emph{crossing} and \emph{nesting} numbers of $\sigma\in\S_n$  as follows:
\begin{subequations}\label{def:cros-nest}
   \begin{align}
    \mathrm{cros}(\sigma)
= \#\bigl\{(i,j)\in [n]\times[n] :
(i<j \le \sigma_i < \sigma_j) \lor (i>j > \sigma_i > \sigma_j)
\bigr\};\\
\mathrm{nest}(\sigma)
= \#\bigl\{(i,j)\in [n]\times[n] :
(i<j \le \sigma_j < \sigma_i) \lor (i>j > \sigma_j > \sigma_i)
\bigr\}.
\end{align} 
\end{subequations}
These statistics are better illustrated by drawing an arc-diagram as follows: put the numbers from $1$ to $n$ on a line and draw an edge from $i$ to $\sigma_i$ above the line if $i$ is a weak excedance, and below the line otherwise. 
 
\begin{figure}[t]
\centering
\vspace*{4cm}
\begin{picture}(40,0)(100, -60)
\setlength{\unitlength}{2mm}
\linethickness{.5mm}
\put(-2,0){\line(1,0){54}}
\put(0,0){\circle*{1,3}}\put(0,0){\makebox(0,-6)[c]{\small 1}}
\put(5,0){\circle*{1,3}}\put(5,0){\makebox(0,-6)[c]{\small 2}}
\put(10,0){\circle*{1,3}}\put(10,0){\makebox(0,-6)[c]{\small 3}}
\put(15,0){\circle*{1,3}}\put(15,0){\makebox(0,-6)[c]{\small 4}}
\put(20,0){\circle*{1,3}}\put(20,0){\makebox(0,-6)[c]{\small 5}}
\put(25,0){\circle*{1,3}}\put(25,0){\makebox(0,-6)[c]{\small 6}}
\put(30,0){\circle*{1,3}}\put(30,0){\makebox(0,-6)[c]{\small 7}}
\put(35,0){ \circle*{1,3}}\put(35,0){\makebox(0,-6)[c]{\small 8}}
\put(40,0){\circle*{1,3}}\put(40,0){\makebox(0,-6)[c]{\small 9}}
\put(45,0){\circle*{1,3}}\put(45,0){\makebox(0,-6)[c]{\small 10}}
\put(50,0){\circle*{1,3}}\put(50,0){\makebox(0,-6)[c]{\small 11}}

\green{\qbezier(0,0)(15,16)(35.5,0)
\qbezier(35.5,0)(20,-15)(15,0)
\qbezier(15,0)(31,17)(45,0)
 \qbezier(45,0)(47.5,7)(50,0)
 \qbezier(50,0)(29,-22)(5,0)
\qbezier(5,0)(7.5,5)(10,0)
\qbezier(10,0)(5.5,-5)(0,0)
}
\red{\qbezier(19.3,0)(25.5,7)(29.5,0)
\qbezier(29.5,0)(27.3,-7)(24.4,0)
\qbezier(24.4,0)(22.3,-7)(19.3,0)
}
\blue{\qbezier(38.5,0)(40,2)(36,4)
\qbezier(38.5,0)(34,2)(36,4)}

\end{picture}
\caption{
   The  arc-diagram of the permutation $\sigma = 8\, 3\, 1\, 10\, 7\, 5\, 6\, 4\,9\, 11\,2$, or
           as product of  disjoint cycles $\sigma =(1,8,4,10,11,2,3)\,(5,7,6)\,(9)$.
   \label{fig.pictorial.1}
}
\end{figure}
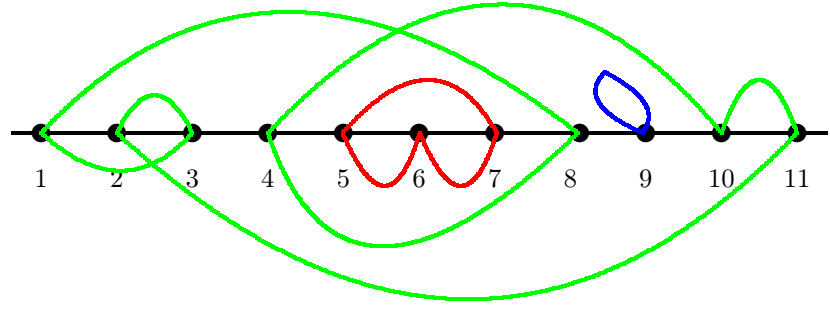
\begin{example}
    Let $\sigma = 8\, 3\, 1\, 10\, 7\, 5\, 6\, 4\,9\, 11\,2\in\S_{11}$. Then we have 
    $\nest(\sigma)=9$ and $\cros(\sigma)=3$; see Fig.~\ref{fig.pictorial.1}.
\end{example}
For
$\sigma=\sigma(1)\ldots \sigma(n)\in\S_n$,  we say that an index  $i\in [n]$ is a
\begin{itemize}
\item \emph{fixed point} (\textbf{fix}) of $\sigma$ if $i=\sigma(i)$;
\item  \emph{cycle peak} (\textbf{cpk}) of $\s$ if $\sigma^{-1}(i)<i>\sigma(i)$; 
\item  \emph{cycle valley} (\textbf{cval}) of $\s$ if $\sigma^{-1}(i)>i<\sigma(i)$; 
\item  \emph{cycle double ascent}  (\textbf{cda}) of $\s$ if $\sigma^{-1}(i)<i<\sigma(i)$; 
\item  \emph{cycle double descent}  (\textbf{cdd}) of $\s$ if $\sigma^{-1}(i)>i>\sigma(i)$.
\end{itemize}
Note  that $\cpk(\s)=\cval(\s)$. A letter $\sigma_i:=\sigma(i)$ is said to be a right-to-left minimum (\textbf{rmin}) of $\sigma$ if $\s_j>\s_i$ and $n\ge j>i\ge 1$. Define the general cycle version of Eulerian polynomials
\begin{subequations}
\begin{multline}\label{generalization-C}   C_n(\mathbf{u},t,s,h\,|\,\alpha, \beta\,;\,p,q):=\sum_{\sigma\in\S_n}(u_1u_2)^{\cval(\sigma)}u_3^{\cda(\sigma)}u_4^{\cdd(\sigma)}t^{\fix(\sigma)}s^{\inv(\sigma)}\\\times h^{\depth(\sigma)}\alpha^{\lmax(\sigma)}\beta^{\rlmin(\sigma)}p^{\cros(\sigma)}q^{\nest(\sigma)}.
\end{multline}

\begin{theorem}\label{thm2}
     We have 
    \begin{equation}\label{equ:continud-fractions-depth-alpha-beta}
        1+\sum_{n\ge 1}C_{n}(\mathbf{u},t,s,h\,|\,\alpha, \beta\,;\,p,q)z^n=\mathcal{J}(z;\mathbf{b},\boldsymbol{\lambda}),
    \end{equation}
    where $b_0=t\alpha\beta$, and 
    \begin{align}   
    b_k&=(sh)^k\biggl(tq^ks^k+u_4[\,k;\beta,1\,]_{p,qs}+u_3p[\,k;\alpha,1\,]_{p,qs}\biggr),\\
    \lambda_{k}&=u_1u_2(sh)^{2k-1}[\,k;\alpha,1\,]_{p,qs}[\,k;\beta,1\,]_{p,qs}.
    \end{align}
\end{theorem}
\end{subequations}

\begin{remark} 
Theorem~2.2 specializes to Theorem~6.5 of~\cite{MV94} when $\alpha=p=q=h=t=1$, after a suitable change of variables.
     When $u_1=u_4=s=h=\alpha=\beta=1$ and $u_2=u_3=t=x$,  Theorem~\ref{thm2} reduces to Eq. (1) in \cite{Cor07}.
\end{remark}

\begin{remark} Petersen~\cite{PetNote} conjectured the special case of the above
$J$-fraction expansion corresponding to the pair of statistics
$(\inv,\depth)$; this conjecture  was recently proved in~\cite{Eu24,US24}.
In fact, this $J$-fraction also follows from Theorem~10 or
Corollary~11 of~\cite{CSZ97} after the successive substitutions
$p\to q$ and then $q\to zq$.

\end{remark}

\begin{remark}
In Theorem~\ref{thm2}, applying  the following substitutions, respectively,
     \begin{itemize}
         \item $\mathbf{u}=(x,1,x,1)$ and $(\alpha,\beta,p,q)=(1,1,1,1)$,
         \item $\mathbf{u}=(x,1,x,1)$ and $(t,p,q,\beta)=(tx,1,1,1)$,
         \item $\mathbf{u}=(x,1,x,1)$ and   $(t,p,q,\alpha)=(1,1,1,1)$,
     \end{itemize}
      we recover Theorem~1.1, (i) and (ii) of Theorem~1.3 in \cite{Eu26}.
\end{remark}

\begin{definition}
    Let $\sigma = \sigma_1 \sigma_2 \cdots \sigma_n \in \S_n$. If  $i\in[n-1]$ is a descent, then
     $\sigma_i$ is called a descent-top, and $\sigma_{i+1}$ is called a 
    descent-bottom. For $i\in[n]$, the entry $\sigma_i$ is called  a nondescent or a nondescent-top (resp. nondescent-bottom) of $\sigma$,
if it is not a descent-top (resp. descent-bottom) of $\sigma$.
The number of nondescents of $\sigma$ is $\ndes(\sigma):=n-\des(\sigma)$.  
\end{definition}

 Let $\sigma\in \S_n$. 
 We denote the sum of the nondescent-tops (resp. nondescent-bottoms) of $\sigma$ by $\NDtop (\sigma)$ (resp. $\NDbot(\sigma)$) and define 
 the \emph{nondescent difference} of $\sigma$ (see \cite{SZ10}) by
\begin{equation}\label{def:nondif}
    \NDdif(\sigma):=\NDbot(\sigma)-\NDtop(\sigma).
\end{equation}
If we impose the boundary condition $\sigma_0=0$ and $\sigma_{n+1}=n+1$, then 
 each nondescent-top (resp. nondescent-bottom) of $\sigma\in\S_n$ is either a valley or a double ascent (resp. a peak or a double ascent), it follows  that
\begin{equation}\label{equ:dis-NDdif}
\NDdif(\sigma)=\sum_{i\in \Pk'(\sigma)} i \;-\;\sum_{j\in \Val'(\sigma)} j,
\end{equation}
where $\Pk'(\sigma)$ and $\Val'(\sigma)$ denote the sets of peaks and valleys of $\sigma$, respectively, with boundary condition $\sigma_0=0$ and $\sigma_{n+1}=n+1$.

\begin{example}
    Let $\sigma=42157368\in\S_8$.  We have 
    \begin{align*}
    \NDtop(\sigma)&=1+3+5+6+8=23,\\ \NDbot(\sigma)&=4+5+6+7+8=30.
    \end{align*}
    Then $\NDdif(\sigma)=7$. By \eqref{equ:dis-NDdif}, we have $\NDdif(\sigma)=(4+7)-(1+3)=7.$
\end{example}
\begin{definition}\label{Def:fmax-pmin}
 Let $\sigma = \sigma_1 \sigma_2 \cdots \sigma_n \in \S_n$. A letter $\sigma_i\in[n]$ is called a
 \begin{itemize}
     \item  foremaximum of~$\sigma$ if $\sigma_i$ is a nondescent top
and also a left-to-right maximum in~$\sigma$;
    \item postminimum of $\sigma$ if $\sigma_i$ is a nondescent bottom  
and also a right-to-left minimum in~$\sigma$.
 \end{itemize}
  The number of foremaxima (resp. postminima) of~$\sigma$ is denoted by $\fmax(\sigma)$ (resp. $\mathrm{pmin}(\sigma)$).
\end{definition}
\begin{example}
    For $\sigma=42157368\in\S_8$, we have 
     $\fmax(\sigma)=|\{5,8\}|=2$ and $\mathrm{pmin}(\sigma)=|\{6,8\}|=2$.
\end{example}

Recall \cite{SZ10} that for $\sigma \in \mathfrak{S}_n$ the statistics $\MAD$ and $\mathrm{MADL}$ are defined by
\begin{subequations}
    \begin{align}
  \MAD(\sigma)
  := \des(\sigma) + (31\!-\!2)(\sigma)  + 2(2\!-\!31)(\sigma); \label{def:Mad}\\
  \mathrm{MADL}(\sigma):=\des(\sigma)+(2\!-\!31)(\sigma)+2(31\!-\!2)(\sigma). \label{def:Madl}
\end{align}
\end{subequations}

For $\sigma\in\S_n$, an index $i\in[n]$ is called  
\begin{itemize}
\item   an \emph{excedance} (\textbf{exc}) of $\sigma$ if $i<\sigma(i)$,
\item a \emph{weak excedance} (\textbf{wex}) of $\sigma$ if $i\le\sigma(i)$,
\item a \emph{drop} ($\mathbf{drop}$) if $\sigma(i)<i$. 
\end{itemize}

Clarke, Steingrímsson and Zeng~\cite{CSZ97} constructed a bijection $\Phi_{CSZ}$ on $\S_n$ to characterize the composition $\Phi_{FZ}^{-1}\circ \Phi_{FV}$.
Later, Corteel~\cite{Cor07} and Shin--Zeng~\cite{SZ10} introduced two variants of the bijection $\Phi_{CSZ}$ on $\S_n$, denoted by $\Phi_C$ and $\Phi_{SZ}$, respectively.
These bijections are known to send
the generalized pattern statistic
\[
(31\!-\!2,\; 2\!-\!31)
\]
to the pairs $(\cros,\nest)$ and $(\nest,\cros)$ on $\S_n$.
Beyond this correspondence, they admit a natural interpretation in terms of the
weighted path model underlying our continued-fraction expansions.
In particular, the following result shows that $\Phi_C$ and $\Phi_{SZ}$ preserve
additional structural features compatible with the Motzkin-path encoding and
the associated Jacobi-type continued fractions.
Related aspects of these constructions have recently been investigated in
\cite{HMZ20,CF23,CFZ25}.

\begin{theorem}\label{thm3}
   The mappings $\Phi_C$ and $\Phi_{SZ}$ are  bijections on $\S_n$ such that for all $\sigma\in\S_n$,  we have 
    \begin{align}
   &(\ndes, \mathrm{pmin}, 31\!-\!2, 2\!-\!31, \mathrm{MADL},\NDdif)\,\sigma= 
(\wex, \fix, \nest, \cros,\inv,\depth)\,\Phi_{C}(\sigma);\label{equ:dis-Corteel}\\[6pt]
 &(\ndes, \fmax, 31\!-\!2, 2\!-\!31, \MAD,\NDdif)\,\sigma= 
(\wex, \fix, \cros, \nest, \inv,\depth)\,\Phi_{SZ}(\sigma).\label{equ:dis}
    \end{align}
Moreover, if $\Phi^*:=\Phi_{SZ}\circ\Phi_{C}^{-1}$, then
\begin{equation}\label{equ:p_i+q_i-2'}(\wex,\nest,\cros,\depth)\,\sigma=(\wex,\cros,\nest,\depth)\,
\Phi^*(\sigma).
    \end{equation}
    \end{theorem}
    
Any polynomial with real coefficients
\(h(x) = \sum_{i=0}^{n} h_i x^i
\)
satisfying the symmetry condition $h_i = h_{n-i}$ can be written uniquely in the form
\[
h(x) = \sum_{k=0}^{\lfloor n/2 \rfloor} \gamma_k\, x^k (1+x)^{n-2k}.
\]
The coefficients $\gamma_k$ are called the \emph{$\gamma$-coefficients} of~$h(x)$.  
If all $\gamma_k$ are nonnegative, then $h(x)$ is said to be \emph{$\gamma$-positive}. We refer the reader to 
the general overview of $\gamma$-positivity phenomena in \cite{At18}.
We now consider two specializations of the polynomials
$A_n(\mathbf{u},f,g,h,t\,|\,\alpha,\beta\,;\,p,q)$:
\begin{subequations}\label{An-Bn}
    \begin{align} 
A_n(x,f,g,h\,|\,p,q):&=A_n((1,x,1,x),f,g,h,1\,|\,1,1\,;\,p,q);\\
B_n(x,f,g,h\,|\,t,\alpha):&=A_n((1,x,1,x),f,g,h,t\,|\,\alpha,\alpha\,;\, 1,1).
\end{align}
By~\eqref{(p,q)-Eulerian-poly}, we obtain
\begin{align}
A_n(x,f,g,h\,|\,p,q)&=\sum_{\sigma \in \mathfrak{S}_{n+1}}\mathbf{w}(\sigma)\;p^{\ldes(\sigma)}q^{\rasc(\sigma)};\\
B_n(x,f,g,h\,|\,t,\alpha)
&=\sum_{\sigma \in \mathfrak{S}_{n+1}}\mathbf{w}(\sigma)\;
 t^{\lrmaxda(\s)+\rlmaxdd(\s)}\alpha^{{\lmax}(\sigma)+{\rmax}(\sigma)-2},
\end{align} 
where the weight $\mathbf{w}(\sigma)$ is defined by 
\begin{equation}
\mathbf{w}(\sigma)=x^{\des(\sigma)}f^{\mathrm{lmaxpk(\sigma)-1}}g^{\bmax(\s)-1}h^{\dis(\sigma)}. 
\end{equation}
\end{subequations}

By the $J$-fraction expansion of Theorem~\ref{thm1}, we show that the polynomials
$A_n(x,f,g,h\,|\,p,q)$ and $B_n(x,f,g,h\,|\,t,\alpha)$ admit natural
$\gamma$-positive decompositions.
The next result provides an explicit combinatorial interpretation of the
$\gamma$-coefficients in terms of permutations without double descents,
refined by generalized pattern statistics.

\begin{theorem}\label{thm4}
For $n\ge 0$, the polynomial $A_n(x,f,g,h\,|\,p,q)$ is $\gamma$-positive and
can be expanded as
\begin{equation}\label{gamma-des-special2-intro}
A_n(x,f,g,h\,|\,p,q)
=\sum_{k=0}^{\lfloor n/2\rfloor}
\gamma_{n,k}^a(f)\,x^{k}(1+x)^{\,n-2k},
\end{equation}
where
\begin{align}\label{gamma2-intro}
\gamma_{n,k}^a(f)
=\sum_{\sigma\in
\S_{n+1,\des=k}^{\mathrm{dd}=0}}
f^{\lmaxpk(\sigma)-1}
g^{\bmax(\sigma)-1}
h^{\dis(\sigma)}
p^{\ldes(\sigma)}
q^{\rasc(\sigma)},
\end{align}
and
\[
\S_{n,\des=k}^{\mathrm{dd}=0}
=\{\sigma\in\S_n:\dd(\sigma)=0\ \text{and}\ \des(\sigma)=k\}.
\]
\end{theorem}

\begin{remark}
When $f=g=h=1$, this expansion reduces to a classical $\gamma$-positive
decomposition of Eulerian-type polynomials. An equivalent formulation of
\eqref{gamma-des-special2-intro} was obtained by Brändén~\cite[Eq.~(5.1)]{Bra08}
via the modified Foata--Strehl action. The present approach derives the
$\gamma$-coefficients directly from the continued-fraction structure.
\end{remark}
\begin{theorem}\label{thm5}
For $n\ge 0$, we have
    \begin{equation}\label{gamma-des-special1}
     B_n(x,f,g,h\,|\,t,\alpha)=\sum_{k=0}^{\lfloor n/2\rfloor}\gamma_{n,k}^b(f)\,x^{k}(1+x)^{n-2k},
      \end{equation}
      where 
    \begin{equation}\label{gamma1-intro}   \gamma_{n,k}^b(f)=\sum_{\sigma\in\S_{n+1,\des=k}^{\mathrm{dd=0}}}f^{\mathrm{lmaxpk(\sigma)-1}}g^{\bmax(\s)-1}h^{\dis(\sigma)}
 t^{\lrmaxda(\s)}\alpha^{{\lmax}(\sigma)+{\rmax}(\sigma)-2},
\end{equation} 
with $\S_{n,\des=k}^{\mathrm{dd=0}}=\{\sigma\in\S_n: \dd(\sigma)=0\text{ and } \des(\sigma)=k\}$.
\end{theorem}
\begin{remark}
    When $f=g=h=1$, we recover (3.28c) of Proposition~3.10 in \cite{XZ24}. 
\end{remark}
Postnikov, Reiner, and Williams~\cite{PRW08} proved that the $h$-polynomials of dual stellohedra
has the combinatorial interpretation
\[
\widetilde{A}_n(x)=\sum_{\sigma\in\mathrm{PRW}_{n+1}}x^{\des(\sigma)},
\]
where $\mathrm{PRW}_n$ is the set of permutations in $\S_n$ such that 
the first descent is at the letter  $n$.
  For example, 
\[
\PRW_2=\{12,21\}\quad \text{and}\quad \mathrm{PRW}_3=\{123,132,231,312,321\}.
\]
Clearly, a permutation $\sigma$ of $[n]$ is in $\PRW_{n}$ if and only if 
$n$ is simultaneously the unique left-to-right maximum and a peak of $\sigma$, namely,
$\pmax(\sigma) = 1$. 
By \eqref{An-Bn}, we have  
\begin{subequations}
    \begin{align} 
A_n(x,0,g,h\,|\,p,q)&=\sum_{\sigma\in\PRW_{n+1}}x^{\des(\sigma)}g^{\bmax(\s)-1}h^{\dis(\sigma)}p^{\ldes(\sigma)}q^{\rasc(\sigma)},\\
B_n(x,0,g,h\,|\,\,t,\alpha)&=\sum_{\sigma\in\PRW_{n+1}}x^{\des(\sigma)}g^{\bmax(\s)-1}h^{\dis(\sigma)}t^{\lrmaxda(\s)+\rlmaxdd(\s)}\alpha^{{\lmax}(\sigma)+{\rmax}(\sigma)-2}.
\end{align}

The following result is an immediate 
consequence of  Theorems~\ref{thm4} and~\ref{thm5}.
\begin{theorem}
    We have 
    \begin{align}\label{gamma-des-special2-PRW}
A_n(x,0,g,h\,|\,p,q)&=\sum_{k=0}^{\lfloor n/2\rfloor}\gamma_{n,k}^a(0)\,x^{k}(1+x)^{n-2k}, \\
  \label{gamma-des-special-PRW}
B_n(x,0,g,h\,|\,\,t,\alpha)&=\sum_{k=0}^{\lfloor n/2\rfloor}\gamma_{n,k}^b(0)\,x^{k}(1+x)^{n-2k},    
  \end{align}         
    where 
    \begin{align}
    \label{gamma2-PRW}
\gamma_{n,k}^a(0)=\sum_{\sigma\in\mathrm{PRW}_{n+1,k}^{\dd=0}}g^{\bmax(\s)-1}h^{\dis(\sigma)}p^{\ldes(\sigma)}q^{\rasc(\sigma)},\\  
\gamma_{n,k}^b(0)=\sum_{\sigma\in\PRW_{n+1,k}^{\dd=0}}g^{\bmax(\s)-1}h^{\dis(\sigma)}
 t^{\lrmaxda(\s)}\alpha^{{\lmax}(\sigma)+{\rmax}(\sigma)-2} \label{gamma1-PRW} 
\end{align}
with 
\begin{equation}
\PRW_{n,k}^{\dd=0}:=\{\sigma\in\PRW_n: \dd(\sigma)=0\text{ and } \des(\sigma)=k\}.
\end{equation}
\end{theorem}
\end{subequations}

\begin{remark} 
Special cases of \eqref{gamma-des-special-PRW} were obtained in 
\cite[Proposition 3.8]{XZ24} and \cite[Theorem~2.1]{JL23}.
\end{remark}

Consider the two special cases of the polynomials $C_n(\mathbf{u},t,s,h\,|\,\alpha,\beta\,;\,p,q)$, 
see~\eqref{generalization-C},
\begin{subequations}
    \begin{align}
C_n^1(x,t,s,h\,|\,\alpha,p,q):&=C_n((x,\,1,\,x/p,\,1),\,t(1+x),s,h\,|\,\alpha,\alpha\,;\, p,q)\nonumber\\
&=\sum_{\sigma\in\S_n}\mathbf{w^c}(\sigma)\,p^{\cros(\sigma)-\cda(\sigma)}q^{\nest(\sigma)},
\end{align}
and 
\begin{align}
  C_n^2(x,t,s,h\,|\,\alpha,p,q):&=C_n((x,\,p,\,x,\,p),\,t(1+x),s,h\,|\,\alpha,\alpha\,;\, p,q)\nonumber\\
&=\sum_{\sigma\in\S_n}\mathbf{w^c}(\sigma)\,p^{\cros(\sigma)+\drop(\sigma)}q^{\nest(\sigma)},  
\end{align}
where 
\begin{equation}
    \mathbf{w^c}(\sigma)=x^{\exc(\sigma)}(t(1+x))^{\fix(\sigma)}s^{\inv(\sigma)} h^{\depth(\sigma)}\alpha^{\rlmin(\sigma)+\lmax(\sigma)}.
\end{equation}
\end{subequations}

 By the $J$-fraction of Theorem~\ref{thm2}, we show that the polynomials $C_n^1(x,t,s,h\,|\,\alpha,p,q)$ and $C_n^2(x,t,s,h\,|\,\alpha,p,q)$ admit the following  $\gamma$-expansions.

\begin{theorem}\label{thm6}
For $n\ge 0$, we have
    \begin{equation}\label{equ:C_n-gamma}
C_n^i(x,t,s,h\,|\,\alpha, p,q)=\sum_{k=0}^{\lfloor n/2\rfloor}\gamma_{n,k}^i(t)\,x^k(1+x)^{n-2k}, \quad (i=1,2)
    \end{equation}   
    where 
    \begin{equation}\label{gamma1-cyclic-intro}   \gamma_{n,k}^1(t)=\sum_{\sigma\in\S_{n,\exc=k}^{\mathrm{cda=0}}}t^{\fix(\sigma)}s^{\inv(\sigma)}h^{\depth(\sigma)}\alpha^{\lmax(\sigma)+\rlmin(\sigma)}p^{\cros(\sigma)}q^{\nest(\sigma)}
\end{equation} 
\begin{equation}\label{gamma1-cyclic-2-intro}   \gamma_{n,k}^2(t)=\sum_{\sigma\in\S_{n,\exc=k}^{\mathrm{cda=0}}}t^{\fix(\sigma)}s^{\inv(\sigma)}h^{\depth(\sigma)}\alpha^{\lmax(\sigma)+\rlmin(\sigma)}p^{\cros(\sigma)+\drop(\sigma)}q^{\nest(\sigma)}
\end{equation} 
with 
\begin{equation}
\S_{n,\exc=k}^{\mathrm{cda=0}}:=\{\sigma\in\S_n: \cda(\sigma)=0 \text{ and } \exc(\sigma)=k\}.
\end{equation}
\end{theorem}
 \begin{remark} 
The special case  $\alpha=s=h=1$ of 
     Eq.~\eqref{equ:C_n-gamma} with $i=2$  is equivalent to Eq.~(4.14) of Theorem~4.8 in~\cite{LWZ19}.
 \end{remark}
The next corollary corresponds to the derangement specialization of the
$\gamma$-positive decomposition arising from the $J$-fraction in
Theorem~\ref{thm1}. By setting the fixed-point parameter to zero,
the continued-fraction structure naturally restricts to weighted
derangements, yielding an explicit combinatorial interpretation of the
$\gamma$-coefficients. Define 
\begin{align}
C_n^1(x,0,s,h\,|\,\alpha,p,q)
&=\sum_{\sigma\in\mathcal{D}_n}
x^{\exc(\sigma)}
s^{\inv(\sigma)}
h^{\depth(\sigma)}
\alpha^{\lmax(\sigma)+\rlmin(\sigma)}
p^{\cros(\sigma)-\cda(\sigma)}
q^{\nest(\sigma)},\\
C_n^2(x,0,s,h\,|\,\alpha,p,q)
&=\sum_{\sigma\in\mathcal{D}_n}
x^{\exc(\sigma)}
s^{\inv(\sigma)}
h^{\depth(\sigma)}
\alpha^{\lmax(\sigma)+\rlmin(\sigma)}
p^{\cros(\sigma)+\drop(\sigma)}
q^{\nest(\sigma)},
\end{align}
where  $\mathcal{D}_n:=\{\sigma\in \S_n: \fix(\sigma)=0\}$ is the set of derangements of $[n]$.
\begin{corollary}
For $n\ge0$, we have
\begin{equation}\label{cor:derangement-gamma}
C_n^i(x,0,s,h\,|\,\alpha,p,q)
=\sum_{k=0}^{\lfloor n/2\rfloor}
\gamma_{n,k}^i(0)\,x^{k}(1+x)^{\,n-2k}, \quad (i=1,2)
\end{equation}
where
\begin{align}
\gamma_{n,k}^1(0)
&=\sum_{\sigma\in
\mathcal{D}_{n,k}^{0}}
s^{\inv(\sigma)}
h^{\depth(\sigma)}
\alpha^{\rlmin(\sigma)+\lmax(\sigma)}
p^{\cros(\sigma)}
q^{\nest(\sigma)},\\
\gamma_{n,k}^2(0)
&=\sum_{\sigma\in
\mathcal{D}_{n,k}^{0}}
s^{\inv(\sigma)}
h^{\depth(\sigma)}
\alpha^{\rlmin(\sigma)+\lmax(\sigma)}
p^{\cros(\sigma)+\drop(\sigma)}
q^{\nest(\sigma)},
\end{align}
with
$\mathcal{D}_{n,k}^{0}
:=\{\sigma\in\mathcal{D}_n:\cda(\sigma)=0
\text{ and }\exc(\sigma)=k\}.$
\end{corollary}

\begin{remark}
When  $h=\alpha=q=1$, Eq.~\eqref{cor:derangement-gamma} reduces to
Theorem~2 of~\cite{SZ16}.
\end{remark}
The remainder of this paper is organized as follows.
In Section~\ref{sec2}, we begin with a detailed review of the
Françon--Viennot bijection and the Foata--Zeilberger bijection between
permutations and Laguerre histories, a class of weighted Motzkin paths.
We then employ these bijections to prove Theorems~\ref{thm1} and
\ref{thm2} using Flajolet's fundamental theorem on $J$-continued
fractions. 

In Section~\ref{sec:SZ-C}, we recall the bijections of Shin--Zeng and
Corteel on permutations and use them to establish
Theorem~\ref{thm3}. 
Section~\ref{sec:gamma} is devoted to the $\gamma$-positive expansions
derived from the continued fractions of Theorems~\ref{thm1} and
\ref{thm2}. In particular, we obtain $\gamma$-expansions for the
$(p,q)$-analogues of the generalized Eulerian polynomials
(Theorems~\ref{thm4}, \ref{thm5}, and \ref{thm6}). 
We also present two distinct factorizations of the corresponding
$\gamma$-coefficients together with combinatorial interpretations in
terms of André permutations of types~1 and~2; see
Theorems~\ref{Thm: generalized PZ}, \ref{thm-andre-minimum}, and
\ref{thm-andre-maximum}.

In Section~\ref{sec:proof-andre}, we apply the Pan--Zeng group action to
prove Theorem~\ref{Thm: generalized PZ} and the Dong--Lin--Pan group
action to prove Theorem~\ref{thm-andre-minimum}. 
Section~\ref{sec:total-positivity} investigates the total positivity of
the $(p,q)$-analogues of generalized Eulerian polynomials via continued
fractions. 
Finally, in Section~\ref{sec:special-cases}, we answer a question of
Eu~et~al.\ by constructing a bijection that makes explicit the symmetry
between two variants of the $\nest$ and $\cros$ statistics. We also
derive several compact closed-form expressions obtained by specializing
one of the parameters in the associated continued fractions to $-1$.

\section{Encoding permutations by Laguerre histories}\label{sec2}

\subsection{Preliminaries}

A \emph{Motzkin path} of length $n$ is a lattice path
$\{(x_i,y_i)\}_{0\le i\le n}$ in the plane that
\begin{itemize}
  \item starts at $(x_0,y_0)=(0,0)$ and ends at $(x_n,y_n)=(n,0)$;
  \item uses only the three types of steps
  \[
     \mathrm{N}=(1,1)\ \text{(North-East)},\qquad
     \mathrm{S}=(1,-1)\ \text{(South-East)},\qquad
     \mathrm{E}=(1,0)\ \text{(East)};
  \]
  \item never goes below the $x$-axis, that is,
  the height $h_i:=y_{i-1}\ge0$ for all $i\in[n]$.
\end{itemize}
Let $\mathcal{M}_n$ denote the set of Motzkin paths of length $n$.

\medskip

Following the classical Flajolet--Viennot approach, we consider a
\emph{weighted Motzkin-path model}.  
Each East step at height $i$ is assigned a weight $a_i$, while each
North-East (resp.\ South-East) step at height $i$ receives a weight
$b_i$ (resp.\ $c_i$).  
The weight of a path $P\in\mathcal{M}_n$ is the product of the weights
of its steps and is denoted by $\omega(P)$.  
The generating function of weighted Motzkin paths admits a Jacobi-type
continued-fraction expansion, given by Flajolet's fundamental lemma
on continued fractions~\cite{Fla80}.

\begin{lemma}\label{lem:fla}
We have
\begin{equation}
\sum_{n\ge0}\;\sum_{P\in\mathcal{M}_n}\omega(P)\,z^n
=\cfrac{1}{1-b_0 z-\cfrac{a_0 c_1 z^2}{1-b_1 z-\cfrac{a_1 c_2 z^2}{\ddots}}}.
\end{equation}
\end{lemma}

For a permutation $\sigma=\sigma_1\cdots \sigma_n\in\mathfrak{S}_n$ and $k=\sigma_j$ with $j\in [n]$,
the coordinate  generalized  pattern statistics  are defined by
\begin{align}\label{patterns:refine}
(31\!-\!2)_k (\sigma)
&= \#\{\, i : i+1 < j \text{ and } 
\sigma_{i+1} < k < \sigma_i \,\};\nonumber\\
(2\!-\!13)_k (\sigma)
&= \#\{\, i : j < i-1 \text{ and } 
\sigma_{i-1} < k < \sigma_i \,\};\\
(2\!-\!31)_k (\sigma)
&= \#\{\, i : j < i-1 \text{ and } 
\sigma_i < k < \sigma_{i-1} \,\}.\nonumber
\end{align}
Clearly, for $\sigma \in \mathfrak{S}_n$, 
 the statistic associated with the generalized pattern $31\!-\!2$ decomposes as
\[
(31\!-\!2)(\sigma)
=\sum_{k=1}^{n} (31\!-\!2)_k(\sigma).
\]
Similarly,
\[
(2\!-\!13)(\sigma)=\sum_{k=1}^{n} (2\!-\!13)_k(\sigma),
\text{ and }
(2\!-\!31)(\sigma)=\sum_{k=1}^{n} (2\!-\!31)_k(\sigma).
\]

\begin{example}
    Let $\sigma=4723516\in\S_7$. We have 
    \[(31\!-\!2)_3(\sigma)=1,\; (2\!-\!13)_4(\sigma)=2 \text{ and } (2\!-\!31)_4(\sigma)=2.\]
\end{example}

 We can identify  a Motzkin path of length  $n$ by  a Motzkin word $w=w_1\ldots w_n$ such that 
$w_i=\mathrm{N}$ (resp. $\mathrm{S}$, $\mathrm{E}$) if and only if the $i$th step is $\mathrm{N}$ (resp. $\mathrm{S}$, $\mathrm{E}$).  Given a Motzkin path $w=w_1w_2\cdots w_n\in\mathcal{M}_n$, the height $h_i$ of the $i$th step is given by 
\begin{equation}\label{def:height}
    h_i:=h_i(w)=\#\{j<i\,:\,w_j=\mathrm{N}\}-\#\{j<i\,:\,w_j=\mathrm{S}\}.
\end{equation}
\subsubsection{Françon-Viennot bijection}
 A \emph{2-Motzkin path} is a Motzkin path whose east steps could come in two colors: red (\red{$\mathrm{E_1}$}) or blue (\blue{$\mathrm{E_2}$}).
\begin{definition}[Laguerre history]
    A Laguerre history of length $n$ is a couple  $(w, p)$, where $w=w_1\cdots w_n$ is a  2-Motzkin path of length $n$ starting at $(0,0)$ and ending at $(n,0)$, and 
$p=(p_1, \ldots, p_n)$ is a label sequence such that  $0 \leq p_i \leq h_i$ for $1\le i\le n$. 
Denote by $\mathcal{L}_n$   the set of Laguerre histories  of length $n$. 
\end{definition}

We describe  a variant of Françon-Viennot's bijection $\Psi_{FV}:\S_{n+1}\to \L_n$ and its inverse, see~\cite{SZ12,Vie83}.  For a permutation $\sigma\in\S_{n+1}$,  the Laguerre history  $\Psi_{FV}(\sigma) := (w_1\ldots w_n,(p_1,\ldots, p_n)) \in \mathcal{L}_{n}$ is defined by
\begin{equation}\label{def:corresponding}
    w_i =
\begin{cases}
\mathrm{N}\;(\textrm{resp.} \;\mathrm{S} ) & \text{if } i \textrm{ is a valley (resp. peak) of }\sigma; \\[2pt]
\mathrm{\red{E_1}} \; (\textrm{resp.}\; \blue{\mathrm{E_2}}) & \text{if } i \text{ is a double descent (resp. double ascent) of }\sigma, 
\end{cases}
\end{equation}
and $p_i = (31\!-\!2)_i(\sigma)$ for each $i \in [n]$.

\begin{figure} 
  \centering

\begin{tikzpicture}[scale=1, line cap=round, line join=round]


\draw[very thick,->] (0,-3) -- (0,3); 
\draw[thick] (1,1.0) -- (1,-3);
\draw[thick] (2,1.0) -- (2,-3);
\draw[thick] (3,2.0) -- (3,-3);
\draw[thick] (4,1.0) -- (4,-3);
\draw[thick] (5,0.0) -- (5,-3);
\draw[thick] (6,0.0) -- (6,-3);
\draw[thick] (7,1.0) -- (7,-3);
\draw[thick] (8,0.0) -- (8,-3);

\fill[color=black] (0,0) circle (1.5pt);
\fill[color=black] (1,1) circle (1.5pt);
\fill[color=black] (2,1) circle (1.5pt);
\fill[color=black] (3,2) circle (1.5pt);
\fill[color=black] (4,1) circle (1.5pt);
\fill[color=black] (5,0) circle (1.5pt);
\fill[color=black] (6,0) circle (1.5pt);
\fill[color=black] (7,1) circle (1.5pt);
\fill[color=black] (8,0) circle (1.5pt);

\draw[very thick,->] (-0.7,0) -- (8.8,0);

\draw[very thick]
  (0,0) -- (1,1) -- (2,1) -- (3,2) -- (4,1) -- (5,0)
  -- (6,0) -- (7,1) -- (8,0);
\node at (1.5,1.3) {\red{$\mathrm{E_1}$}};
\node at (5.5,0.3) {\blue{$\mathrm{E_2}$}};
\draw[very thick,red]  (1,1) -- (2,1);
\draw[very thick,blue] (5,0) -- (6,0);

\foreach \y in {-1,-2,-3} {
  \draw[thick] (-0.7,\y) -- (8,\y);
}

\node at (-0.5,-0.5) {$i$};
\node at (-0.5,-1.5) {$h_i$};
\node at (-0.5,-2.5) {$p_i$};

\foreach \k in {1,...,8} {
  \node at (\k-0.5,-0.5) {\k};
}
\foreach \k/\v in {1/0,2/1,3/1,4/2,5/1,6/0,7/0,8/1} {
  \node at (\k-0.5,-1.5) {\v};
}
\foreach \k/\v in {1/0,2/1,3/0,4/2,5/1,6/0,7/0,8/1} {
  \node at (\k-0.5,-2.5) {\v};
}

\end{tikzpicture}
\caption{The Laguerre history $\Psi_{FV}(\sigma)$, where $\sigma=697835142$.}
  \label{fig:motzkin}
\end{figure}
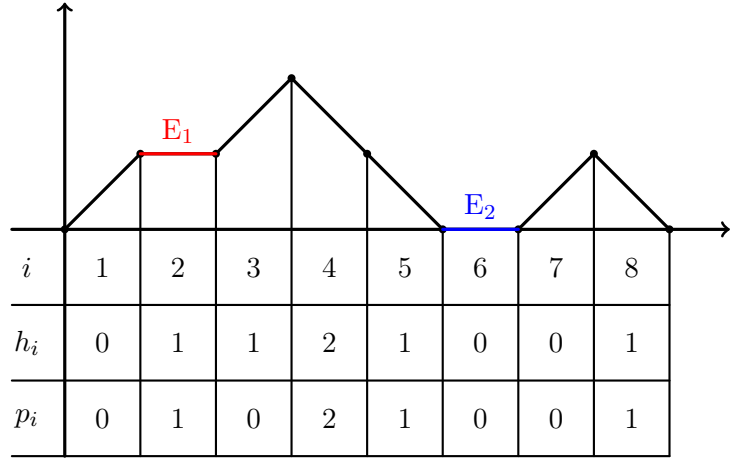

Inversely, starting from a Laguerre history $(w,p)\in\L_{n}$ one can recover the corresponding permutation $\sigma$
by the following algorithm:
\begin{itemize}
  \item Initialization: $\sigma=\diamond$;
  \item At the $i$th $(1\le i\le n)$ step of the algorithm, replace the $(p_i+1)$th $\diamond$ (from left to right) of $\sigma$ by
  \[
  \left\{
  \begin{array}{@{}l@{\quad}l}
    \diamond\, i\, \diamond & \text{if } w_i = \mathrm{N};\\[2pt]
    i\,\diamond              & \text{if } w_i = \mathrm{\blue{E_2}};\\[2pt]
    i                        & \text{if } w_i = \mathrm{S};\\[2pt]
    \diamond\, i             & \text{if } w_i = \mathrm{\red{E_1}}.\\
  \end{array}
  \right.
  \]
  \item The final permutation is obtained by putting $n+1$ at the last remaining $\diamond$.
\end{itemize}
As a consequence,  the cardinality of $\L_n$ is $(n+1)!$.
\begin{example}
Let $\sigma=697835142\in\S_{9}$.  We have $$\Psi_{FV}(\sigma)=(\mathrm{N\red{E_1}NSS\blue{E_2}NS},(0,1,0,2,1,0,0,1))\in\L_{8},$$ see Fig.~\ref{fig:motzkin}.

Inversely,    we have
\begin{align*}
   \sigma=\,&\diamond\to \diamond1\diamond  \;\to\; \diamond1\diamond2
  \;\to\; \diamond3\diamond1\diamond2 \;\to\; \diamond3\diamond 142
  \;\to\; \diamond35142
  \;\to\; 6\diamond35142\\
   &\to\; 6\diamond7\diamond35142
  \;\to\; 6\diamond7835142
  \;\to\; 697835142.
\end{align*}
\end{example}

There is another more visual description of  patterns $(31\!-\!2)_x\text{ and } (2\!-\!13)_x $ using the so-called $x$-decomposition;  see~\cite{FV79}.  For letter $x \in [n]$,  the \emph{$x$-decomposition} of $\sigma$ is the unique factorization  
\begin{equation}\label{x-decomposition}
   \sigma = u_1 v_1 \cdots u_k v_k u_{k+1}, \quad k \geq 1,
\end{equation}
in which the words $u_i$ and $v_j$, except possibly $u_1$ and $u_{k+1}$, are nonempty.   Moreover, each $u_i$ ($1 \leq i \leq k+1$) consists solely of letters $< x$,   while each $v_j$ ($1 \leq j \leq k$) consists solely of letters $\geq x$. Assume that $x$ lies inside $v_j$. Then every adjacent pair $(v_i,u_{i+1})$ with $i\le j-1$ contributes to a $(31\!-\!2)_x(\sigma)$, while every adjacent pair $(u_i,v_{i})$  with $i\ge j$ contributes to a $(2\!-\!13)_x(\sigma)$. In Fig~\ref{fig:x-decomposition}, we see that 
$$
(31\!-\!2)_x(\sigma)=j-1\;\text{ and }
(2\!-\!13)_x(\sigma)=k-j.
$$

\begin{example}
    If $\sigma = 6\,4\,1\,9\,5\,3\,8\,10\,2\,7\in\S_{10}$ and $x = 5$.
Then the $5$-decomposition is
\[
\sigma = u_1 v_1 u_2 v_2 u_3v_3u_4v_4,
\]
with
\[
(u_1,u_2,u_3,u_4)=(\varepsilon,\,41,\,3,\,2)\,\text{ and }\,
(v_1,v_2,v_3,v_4)=(6,\,9\,5,\,8\,10,\,7).
\]
Thus, we have $(31\!-\!2)_5(\sigma)=1$ and $(2\!-\!13)_5(\sigma)=2$.
\end{example} 

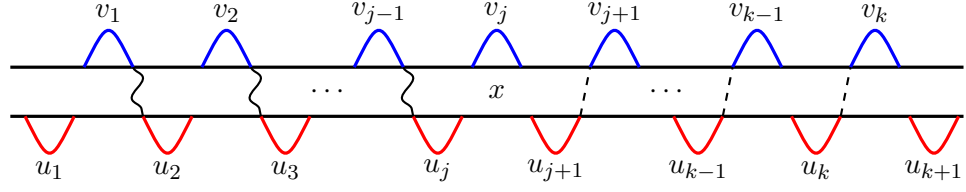
\begin{figure}
   \centering
\begin{tikzpicture}[scale=0.65]
\draw[very thick] (-3.4,1) -- (16,1);
\draw[very thick] (-3.4,0)   -- (16,0);
\draw[red, very thick](-3.1,0) .. controls (-2.6,-1) .. (-2.1,0);
\draw[blue, very thick](-1.9,1) .. controls (-1.4,2) .. (-0.9,1);
\draw[red, very thick](-0.7,0) .. controls (-0.2,-1) .. (0.3,0);
\draw[blue, very thick](0.5,1) .. controls (1.0,2) .. (1.5,1);
\draw[red, very thick](1.7,0) .. controls (2.2,-1) .. (2.7,0);

\draw[blue, very thick](3.6,1) .. controls (4.1,2) .. (4.6,1);
\draw[red,  very thick] (4.8,0) .. controls (5.3,-1) .. (5.8,0);
\draw[blue, very thick] (6.0,1) .. controls (6.5, 2) .. (7.0,1);
\draw[red,  very thick] (7.2,0) .. controls (7.7,-1) .. (8.2,0);
\draw[blue, very thick] (8.4,1) .. controls (8.9, 2) .. (9.4,1);

\draw[red,  very thick] (10.1,0) .. controls (10.6,-1) .. (11.1,0);
\draw[blue, very thick] (11.3,1) .. controls (11.8,2) .. (12.3,1);
\draw[red,  very thick] (12.5,0) .. controls (13.0,-1) .. (13.5,0);
\draw[blue, very thick] (13.7,1) .. controls (14.2,2) .. (14.7,1);
\draw[red, very thick] (14.9,0) .. controls (15.4,-1) .. (15.9,0);

\node at (-2.6,-1.1) {$u_1$};
\node at (-0.2,-1.1) {$u_2$};
\node at (2.2,-1.1) {$u_3$};
\node at (5.3,-1.1)   {$u_j$};
\node at (7.7,-1.1)   {$u_{j+1}$};
\node at (10.6,-1.1)  {$u_{k-1}$};
\node at (13.0,-1.1)  {$u_{k}$};
\node at (15.4,-1.1)  {$u_{k+1}$};

\node at (6.5,0.5)   {$x$};
\node at (3.1,0.5) {$\cdots$};
\node at (10,0.5) {$\cdots$};
\node at (-1.4,2.1) {$v_1$};
\node at (1,2.1) {$v_2$};
\node at (4.1,2.1) {$v_{j-1}$};
\node at (6.5,2.1)   {$v_j$};
\node at (8.9,2.1)   {$v_{j+1}$};
\node at (11.8,2.1)  {$v_{k-1}$};
\node at (14.2,2.1)  {$v_k$};

\draw[thick, decorate, decoration={snake, amplitude=1mm, segment length=5mm}]
      (-0.9,1) -- (-0.7,0);
\draw[thick, decorate, decoration={snake, amplitude=1mm, segment length=5mm}]
      (1.5,1) -- (1.7,0);
\draw[thick, decorate, decoration={snake, amplitude=1mm, segment length=5mm}]
      (4.6,1) -- (4.8,0);


\draw[dashed, thick] (8.2,0) -- (8.4,1);
\draw[dashed, thick] (11.1,0) -- (11.3,1);
\draw[dashed, thick] (13.5,0) -- (13.7,1);

\end{tikzpicture}
\caption{The $x$-decomposition of $\sigma$.}\label{fig:x-decomposition}
\end{figure}

\begin{lemma}\label{lem:height}
    Given a permutation $\sigma\in\S_{n+1}$ and the corresponding Laguerre history  $\Psi_{FV}(\sigma) = (w,p) \in \mathcal{L}_{n}$.  For any $x\in[n]$, if $h_x$ is the height of the $x$th step of $w$, then 
    \begin{equation}\label{equ:h_x}
        (31\!-\!2)_x(\sigma)+(2\!-\!13)_x(\sigma)=h_x.
    \end{equation}
\end{lemma}
\begin{proof}
Suppose that the $x$-decomposition of $\sigma$ is
\begin{equation}
   \sigma = u_1 v_1 \cdots u_k v_k u_{k+1}, \quad k \geq 1,
\end{equation}
where $x$ lies inside $v_j$. According to Definition~\ref{def:height} and the construction of $\Psi_{FV}$ in~\eqref{def:corresponding}, it follows that the height of the \(x\)th step of \(w\) is
\[
h_x=\sum_{j=1}^{k+1}\bigl(\val(u_j)-\pk(u_j)\bigr),
\]
where \(\val(u_j)\) (resp. \(\pk(u_j)\)) denotes the number of valleys (resp. peaks) in \(u_j\).
 It is straightforward  to verify that 
    \begin{itemize}
        \item $\val(u_j)=\pk(u_j)$ for $j=1,k+1$;
        \item $\val(u_j)=\pk(u_j)+1$ for $2\le j\le k$.
    \end{itemize}
Hence we obtain $h_x(\sigma)=k-1$.  Every adjacent pair $(v_i,u_{i+1})$ with $i\le j-1$ contributes to a $(31\!-\!2)_x(\sigma)$, while every adjacent pair $(u_i,v_{i})$  with $i\ge j$ contributes to a $(2\!-\!13)_x(\sigma)$.
Thus, $h_x$ also can be calculated by using the patterns $(31\!-\!2)_x(\sigma)$ and $(2\!-\!13)_x(\sigma)$, that is~\eqref{equ:h_x}.

    Now we prove~Eq.~\eqref{equ:h_x} by induction for all $x\in[n]$. For $x=1$, it is obvious that $(31\!-\!2)_1(\sigma)+(2\!-\!13)_1(\sigma)=0=h_1$.  If $
x>1$, it suffices to show that  \begin{equation}\label{equ:h_x-nu}
    (31\!-\!2)_x(\sigma)+(2\!-\!13)_x(\sigma)=(31\!-\!2)_{x-1}(\sigma)+(2\!-\!13)_{x-1}(\sigma)+\nu,
\end{equation}
    where $\nu\in \{0, 1,-1\}$. 
    Assume that the $(x\!-\!1)$-decomposition of $\sigma$ is  \[
\sigma = u_1 v_1 \cdots u_k v_k u_{k+1}, \quad k \geq 1,
\] and $(x\!-\!1)$ is in $v_j$. In view of this $(x-1)$-decomposition, we see that $(31\!-\!2)_{x-1}(\sigma)=j-1$,  $(2\!-\!13)_{x-1}(\sigma)=k-j$ and $(31\!-\!2)_{x-1}(\sigma)+(2\!-\!13)_{x-1}(\sigma)=k-1$.
\begin{itemize}
    \item 
    For case $\nu=1$. If $(x-1)$ is a valley of $\sigma$, and $x$ is also a valley, 
    \begin{itemize}
        \item  then  $x$ is in $v_l$ where $1\le l< j$ (resp. $j<l\le k$). In this case,  $(x-1)$ together with its immediate right (resp. left) neighbor contributes 1 to  $(2\!-\!13)_x(\sigma)$ (resp. $(31\!-\!2)_x(\sigma)$). Hence, we have
        \[(31\!-\!2)_{x}(\sigma)+(2\!-\!13)_{x}(\sigma)=(l-1)+(k-l)+1=k.\]
        \item  Or $x$ is in $v_j$.  Suppose that $v_j=v_j^l(x-1)v_j^r$ where $v_j^l$ and $v_j^r$ are (contiguous) words consisting of entries strictly larger than $x-1$. Both $v_j^l$ and $v_j^r$ are 
        nonempty.
If $x\in v_j^{l}$, then $(x\!-\!1)$ together with its immediate right neighbor contributes 1 to $(2\!-\!13)_x(\sigma)$; if $x\in v_j^{r}$, then $x-1$ together with its immediate left neighbor contributes 1 to $(31\!-\!2)_x(\sigma)$. Hence, we have
        \[(31\!-\!2)_{x}(\sigma)+(2\!-\!13)_{x}(\sigma)=(j-1)+(k-j)+1=k.\]
    \end{itemize}
    Analogous arguments hold in the cases when \(x\) is a peak, a double ascent, or a double descent of \(\sigma\). 
    
\item For case $\nu=-1$. If $(x-1)$ is a peak of $\sigma$, then from the $(x-1)$-decomposition we see that $v_j$ contains only $(x-1)$. Assume that $x$ is a valley of $\sigma$, and that $x$ is in $v_l$ where $1\le l<j$ (resp. $j<l\le k$). In the $x$-decomposition of $\sigma$, components $u_j$, $v_j=x-1$ and $u_{j+1}$ will form a new $u$-word. Therefore, compared with the 
$(x-1)$-decomposition, the $x$-decomposition has one fewer $u$-word.   Hence, we have
        \[(31\!-\!2)_{x}(\sigma)+(2\!-\!13)_{x}(\sigma)=(l-1)+(k-l)-1=k-2.\]
Analogous arguments hold in the cases when \(x\) is a peak, a double ascent, or a double descent of \(\sigma\).

\item For case $\nu=0$. If $(x-1)$ is a double descent (resp. double ascent) of $\sigma$, then $(x-1)$ is the last (resp. first) element of $v_j$. Assume that $x$ is a valley of $\sigma$ and that $x$ is in $v_l$ where $1\le l\le k$. In these cases, it is easy to check that there is no change in the total number of $v_i$ in both $(x-1)$-decomposition and $x$-decomposition of $\sigma$. Hence, we have 
\[(31\!-\!2)_{x}(\sigma)+(2\!-\!13)_{x}(\sigma)=(l-1)+(k-l)=k-1.\]
Analogous arguments hold in the cases when \(x\) is a peak, a double ascent, or a double descent of \(\sigma\). Combining these cases,  we obtain~\eqref{equ:h_x-nu}.
\end{itemize}
\end{proof}

For a Motzkin path $w$, we define $\area(w)$  to be  the geometric area of the region between  the $x$-axis and the Motzkin path $w$, denoted by $\area(w)$. We have the following characterization for $\area$.
\begin{lemma}\label{lem:area}
    Let a Motzkin path $w\in\mathcal{M}_n$. The area of $w$ is given by 
    \begin{equation}\label{area:sum-h}
        \mathrm{area}(w)=\sum_{i=1}^{n}h_i,
    \end{equation}
    where $h_i$ is the height of the $i$th step of $w$ defined in~\eqref{def:height}.
\end{lemma}
\begin{proof}
    We compute $\area(w)$ by the step heights of the Motzkin path $w\in\mathcal{M}_n$.
The contribution of the $i$th step equals the area of the trapezoid of width $1$ whose
vertical sides have lengths $h_i$ and $h_i+\Delta_i$, where $\Delta_i\in\{1,0,-1\}$ is the
vertical increment of that step. It is clear that  the $i$th step contributes
$h_i+\frac{\Delta_i}{2}$
to $\area(w)$, therefore 
\[
\area(w)=\sum_{i=1}^{n}\left(h_i+\frac{\Delta_i}{2}\right)
=\sum_{i=1}^{n} h_i+\frac12\sum_{i=1}^{n}\Delta_i.
\]
Since $w$ starts and ends on the $x$-axis, its total vertical increment is
$\sum_{i=1}^{n}\Delta_i=0$, and the above equation yields~\eqref{area:sum-h}. 
\end{proof}

\begin{lemma}\label{lem:dis}
For $\sigma\in\mathfrak{S}_{n+1}$, let $\Psi_{FV}(\sigma):=(w,p)\in\mathcal{L}_n$  be the corresponding Laguerre history.
Then 
    \begin{equation}\label{lemma-area}
        \dis(\sigma)=\area(w).
    \end{equation} 
\end{lemma}
    \begin{proof}
    Let $w=w_1\cdots w_n\in \mathcal{M}_n$.
   According to the construction of $\Psi_{FV}$ in~\eqref{def:corresponding},  we see that 
\begin{equation}\label{dep:S-N}   \dis(\sigma)=\sum_{w_i=\mathrm{S}}i-\sum_{w_j=\mathrm{N}}j.
\end{equation}

 For each South-East step $w_i=\mathrm{S}$, match it with the unique North-East step $w_j=\mathrm{N}$ with $j<i$ such that the path between steps $w_j$ and $w_i$ stays
strictly above height $h_j$, and $h_i=h_j+1$.
The matched pair of steps $w_i=\mathrm{S}$ and $w_j=\mathrm{N}$ determines a horizontal strip of area $i-j$ lying within  the Motzkin path. Summing over all matched pairs gives 
\begin{equation}\label{area:S-N}
    \area(w)=\sum_{w_i=\mathrm{S}}i-\sum_{w_j=\mathrm{N}}j.
\end{equation}
Combining this with Eq.~\eqref{dep:S-N} yields~\eqref{lemma-area}.
\end{proof}

\begin{example}
    Let $\sigma=697835142\in\S_{9}$.  The corresponding Motzkin path is $w=\mathrm{N\red{E_1}NSS\blue{E_2}NS}$; see Fig.~\ref{fig:motzkin}. We have 
    \begin{align*}
    \dis(\sigma)&=(8+5+4)-(7+3+1)=6,\\
       \area(w)&=\sum_{i=1}^8h_i=4\cdot1+2=6. 
    \end{align*}
\end{example}

\subsubsection{Foata-Zeilberger bijection}\label{sec-FZ}

A \emph{3-Motzkin path} is a Motzkin path whose east steps could come in three colors: red (\red{$\mathrm{E_1}$}), blue (\blue{$\mathrm{E_2}$}) and green (\green{$\mathrm{E_3}$}).
\begin{definition}[3-restricted Laguerre history]\label{3-Laguerre}
    A $3$-restricted Laguerre history of length $n$ is a couple  $(w, p)$, where $w=w_1\cdots w_n$ is a  3-Motzkin path of length $n$ starting at $(0,0)$ and ending at $(n,0)$, and 
$p=(p_1, \ldots, p_n)$ is a label sequence satisfying $1\le i\le n$,
\[
0 \leq p_i \leq 
\begin{cases}
h_i,   & \text{if } w_i \in \{\mathrm{N,\blue{\mathrm{E_2}}},\green{\mathrm{E_3}}\};\\[4pt]
h_i-1, & \text{if } w_i \in\{ \mathrm{S,\red{E_1}}\},
\end{cases}
\]
 where $h_i$ is the height of the $i$th step of $w$.  
 Denote by $\mathcal{L}_n^*$   the set of 3-restricted Laguerre histories of length $n$. 
\end{definition}

 For a permutation $\sigma\in\S_n$, we describe Foata--Zeilberger's
bijection $\Psi_{FZ}(\sigma):=(w,p)\in\mathcal{L}_n^*$ and its inverse.
The associated 3-restricted Laguerre history
\[
\Psi_{FZ}(\sigma)
   = \bigl((w_1,\dots,w_n),(p_1,\dots,p_n)\bigr)\in\mathcal{L}_n^*
\]
is defined as follows. The step $w_i$ is given by
\begin{equation}\label{equ:FZ-corresponding}
w_i=
\begin{cases}
\mathrm{N}\;(\text{resp.\ }\mathrm{S})
  & \text{if $i$ is a cycle valley (resp.\ cycle peak) of $\sigma$;}\\[2pt]
\textcolor{red}{\mathrm{E}_1}\;(\text{resp.\ }\textcolor{blue}{\mathrm{E}_2})
  & \text{if $i$ is a cycle double descent (resp.\ cycle double ascent) of $\sigma$;}\\[2pt]
\textcolor{green}{\mathrm{E}_3}
  & \text{if $i$ is a fixed point of $\sigma$,}
\end{cases}
\end{equation}
and the integer $p_i$ is defined by
\begin{equation}\label{eq:pi}
p_i :=
\begin{cases}
\#\{\, j<i : \sigma(j)>\sigma(i)\,\},
  & \text{if }\sigma(i)\ge i;\\[4pt]
\#\{\, j>i : \sigma(j)<\sigma(i)\,\},
  & \text{if }\sigma(i)< i.
\end{cases}
\end{equation}

For a permutation $\sigma=\sigma(1)\cdots\sigma(n)\in\S_n$, we say that $i$ (resp. $\sigma_i$) is the excedance bottom (resp. top) of $\sigma$ if $i<\sigma(i)$, and that $i$ (resp. $\sigma_i$) is the nonexcedance top (resp. bottom) of $\sigma$ if  $i\ge\sigma(i)$. Let $\mathrm{Etop}(\sigma)$, $\mathrm{Ebot}(\sigma)$, $\mathrm{NEtop}(\sigma)$ and $\mathrm{NEbot}(\sigma)$ denote the sets of the excedance tops, excedance bottoms, nonexcedance tops and nonexcedance bottoms of $\sigma$, respectively. Let $\mathrm{Cval}(\sigma)$, $\mathrm{Cpk}(\sigma)$, $\mathrm{Cda}(\sigma)$, $\mathrm{Cdd}(\sigma)$ and $\mathrm{Fix}(\sigma)$ be the sets of cycle valleys, cycle peaks, cycle double ascents, cycle double descents and fixed points of $\sigma$, respectively.

It is easy to check that the following identities hold: 
 \begin{subequations}\label{equ:clas}
      \begin{align}
    \mathrm{Cpk}(\sigma)&\cup\mathrm{Cda}(\sigma)=\mathrm{Etop}(\sigma),\quad
        \mathrm{Cval}(\sigma)\cup\mathrm{Cda}(\sigma)=\mathrm{Ebot}(\sigma),\label{exc-top-bottom}\\        
        &\mathrm{Cpk}(\sigma)\cup\mathrm{Cdd}(\sigma)\cup\mathrm{Fix}(\sigma)=\mathrm{NEtop}(\sigma),\\
        &\mathrm{Cval}(\sigma)\cup\mathrm{Cdd}(\sigma)\cup\mathrm{Fix}(\sigma)=\mathrm{NEbot}(\sigma).
    \end{align}
 \end{subequations}

\begin{definition}\label{inversion-table}
    The left-to-right inversion table (or Lehmer code) of a permutation $\sigma$ is the sequence
$(e_1,e_2,\dots,e_n)$, where $e_i$ denotes the number of entries to the left of
position $i$ that are bigger than $\sigma_i$, i.e.,  
\begin{equation}
    e_i=\#\{\,j<i:\ \sigma_j>\sigma_i\,\}.
\end{equation}
Dually, the right-to-left inversion table of $\sigma$ is the sequence
$(f_1,\dots,f_n)$, where $f_i$ denotes the number of  entries to the right of position $i$ that are smaller than $\sigma_i$, i.e., 
\begin{equation}
    f_i=\#\{\,j>i:\ \sigma_j<\sigma_i\,\}.
\end{equation}
\end{definition}

The inverse algorithm $\Psi_{FZ}^{-1}$ building a permutation $\sigma$ from a 3-restricted Laguerre history $(w,p)\in\L_{n}^*$ can be described as follows:

According to the classification of steps in 3-restricted Laugerre history and Eq.~\eqref{equ:clas}, we first obtain the four sets $\mathrm{Etop}(\sigma)$, $\mathrm{Ebot}(\sigma)$, $\mathrm{NEtop}(\sigma)$, and $\mathrm{NEbot}(\sigma)$, where $\sigma$ is the resulting permutation obtained at the end of the inverse bijection. We then construct two biwords,
\[
\sigma_{\exc}:=\left(\genfrac{}{}{0pt}{}{\eta}{\eta'}\right) \text{ and } \sigma_{\mathrm{nexc}}:=\left(\genfrac{}{}{0pt}{}{\zeta}{\zeta'}\right),
\]
where $\eta$ and $\zeta$ are the subwords, written in increasing order, of the sets $\mathrm{Ebot}(\sigma)$ and $\mathrm{NEtop}(\sigma)$, respectively, and $\eta'$ and $\zeta'$ are some subwords of $\mathrm{Etop}(\sigma)$ and $\mathrm{NEtop}(\sigma)$ obtained from the left-to-right and right-to-left tables with entries $p_i$, respectively. Then we recover the corresponding permutation $\Psi_{FZ}^{-1}((w,p))$.

As a consequence, the cardinality of $\L_n^*$ is $n!$.

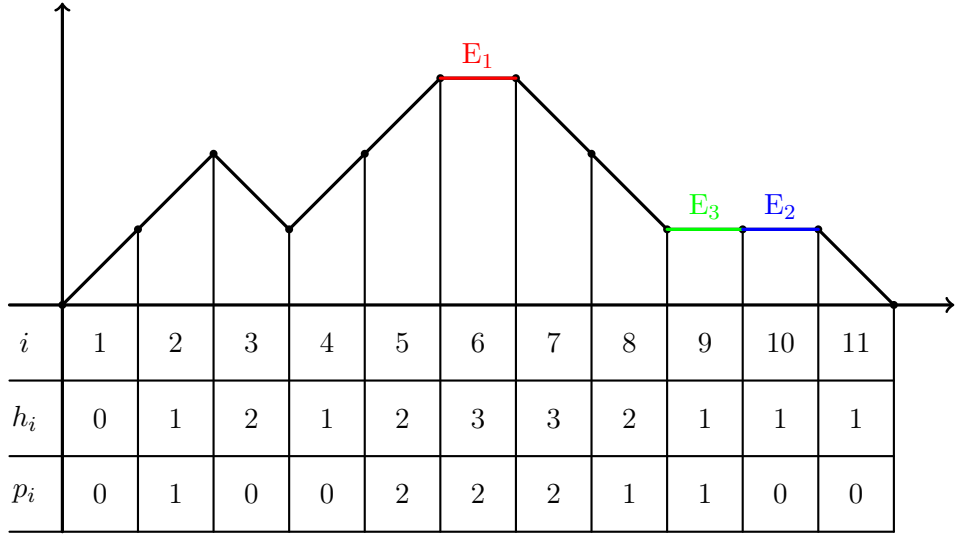
\begin{figure}[t] 
  \centering

\begin{tikzpicture}[scale=1, line cap=round, line join=round]


\draw[very thick,->] (0,-3) -- (0,4); 
\draw[thick] (1,1.0) -- (1,-3);
\draw[thick] (2,2.0) -- (2,-3);
\draw[thick] (3,1.0) -- (3,-3);
\draw[thick] (4,2.0) -- (4,-3);
\draw[thick] (5,3.0) -- (5,-3);
\draw[thick] (6,3.0) -- (6,-3);
\draw[thick] (7,2.0) -- (7,-3);
\draw[thick] (8,1.0) -- (8,-3);
\draw[thick] (9,1.0) -- (9,-3);
\draw[thick] (10,1.0) -- (10,-3);
\draw[thick] (11,0.0) -- (11,-3);

\fill[color=black] (0,0) circle (1.5pt);
\fill[color=black] (1,1) circle (1.5pt);
\fill[color=black] (2,2) circle (1.5pt);
\fill[color=black] (3,1) circle (1.5pt);
\fill[color=black] (4,2) circle (1.5pt);
\fill[color=black] (5,3) circle (1.5pt);
\fill[color=black] (6,3) circle (1.5pt);
\fill[color=black] (7,2) circle (1.5pt);
\fill[color=black] (8,1) circle (1.5pt);
\fill[color=black] (9,1) circle (1.5pt);
\fill[color=black] (10,1) circle (1.5pt);
\fill[color=black] (11,0) circle (1.5pt);

\draw[very thick,->] (-0.7,0) -- (11.8,0);

\draw[very thick]
  (0,0) -- (1,1)--(2,2) -- (3,1) -- (4,2) -- (5,3) -- (6,3)
  -- (8,1) -- (9,1) -- (10,1)--(11,0);
 \node at (5.5,3.3) {\red{$\mathrm{E_1}$}};
 \node at (9.5,1.3) {\blue{$\mathrm{E_2}$}};
 \node at (8.5,1.3) {\green{$\mathrm{E_3}$}};
\draw[very thick,red]  (5,3) -- (6,3);
\draw[very thick,green] (8,1) -- (9,1);
\draw[very thick,blue] (9,1) -- (10,1);
\foreach \y in {-1,-2,-3} {
  \draw[thick] (-0.7,\y) -- (11,\y);
}

\node at (-0.5,-0.5) {$i$};
\node at (-0.5,-1.5) {$h_i$};
\node at (-0.5,-2.5) {$p_i$};

\foreach \k in {1,...,11} {
  \node at (\k-0.5,-0.5) {\k};
}
\foreach \k/\v in {1/0,2/1,3/2,4/1,5/2,6/3,7/3,8/2,9/1,10/1,11/1} {
  \node at (\k-0.5,-1.5) {\v};
}
\foreach \k/\v in {1/0,2/1,3/0,4/0,5/2,6/2,7/2,8/1,9/1,10/0,11/0} {
  \node at (\k-0.5,-2.5) {\v};
}

\end{tikzpicture}
\caption{The 3-restricted Laguerre history $\Psi_{FZ}(\sigma)$ with 
$\sigma = 8\, 3\, 1\, 10\, 7\, 5\, 6\, 4\,9\, 11\,2\, = (1,8,4,10,11,2,3)\,(5,7,6)\,(9).$}
  \label{fig2:motzkin}
\end{figure}
\begin{example}
    If $\sigma = 8\, 3\, 1\, 10\, 7\, 5\, 6\, 4\,9\, 11\,2\,
           = (1,8,4,10,11,2,3)\,(5,7,6)\,(9)\in \S_{11}$, then we have $$\Psi_{FZ}(\sigma)=(\mathrm{NNSNN\red{E_1}SS\green{E_3}\blue{E_2}S},(0,1,0,0,2,2,2,1,1,0,0))\in\L_{11}^*.$$ 
Inversely, from the 3-restricted Laguerre history, we have   
\begin{align*}
    \mathrm{Cval}(\sigma)=\{1,2,4,5\}, \mathrm{Cpk}(\sigma)=\{3,7,8,11\},\\\mathrm{Cdd}(\sigma)=\{6\}, \mathrm{Cda}(\sigma)=\{10\} \text{ and }\mathrm{Fix}(\sigma)=\{9\}.
\end{align*}
     According to Eq.~\eqref{equ:clas}, we obtain
    \begin{align*}
        &\mathrm{Ebot}(\sigma)=\{1,2,4,5,10\},\,
        \mathrm{Etop}(\sigma)=\{3,7,8,10,11\};\\
        &\mathrm{NEbot}(\sigma)=\{1,2,4,5,6,9\}\text{ and }\mathrm{NEtop}(\sigma)=\{3,6,7,8,9,11\}.
    \end{align*}
    We now construct the biwords $\sigma_{\exc}$ and $\sigma_{\mathrm{nexc}}$
    \[\sigma_{\exc}:=
\left(\genfrac{}{}{0pt}{}{\eta}{\eta'}\right)
= \left(
\genfrac{}{}{0pt}{}{1_0}{8}\,
\genfrac{}{}{0pt}{}{2_1}{3}\,
\genfrac{}{}{0pt}{}{4_0}{10}\,
\genfrac{}{}{0pt}{}{5_2}{7}\,
\genfrac{}{}{0pt}{}{10_0}{11}
\right),
\quad
\sigma_{\mathrm{nexc}}:=\left(\genfrac{}{}{0pt}{}{\zeta}{\zeta'}\right)
= \left(
\genfrac{}{}{0pt}{}{3_0}{1}\,
\genfrac{}{}{0pt}{}{6_2}{5}\,
\genfrac{}{}{0pt}{}{7_2}{6}\,
\genfrac{}{}{0pt}{}{8_1}{4}\,
\genfrac{}{}{0pt}{}{9_1}{9}\,
\genfrac{}{}{0pt}{}{11_0}{2}
\right).
\]
Combining these two biwords yields the  permutation $\sigma=8\, 3\, 1\, 10\, 7\, 5\, 6\, 4\,9\, 11\,2$; see Fig.~\ref{fig2:motzkin}.       
\end{example}

\begin{definition}[Indexed crossing]\label{def:crossing}
For \(k\in[n]\) and \(\sigma\in\mathfrak{S}_n\), the upper and lower crossings indexed at \(k\) of \(\sigma\) are defined, respectively, by
\[\cros_k^+(\sigma)
:= \#\bigl\{l \,|\,
    l < k \le \sigma_l<\sigma_k
\bigr\} \text{ and } \cros_{k}^-(\sigma)=\#\bigl\{ l\,|\, 
    \sigma_k<  \sigma_\ell<k < l\}.
\]
The number of crossings indexed at $k$ is $\cros_k(\sigma)=\cros_k^+(\sigma)+\cros_{k}^-(\sigma)$; see Fig~\ref{fig:cros}.
\end{definition}
\begin{figure}[htp]
\centering
\begin{picture}(250,70)(-30,0)
\setlength{\unitlength}{1.5mm}
\linethickness{.5mm}

\put(-23,10){\line(1,0){28}}
\put(-20,10){\circle*{1.3}}\put(-20,10){\makebox(0,-6)[c]{\small $l$}}
\put(-13,10){\circle*{1.3}}\put(-13,10){\makebox(0,-6)[c]{\small $k$}}
\put(-6,10){\circle*{1.3}}\put(-6,10){\makebox(0,-6)[c]{\small $\sigma_{l}$}}
\put(1,10){\circle*{1.3}}\put(1,10){\makebox(0,-6)[c]{\small $\sigma_k$}}
\red{\qbezier(-20,10)(-13,20)(-6,10)}
\blue{\qbezier(-14,10)(-7,20)(0,10)}
\put(10,2){\makebox(0,0)[c]{\small Upper crossing $k\leq \sigma_l$}}

\put(8,10){\line(1,0){28}}
\put(11,10){\circle*{1.3}}\put(11,10){\makebox(0,-6)[c]{\small $l$}}
\put(22,10){\circle*{1.3}}\put(22,10){\makebox(0,-6)[c]{\small $k=\sigma_{l}$}}
\put(32,10){\circle*{1.3}}\put(32,10){\makebox(0,-6)[c]{\small $\sigma_k$}}
\red{\qbezier(11,10)(16,20)(22,10)}
\blue{\qbezier(21.5,10)(25.5,20)(31,10)}

\put(39,10){\line(1,0){28}}
\put(43,10){\circle*{1.3}}\put(43,10){\makebox(0,-6)[c]{\small $\sigma_k$}}
\put(50,10){\circle*{1.3}}\put(50,10){\makebox(0,-6)[c]{\small $\sigma_{l}$}}
\put(57,10){\circle*{1.3}}\put(57,10){\makebox(0,-6)[c]{\small $k$}}
\put(64,10){\circle*{1.3}}\put(64,10){\makebox(0,-6)[c]{\small $l$}}
\red{\qbezier(43,10)(50,0)(57,10)}
\blue{\qbezier(49,10)(56,0)(63,10)}
\put(53,2){\makebox(0,0)[c]{\small Lower crossing}}

\end{picture}

\caption{Upper crossing and lower crossing indexed at $k$.}\label{fig:cros}
\end{figure}

\begin{definition}[Indexed nesting]\label{def:nesting}
    For \(k\in[n]\) and \(\sigma\in\mathfrak{S}_n\), the upper and lower nestings indexed at \(k\) of \(\sigma\) are defined, respectively, by
\[\nest_k^+(\sigma)
:= \#\bigl\{l \,|\,
    l < k \le \sigma_k < \sigma_l
\bigr\} \text{ and } \nest_{k}^-(\sigma)=\#\bigl\{ l\,|\, 
    \sigma_{l} < \sigma_{k} < k < l\}.
\]
The number of crossings indexed at $k$ is $\nest_k(\sigma)=\nest_k^+(\sigma)+\nest_{k}^-(\sigma)$; see Fig.~\ref{fig:nest}.
\end{definition}
\begin{figure}[htp]
\centering
\begin{picture}(400,80)(0,0)
\setlength{\unitlength}{1.5mm}
\linethickness{.5mm}

\put(2,10){\line(1,0){28}}
\put(5,10){\circle*{1.3}}\put(5,10){\makebox(0,-6)[c]{\small $l$}}
\put(12,10){\circle*{1.3}}\put(12,10){\makebox(0,-6)[c]{\small $k$}}
\put(19,10){\circle*{1.3}}\put(19,10){\makebox(0,-6)[c]{\small $\sigma_k$}}
\put(26,10){\circle*{1.3}}\put(26,10){\makebox(0,-6)[c]{\small $\sigma_l$}}
\red{\qbezier(5,10)(15.5,20)(26,10)}
\blue{\qbezier(11,10)(14.5,15)(18,10)}
\put(33,2){\makebox(0,0)[c]{\small Upper nesting $k\le \sigma_k$}}

\put(36,10){\line(1,0){28}}
\put(39,10){\circle*{1.3}}\put(39,10){\makebox(0,-6)[c]{\small $l$}}
\put(49,10){\circle*{1.3}}\put(49,10){\makebox(0,-6)[c]{\small $k=\sigma_k$}}
\put(60,10){\circle*{1.3}}\put(60,10){\makebox(0,-6)[c]{\small $\sigma_l$}}
\red{\qbezier(39,10)(49.5,20)(60,10)}
\blue{\qbezier(48.3,10)(50.5,13)(47.5,13)
      \qbezier(48.3,10)(46,11)(47.5,13)}

\put(70,10){\line(1,0){28}}
\put(74,10){\circle*{1.3}}\put(74,10){\makebox(0,-6)[c]{\small $\sigma_l$}}
\put(81,10){\circle*{1.3}}\put(81,10){\makebox(0,-6)[c]{\small $\sigma_k$}}
\put(88,10){\circle*{1.3}}\put(88,10){\makebox(0,-6)[c]{\small $k$}}
\put(95,10){\circle*{1.3}}\put(95,10){\makebox(0,-6)[c]{\small $l$}}
\red{\qbezier(74,10)(84,1)(95,10)}
\blue{\qbezier(80,10)(83.5,6)(87.5,10)}
\put(84,2){\makebox(0,0)[c]{\small Lower nesting}}

\end{picture}

\caption{Upper nesting and lower nesting indexed at $k$}\label{fig:nest}
\end{figure} 
Note that for any $\sigma\in\S_n$, we have \[\cros(\sigma)=\sum_{k=1}^{n}\cros_k(\sigma) \text{ and } \nest(\sigma)=\sum_{k=1}^{n}\nest_k(\sigma) \]
The definition of labels $p_i$ in \eqref{eq:pi} admits a straightforward interpretation in terms of the statistic nesting,
\[
p_i :=
\begin{cases}
\operatorname{nest}_i^+(\sigma),
& \text{if } \sigma(i) \ge i \quad \text{i.e. } i \in \mathrm{Cval}(\sigma)\cup \mathrm{Cda}(\sigma)\cup\mathrm{Fix}(\sigma);\\[4pt]
\operatorname{nest}_i^-(\sigma),
& \text{if } \sigma(i) < i \quad \text{i.e. } i \in \mathrm{Cpk}(\sigma)\cup \mathrm{Cdd}(\sigma).
\end{cases}
\]

The following Lemma is due to Corteel~\cite[Lemma~3]{Cor07}. Since the  original proof is brief,  we provide a visual proof based on arc diagrams of permutations.

\begin{lemma}[\cite{Cor07}]
\label{lem:height-FZ}
Let $\Psi_{FZ}(\sigma)=(w,p)\in\L_n^*$ for $\sigma\in\S_n$. For any $k\in[n]$,
    if $h_k$ is the height of the $k$th step of $w$, then
\begin{equation}\label{def:height-FZ}
    \cros_k(\sigma)+\nest_k(\sigma)=\begin{cases}
        h_k, &\text{ if }k\le \sigma(k);\\[4pt]
        h_k-1,&\text{ if } k> \sigma(k).
    \end{cases}
\end{equation}
\end{lemma}
\begin{proof}
    It is easy to check that the height of the $k$th step of $w$ is equal to the number of upper arcs (or lower arcs) that cover node $k$, that is, more precisely, arcs satisfying one of the following conditions:
\begin{itemize}
    \item the left endpoint lies to the left of $k$ and the right endpoint lies to the right of $k$;
\item  the left endpoint lies to the left of $k$ and the right endpoint is $k$;
\item  the left endpoint is $k$ and the right endpoint lies to the right of $k$;
\item  both the left and right endpoints are $k$.
\end{itemize}
In view of the upper and lower crossing (resp. nesting) indexed at $k$, see Definition~\ref{def:crossing} and~\ref{def:nesting}, it suffices to show that
    \begin{itemize} 
    \item[(1)] if $k\in\Fix(\sigma)$, then  $\cros_{k}^{+}(\sigma)=0$ and $\nest_{k}^{+}(\sigma)=h_k$;
    \item[(2)] if $k\in\Cval(\sigma)$, then \[\cros_{k}^{+}(\sigma)+\nest_{k}^{+}(\sigma)=h_k\;\text{ with }\; 0\le\nest_{k}^{+}(\sigma)\le h_k;\]
    \item[(3)] if $k\in\Cda(\sigma)$, then \[\cros_{k}^{+}(\sigma)+\nest_{k}^{+}(\sigma)=h_k\;\text{ with }\; 0\le\nest_{k}^{+}(\sigma)\le h_k-1;\]
    \item[(4)] and if $k\in\Cpk(\sigma)\cup\Cdd(\sigma)$, then \[\cros_{k}^{-}(\sigma)+\nest_{k}^{-}(\sigma)=h_k-1\;\text{ with }\; 0\le\nest_{k}^{-}(\sigma)\le h_k-1.\]
    \end{itemize}   
\begin{center}
    \begin{picture}(72,50)(145, 0)
    \setlength{\unitlength}{1.5mm}
\linethickness{.3mm}
\put(2,0){\line(1,0){28}}
\put(5,0){\circle*{1,3}}\put(5,0){\makebox(0,-6)[c]{\small $1$}}
\put(12,0){\circle*{1,3}}\put(12,0){\makebox(0,-6)[c]{\small $2$}}
\put(19,0){\circle*{1,3}}\put(19,0){\makebox(0,-6)[c]{\small $3$}}
\put(26,0){\circle*{1,3}}\put(26,0){\makebox(0,-6)[c]{\small $4$}}
\put(34.5,0){\makebox(0,0){\Large $\boldsymbol{\cdots}$}}
\put(46,8.4){\makebox(0,0){\Large $\boldsymbol{\cdots}$}}
\put(46,5.4){\makebox(0,0){\Large $\boldsymbol{\cdots}$}}
\put(48,-8.5){\makebox(0,0){\Large $\boldsymbol{\cdots}$}}

\put(39,0){\line(1,0){28}}
\put(43,0){\circle*{1,3}}\put(43,0){\makebox(0,-6)[c]{\small $k$}}
\qbezier(5,0)(12,10)(19,0)
\qbezier(5,0)(12,-16)(50,-10)
\qbezier(26,0)(30,5)(47,7)
\qbezier(12,0)(18,10)(50,10)
\qbezier(12,0)(15,-10)(49,-7)
\qbezier(19,0)(22,5,8)(26,0)
\red{\qbezier(43,0)(47,4)(50,5)}
\blue{\qbezier(42,0)(40,-4)(38,-5)}
\end{picture}
\vspace{2cm}
\end{center}
Case $(1)$ is obvious. If $k\in\Cval(\sigma)$, and there are $h_k=j$ arcs covering $k$, then the arc from $k$ to $\sigma(k)$ must intersect exactly $r$ of these arcs, for some $0\le r\le j$. We order these $j$ arcs from top to bottom and denote them by $a_1,a_2,\dots,a_j$. Suppose that the arc $(k \to \sigma(k))$ intersects exactly $i$ of the 
$j$ arcs above $k$, namely $a_{j-i+1}, a_{j-i}, \dots, a_{j}$, where $0 \le i \le j$. 
These intersected arcs together with the arc $(k \to \sigma(k))$ 
contribute \[
\cros_k^{+}(\sigma)=i,
\]
while the remaining $j-i$ arcs that do not intersect $(k \to \sigma(k))$ 
contribute 
\[
\nest_k^{+}(\sigma)=j-i,
\]
which proves~(2). 

If $k\in\Cda(\sigma)$ and there are $h_k=j$ arcs that cover $k$, exactly one of which has right endpoint $k$ (this arc together with the arc ($k\to \sigma(k)$ forms an upper crossing), then the arc from $k$ to $\sigma(k)$ must intersect exactly $r$ of the remaining $j-1$ arcs, for some $0\le r\le j-1$.
We order these $j-1$ arcs from top to bottom and denote them by $a_1,a_2,\dots,a_{j-1}$. Suppose that the arc $(k \to \sigma(k))$ intersects exactly $i$ of the 
$j-1$ arcs above $k$, namely $a_{j-i}, a_{j-i}, \dots, a_{j-1}$, where $0 \le i \le j-1$. 
These intersected arcs, together with the arc $(k \to \sigma(k))$, 
contribute \[
\cros_k^{+}(\sigma)=i,
\]
while the remaining $j-i-1$ arcs that do not intersect $(k \to \sigma(k))$ 
contribute 
\[
\nest_k^{+}(\sigma)=j-i-1,
\]
which proves~(3).

If $k\in\Cpk(\sigma)\cup\Cdd(\sigma)$ and there are $h_k=j$ arcs below $k$, denoted by $b_1,b_2,\dots,b_j$ from bottom to top,   exactly one of which has a right endpoint $k$, namely $b_i$ where $1\le i\le j$, then the arc ($k\to\sigma(k)$) must intersect exactly $b_{i+1},b_{i+2},\dots,b_j$.  These intersected arcs, together with the arc ($k\to \sigma(k)$), contribute 
\[\cros_k^-(\sigma)=j-i,\]
while the remaining $i-1$ arcs that do not intersect the arc ($k\to\sigma(k)$) contribute
\[\nest_{k}^-(\sigma)=i-1,\]
which proves~(4).
Eq.~\eqref{def:height-FZ} follows by induction on $k$
and can be proved analogously to Lemma~\ref{lem:height}; we omit the details.
\end{proof}

\begin{lemma}\label{lem:dep}
For $\sigma\in\mathfrak{S}_n$, let $\Psi_{FZ}(\sigma):=(w,p)\in\mathcal{L}_n^*$  be the corresponding 3-restricted Laguerre history.
For each $i\in[n]$, denote by $h_i$ the height of the $i$th step of $w$. Then 
    \begin{equation}\label{equ:depth}
        \depth(\sigma)=\area(w). 
    \end{equation} 
\end{lemma}

\begin{proof}
     By Definition \eqref{def:depth} and \eqref{exc-top-bottom}, we see that
   \begin{equation}\label{depth-cpk-cval}      \depth(\sigma)=\mathrm{Etop}(\sigma)-\mathrm{Ebot}(\sigma)=\Cpk(\sigma)-\Cval(\sigma).
   \end{equation}
Translating~\eqref{depth-cpk-cval} through the $\Psi_{FZ}$ yields
\begin{equation}\label{dep2:S-N}   \depth(\sigma)=\sum_{w_i=\mathrm{S}}i-\sum_{w_j=\mathrm{N}}j.
\end{equation}
 By the same argument in the proof of  Lemma~\ref{lem:dis} we obtain \eqref{equ:depth}.
\end{proof}

\begin{example}      
If $\sigma = 8\, 3\, 1\, 10\, 7\, 5\, 6\, 4\,9\, 11\,2\in\S_{11}$ and the corresponding Motzkin path is $w=\mathrm{NNSNN\red{E_1}\mathrm{SS\green{E_3}\blue{E_2}S}}$, see Fig~\ref{fig2:motzkin}, then by Definition~\eqref{def:depth}, we have
$$\depth(\sigma)=(8-1)+(3-2)+(10-4)+(7-5)+(11-10)=17.$$
Moreover, we have 
\begin{align*}
    \area(w)=\sum_{i=1}^{11}h_i
    =5\cdot1+3\cdot 2+ 2\cdot 3=17.
\end{align*}
\end{example}

\subsection{Proof of Theorem~\ref{thm1} and Theorem~\ref{thm2}}\label{sec:Laguerre}

\begin{lemma}\label{lem:lmax-rmax}
Given a permutation $\sigma\in\S_{n+1}$ and $x\in[n]$, letter $x$ is a   
\begin{itemize}
    \item  left-to-right maximum of $\sigma$, if and only if $p_x$ is equal to $0$ and the $x$th step is $\mathrm{S}$ or $\blue{\mathrm{E_2}}$ in the Laguerre history $\Psi_{FV}(\sigma)$;
    \item  right-to-left maximum of $\sigma$, if and only if $p_x$ is equal to $h_x$ and the $x$th step is $\mathrm{S}$ or $\red{\mathrm{E_1}}$ in the Laguerre history $\Psi_{FV}(\sigma)$.
\end{itemize}
\end{lemma}
\begin{proof}It is easy to see that
if $x$ is a left-to-right (resp. right-to-left) maximum of $\sigma$, then $x$ is a peak or a double ascent (resp. a peak or a double descent) of $\sigma$.
    By the construction of $\Psi_{FV}$,  the $x$th step of $\Psi_{FV}(\sigma)$ is given by $\mathrm{S}$ or $\blue{\mathrm{E_2}}$   (resp. $\mathrm{S}$ or $\red{\mathrm{E_1}}$) if $x$ is a left-to-right (resp. right-to-left) maximum of $\sigma$. If $x$ is a left-to-right maximum of $\sigma$, i.e., there is no element to the left of $x$ that is greater than $x$, which means $p_x=(31\!-\!2)_x=0$. If $x$ is a right-to-left maximum of $\sigma$, i.e., there is no element to the right of $x$ that is greater than $x$, which is $(2\!-\!13)_x=0$. Then by Lemma~\ref{lem:height}, we  derive that $(31\!-\!2)_x=h_x$.  
\end{proof}

\begin{proof}[\textbf{Proof of Theorem~\ref{thm1}}]For $\sigma\in\S_{n+1}$,  let    $\Psi_{FV}(\sigma) =(w,p):= (w_1\ldots w_n,(p_1,\ldots, p_n)) \in \mathcal{L}_{n}$ be the corresponding Laguerre history. Recall \eqref{weight-functions} that 
\begin{align}
\wt(\sigma)=&(u_1u_2)^{{\val}(\sigma)}u_3^{\da(\sigma)}u_4^{\dd(\sigma)}
\alpha^{{\lmax}(\sigma)-1}{\beta}^{{\rmax}(\sigma)-1}\notag\\
&\times f^{\mathrm{lmaxpk(\sigma)-1}}g^{\bmax(\s)-1} t^{\lrmaxda(\s)+\rlmaxdd(\s)}.
\end{align}
Combining  Lemmas~\ref{lem:height}--\ref{lem:dis} and \ref{lem:lmax-rmax}, 
we have
   \begin{align}
\wt(\sigma)&\,h^{\dis(\sigma)}p^{\ldes(\sigma)}q^{\rasc(\sigma)}\notag\\
&=
\prod_{\substack{w_i=\mathrm{N}}} \!\bigl(u_1h^{h_i}p^{p_i}q^{h_i-p_i}\bigr)
  \prod_{\substack{w_i=\mathrm{S}}} \!\bigl(u_2h^{h_i}(\alpha f)^{\chi(p_i=0)}(\beta g)^{\chi(p_i=h_i)}p^{p_i}q^{h_i-p_i}\bigr) \notag\\[3pt]
&\quad\times
  \prod_{\substack{w_i=\textcolor{red}{\mathrm{E_1}}}} \!\bigl(u_4h^{h_i}\,(\beta t)^{\chi(p_i=h_i)}p^{p_i}q^{h_i-p_i}\bigr)
  \prod_{\substack{w_i=\textcolor{blue}{\mathrm{E_2}}}} \!\bigl(u_3h^{h_i}\,(\alpha t)^{\chi(p_i=0)}p^{p_i}q^{h_i-p_i}\bigr). 
\end{align}
 Summing over $p_i\in \{0,\ldots,h_i\}$ for $i\in[n]$  we obtain
    \begin{align}
A_n(\mathbf{u},f,g,h,t\,|\,&\alpha,\beta\,;\, p,q)\nonumber\\
&=\sum_{w\in \mathcal{M}^2_n}
  \prod_{\substack{w_i=\mathrm{N}}} u_1h^{h_i}[\,h_i+1\,]_{p,q}
  \prod_{\substack{w_i=\mathrm{S}}} u_2h^{h_i}\bigl[\,h_i+1;\beta g, \alpha f\,\bigr]_{p,q} \notag\\[3pt]
&\quad\times
  \prod_{\substack{w_i=\textcolor{red}{\mathrm{E_1}}}} u_4h^{h_i}[\,h_i+1;\beta  t,1\,]_{p,q}
  \prod_{\substack{w_i=\textcolor{blue}{\mathrm{E_2}}}} u_3h^{h_i}\bigl[\,h_i+1;1,\alpha t\,\bigr]_{p,q},
\end{align}
where $\mathcal{M}^2_n$ is the set of  Motzkin paths of length $n$ with two colors in step E. This completes the proof by an application of Flajolet’s fundamental Lemma~\ref{lem:fla}.
\end{proof}


    For $\sigma=\sigma_1\dots\sigma_n\in\S_n$,  we say that an integer $i\in[n]$ is a 
\begin{itemize}
    \item left-to-right maximum index  of $\sigma$, if $\sigma_j<\sigma_i$ for all $j<i$;
    \item right-to-left minimum index  of $\sigma$, if $\sigma_j>\sigma_i$ for all $j>i$.
\end{itemize}
Since $\sigma$ is a permutation, the map $i\mapsto\sigma(i)=\sigma_i$ is a bijection between positions and values. Consequently, counting left-to-right maxima (resp.\ right-to-left minima) by record positions or by record values
gives the same cardinality. In particular,
\begin{align*}
    \lmax(\sigma)&=\#\{i\in[n]\,:\sigma_j<\sigma_i \text{ for all }j<i\,\};\\
    \rlmin(\sigma)&=\#\{i\in[n]\,:\sigma_j>\sigma_i \text{ for all }j>i\,\}.
\end{align*}
The  following result provides a characterization of $\rlmin$ and $\lmax$ through the bijection $\Psi_{FZ}$.
\begin{lemma}\label{lem:FZ-lmax-rmin}
Given a permutation $\sigma\in\S_n$ and $x\in[n]$, an integer  $x$ is a
    \begin{itemize}
    \item  left-to-right maximum index of $\sigma$ if and only if $p_x$ is equal to $0$ and the $x$th step is $\mathrm{S}$, $\blue{\mathrm{E_2}}$, or $\green{\mathrm{E_3}}$ in the 3-restricted Laguerre history $\Psi_{FZ}(\sigma)$;
    \item right-to-left minimum index of $\sigma$ if and only if $p_x$ is equal to $0$ and the $x$th step is $\mathrm{N}$, $\red{\mathrm{E_1}}$, or $\green{\mathrm{E_3}}$ in the 3-restricted Laguerre history $\Psi_{FZ}(\sigma)$.
\end{itemize}
\end{lemma}
\begin{proof}
One can readily show that if $x$ is a left-to-right maximum  (resp. right-to-left minimum) index of $\sigma$, then $\sigma(x)\ge x$ (resp. $\sigma(x)\le x$). By the construction of $\Psi_{FZ}$, the $x$th step of $\Psi_{FZ}(\sigma)$ is given by $\mathrm{S}$, $\blue{\mathrm{E_2}}$ or $\green{\mathrm{E_3}}$ (resp. $\mathrm{N}$, $\red{\mathrm{E_1}}$ or $\green{\mathrm{E_3}}$) if $x$ is a left-to-right maximum  (resp. right-to-left minimum) index of $\sigma$. If $x$ is a left-to-right maximum index of $\sigma$, then there is no upper arc above the arc ($x\to\sigma(x)$) because $(\sigma(i)<\sigma(x))$ for all $i<x$; and  if $x$ is a right-to-left minimum index of $\sigma$, then there is no lower arc below the arc ($x\to\sigma(x)$) because $(\sigma(i)>\sigma(x))$ for all $i>x$.
\end{proof}

\begin{lemma}[\cite{MV94}]\label{lem:inv}
    Let $\sigma\in\S_n$ and $\Psi_{FZ}(\sigma)=(w,p)\in\mathcal{L}_n^*$. Then
    \begin{align}
        \inv(\sigma)=\,\mathrm{area}(w)+\nest(\sigma)=\,\sum_{i=1}^n(h_i+p_i),
        \end{align}   
        where $h_i$ denotes the height of the $i$th step of $w$ for $i\in[n]$.
\end{lemma}

  \begin{proof}[\textbf{Proof of Theorem~\ref{thm2}}]
For $\sigma\in\S_n$, let $\Psi_{FZ}(\sigma):=(w_1\ldots w_n,(p_1,\ldots, p_n)) \in\L_n^*$ be the corresponding 3-restricted Laguerre history. Define the weight function
\begin{equation}\label{weight-c}
    \wt^c(\sigma):=(u_1u_2)^{\cval(\sigma)}u_3^{\cda(\sigma)}u_4^{\cdd(\sigma)}t^{\fix(\sigma)}s^{\inv(\sigma)} h^{\depth(\sigma)}\alpha^{\lmax(\sigma)}\beta^{\rlmin(\sigma)}p^{\cros(\sigma)}q^{\nest(\sigma)}.
\end{equation}
Combining Lemmas~\ref{lem:dep}, \ref{lem:FZ-lmax-rmin}, and~\ref{lem:inv}, we have
\begin{align}
\wt^c(\sigma)&=\prod_{\substack{w_i=\mathrm{N}}}\!\bigl(u_1\beta^{\chi(p_i=0)}p^{h_i-p_i}q^{p_i}s^{h_i+p_i}h^{h_i}\bigr)
   \prod_{\substack{w_i=\mathrm{S}}}\!\bigl(u_2\alpha^{\chi(p_i=0)}p^{h_i-p_i-1}q^{p_i}s^{h_i+p_i}h^{h_i}\bigr) \notag\\[2pt]
&\quad\times  
   \prod_{\substack{w_i=\textcolor{red}{\mathrm{E_1}}}}\!\bigl(u_4\beta^{\chi(p_i=0)}p^{h_i-p_i-1}q^{p_i}s^{h_i+p_i}h^{h_i}\bigr)
   \prod_{\substack{w_i=\textcolor{blue}{\mathrm{E_2}}}}\!\bigl(u_3\alpha^{\chi(p_i=0)}p^{h_i-p_i}q^{p_i}s^{h_i+p_i}h^{h_i}\bigr)\nonumber\\
   &\qquad\times \prod_{\substack{w_i=\textcolor{green}{\mathrm{E_3}}}}\!\bigl(t (\alpha\beta)^{\chi(h_i=0)}q^{h_i}s^{2h_i}h^{h_i}\bigr).
\end{align}
By Definition \ref{3-Laguerre}, summing over $p_i\in\{0,1,\dots,h_i\}$ if $w_i\in\{\mathrm{N}, \blue{\mathrm{E_2}}, \green{\mathrm{E_3}}\}$ and 
$p_i\in\{0,1,\dots,h_i-1\}$ if $w_i\in\{\mathrm{S},  \red{\mathrm{E_1}}\}$ for $i\in[n]$ we have
\begin{align*}
C_{n}(\mathbf{u},&t,s,h\,|\,\alpha,\beta\,;\, p,q)
=\sum_{w\in \mathcal{M}_n^3}
   \prod_{\substack{w_i=\mathrm{N}}} u_1 (sh)^{h_i}[h_i+1;\beta,1]_{p,qs}
   \prod_{\substack{w_i=\mathrm{S}}} u_2(sh)^{h_i}[h_i;\alpha,1]_{p,qs} \notag\\[4pt]
&\quad\times
   \prod_{\substack{w_i=\textcolor{red}{\mathrm{E_1}}}} u_4(sh)^{h_i}[h_i;\beta,1]_{p,qs}
   \prod_{\substack{w_i=\textcolor{blue}{\mathrm{E_2}}}} u_3p (sh)^{h_i}[h_i;\alpha,1]_{p,qs}\prod_{\substack{w_i=\textcolor{green}{\mathrm{E_3}}}} t(\alpha\beta)^{\chi(h_i=0)}(qs^2h)^{h_i},
\end{align*}
    where $\mathcal{M}_n^3$ denotes the set of all $3$-Motzkin paths of length $n$, that is, Motzkin paths whose horizontal step E is endowed with three possible colors: $\textcolor{red}{\mathrm{E_1}}$, $\textcolor{blue}{\mathrm{E_2}}$, and $\textcolor{green}{\mathrm{E_3}}$.  By applying Flajolet’s fundamental Lemma~\ref{lem:fla}, the proof is complete.
    \end{proof}



\section{Two bijections  over permutations}\label{sec:SZ-C}

The bijections $\Phi_{SZ}$ and $\Phi_{C}$ actually arise from the composition of a restricted Françon--Viennot bijection and a simplified  Foata--Zeilberger bijection. First, we need a  variant of  the bijection $\Psi_{FV}$ (resp.  $\Psi_{FZ}$)
defined in \eqref{def:corresponding} (resp. \eqref{equ:FZ-corresponding}) as follows.

\begin{definition}[2-restricted Laguerre history]
    A 2-restricted Laguerre history of length $n$ is a couple  $(w, p)$, 
    where $w=w_1\cdots w_n$ is a  2-Motzkin path of length $n$ starting at $(0,0)$ and ending at $(n,0)$, and 
$p=(p_1, \ldots, p_n)$ is an integer sequence satisfying $1\le i\le n$,
\[
0 \leq p_i \leq 
\begin{cases}
h_i,   & \text{if } w_i \in \{\mathrm{N,\blue{\mathrm{E_2}}}\};\\[4pt]
h_i-1, & \text{if } w_i \in\{ \mathrm{S,\red{E_1}}\},
\end{cases}
\]
 where $h_i$ is the height of the $i$th step of $w$.   Denote by $\widetilde{\mathfrak{L}}_n^*$  the set of 2-restricted Laguerre histories of length $n$. 
\end{definition}

We describe the  restricted version of  Françon-Viennot's bijection $\widetilde{\Psi}_{FV}:\S_{n}\to \widetilde{\L}_n^*$ as follows: For a permutation $\sigma\in\S_{n}$ with a boundary condition $\sigma_0=0$ and $\sigma_{n+1}=n+1$,  the 2-restricted Laguerre history  $\widetilde{\Psi}_{FV}(\sigma) := (w_1\ldots w_n,(p_1,\ldots, p_n)) \in \widetilde{\mathcal{L}}_{n}^*$ is defined by
\begin{equation}
    w_i =
\begin{cases}
\mathrm{N}\;(\textrm{resp.} \;\mathrm{S} ) & \text{if } i \textrm{ is a valley (resp. peak) of }\sigma; \\[2pt]
\mathrm{\red{E_1}} \; (\textrm{resp.}\; \blue{\mathrm{E_2}}) & \text{if } i \text{ is a double descent (resp. double ascent) of }\sigma, 
\end{cases}
\end{equation}
and $p_i = (31\!-\!2)_i(\sigma)$ for each $i \in [n]$.

\begin{lemma}\cite{Cor07}
Given a permutation $\sigma\in\S_n$ and the corresponding 2-restricted Laguerre history $\widetilde{\Psi}_{FV}(\sigma)=(w,p)\in\widetilde{\mathcal{L}}_n^*$, for any $x\in[n]$, we have 
    \begin{equation}\label{equ:h_X-restricted-1}
(31\!-\!2)_x(\sigma)+(2\!-\!31)_x(\sigma)=
\begin{cases}
h_x,   & \text{if } w_x \in \{\mathrm{N,\blue{\mathrm{E_2}}}\}; \\[4pt]
h_x-1, & \text{if } w_x \in\{ \mathrm{S,\red{E_1}}\}.
\end{cases}
\end{equation}
\end{lemma}
\begin{proof}
    Let $\sigma\in\S_n$ and $\sigma'=\sigma\star\sigma_{n+1}\in\S_{n+1}$ with $\sigma_{n+1}=n+1$, where $\star$ denotes the concatenation operator of two words. 
Suppose that for $x\in[n]$,  the $x$-decomposition of $\sigma'$ is 
 \[\sigma' = u_1 v_1 \cdots u_k v_k, \quad k \geq 1,\]  
 where $x$ lies inside $v_j$, see \eqref{x-decomposition}.
    It is straightforward to verify the following:
    \begin{itemize}
        \item word $v_k$ has $n+1$ as its last element, possibly together with other entries preceding $n$.
        \item If $x$ is a peak or a double descent of $\sigma'$, then $x$ cannot lie in $v_k$.       
    \end{itemize}
When $x$ is a valley or a double ascent, the word $v_j$  must contain at least one element larger than $x$ on the right side. Consequently, each contribution to $(2\!-\!13)_x(\sigma')$  corresponds to a contribution to $(2\!-\!31)_x(\sigma)$.
    When $x$ is a peak or a double descent, and the word $v_j$ contains $x$ as the last element. In these two cases, one contribution to $(2\!-\!31)_x(\sigma)$ is lost. See the corresponding illustration below.
    
\begin{center}
    \begin{tikzpicture}[scale=0.65]
\draw[very thick] (-3.4,1) -- (15,1);
\draw[very thick] (-3.4,0)   -- (15,0);
\draw[red, very thick](-3.1,0) .. controls (-2.6,-1) .. (-2.1,0);
\draw[blue, very thick](-1.9,1) .. controls (-1.4,2) .. (-0.9,1);
\draw[red, very thick](-0.7,0) .. controls (-0.2,-1) .. (0.3,0);
\draw[blue, very thick](0.5,1) .. controls (1.0,2) .. (1.5,1);
\draw[red, very thick](1.7,0) .. controls (2.2,-1) .. (2.7,0);

\draw[blue, very thick](3.6,1) .. controls (4.1,2) .. (4.6,1);
\draw[red,  very thick] (4.8,0) .. controls (5.3,-1) .. (5.8,0);
\draw[blue, very thick] (6.0,1) .. controls (6.5, 2) .. (7.0,1);
\draw[red,  very thick] (7.2,0) .. controls (7.7,-1) .. (8.2,0);
\draw[blue, very thick] (8.4,1) .. controls (8.9, 2) .. (9.4,1);

\draw[red,  very thick] (10.1,0) .. controls (10.6,-1) .. (11.1,0);
\draw[blue, very thick] (11.3,1) .. controls (11.8,2) .. (12.3,1);
\draw[red,  very thick] (12.5,0) .. controls (13.0,-1) .. (13.5,0);
\draw[blue, very thick] (13.7,1) .. controls (14.2,2) .. (14.7,1);

\node at (-2.6,-1.1) {$u_1$};
\node at (-0.2,-1.1) {$u_2$};
\node at (2.2,-1.1) {$u_3$};
\node at (5.3,-1.1)   {$u_j$};
\node at (7.7,-1.1)   {$u_{j+1}$};
\node at (10.6,-1.1)  {$u_{k-1}$};
\node at (13.0,-1.1)  {$u_{k}$};

\node at (6.5,0.5)   {$x$};
\node at (3.1,0.5) {$\cdots$};
\node at (10,0.5) {$\cdots$};
\node at (-1.4,2.1) {$v_1$};
\node at (1,2.1) {$v_2$};
\node at (4.1,2.1) {$v_{j-1}$};
\node at (6.5,2.1)   {$v_j$};
\node at (8.9,2.1)   {$v_{j+1}$};
\node at (11.8,2.1)  {$v_{k-1}$};
\node at (14.2,2.1)  {$v_k$};

\draw[thick, decorate, decoration={snake, amplitude=1mm, segment length=5mm}]
      (-0.9,1) -- (-0.7,0);
\draw[thick, decorate, decoration={snake, amplitude=1mm, segment length=5mm}]
      (1.5,1) -- (1.7,0);
\draw[thick, decorate, decoration={snake, amplitude=1mm, segment length=5mm}]
      (4.6,1) -- (4.8,0);


\draw[dashed, thick] (8.2,0) -- (8.4,1);
\draw[dashed, thick] (11.1,0) -- (11.3,1);
\draw[dashed, thick] (13.5,0) -- (13.7,1);

\draw[dashed, thick] (7.0,1) -- (7.2,0);
\draw[dashed, thick] (12.3,1) -- (12.5,0);
\node at (7.7,0.5) {$\leftrightarrow$
};
\node at (13,0.5) {$\leftrightarrow$};
\end{tikzpicture}
\end{center}
Thus, for any $x\in[n]$, we have
\begin{subequations}\label{equ:1-2}
    \begin{equation}
        (31\!-\!2)_x(\sigma)=(31\!-\!2)_x(\sigma'), 
    \end{equation}
and 
    \begin{equation}
    (2\!-\!31)_x(\sigma)=\begin{cases}
        (2\!-\!13)_x(\sigma')\; &\text{if } x \text{ is a valley or a double ascent;}\\[5pt]
        (2\!-\!13)_x(\sigma')-1\; &\text{if } x \text{ is a peak or a double descent.} 
        \end{cases}
    \end{equation}
\end{subequations}
Combining Eqs.~\eqref{equ:1-2} and Lemma~\ref{lem:height}, we derive~\eqref{equ:h_X-restricted-1}. 
\end{proof}


We describe  a simplified version of Foata–Zeilberger’s bijection $\widetilde{\Psi}_{FZ}$ as follows: For a permutation $\sigma\in\S_{n}$,
  the 2-restricted Laguerre history $\widetilde{\Psi}_{FZ}(\sigma):=(w_1\dots w_n,(p_1,\dots,p_n))\in\widetilde{\mathcal{L}}_n^*$ is defined by
\begin{equation}\label{equ:corresponding-FZ*}
   w_i=
\begin{cases} 
\mathrm{N} & \text{if $i$ is a cycle valley of $\sigma$;}\\
\mathrm{S} & \text{if $i$ is a cycle peak of $\sigma$;}\\
\textcolor{red}{\mathrm{E_1}} & \text{if $i$ is a cycle double descent of $\sigma$;}\\
\textcolor{blue}{\mathrm{E_2}} & \text{if $i$ is a cycle double ascent or a fixed point of $\sigma$,}
\end{cases} 
\end{equation}
and with $p_i :=\nest_i(\sigma)$ for each $i\in[n]$.

The following property follows directly from Lemma~\ref{lem:height-FZ} and \eqref{equ:corresponding-FZ*}.
\begin{lemma}[\cite{Cor07}]
Given a permutation $\sigma\in\S_n$, and the corresponding 2-restricted Laguerre history $\widetilde{\Psi}_{FZ}(\sigma)=(w,p)\in\widetilde{\L}_n^*$ for $\sigma\in\S_n$. For any $k\in[n]$,
    if $h_k$ is the height of the $k$th step of $w$, then
\begin{equation}\label{equ:h_X-restricted-2}
\cros_k(\sigma)+\nest_k(\sigma)=\begin{cases}
        h_k, & \text{if } w_x \in \{\mathrm{N,\blue{\mathrm{E_2}}}\};\\[4pt]
        h_k-1,&\text{if } w_x \in\{ \mathrm{S,\red{E_1}}\}.
    \end{cases}
\end{equation}
In particular, if $k\in[n]$ is a fixed point of $\sigma$, we have
\begin{equation}\label{fix:height}
    \nest_k(\sigma)=h_k.
\end{equation}
\end{lemma}

For a 2-restricted Laguerre history $(w,(p_1,\dots,p_n))\in\widetilde{\L}^*_n$, we define the complement operator $\rho$ by
\begin{equation}\label{def:rho}
    \rho\bigl((w,(p_1,\dots,p_n))\bigr)=(w,(p'_1,\dots,p'_n)),
\end{equation}
where the pairs $p_i$ and $p'_i$ are related by
\begin{equation}\label{def:height-FZ-2}
    p_i+p'_i=\begin{cases}
        h_k, & \text{if } w_x \in \{\mathrm{N,\blue{\mathrm{E_2}}}\};\\[4pt]
        h_k-1,&\text{if } w_x \in\{ \mathrm{S,\red{E_1}}\}.
    \end{cases}
\end{equation}

\subsection{Bijection \texorpdfstring{$\Phi_C$}{Phi\_C}}\label{section: corteel's bijection}
We recall Corteel's bijection $\Phi_C:=\widetilde{\Psi}_{FZ}^{-1}\circ\widetilde{\Psi}_{FV}$ in \cite[Sec.~5]{Cor07}, which is constructed  in such a way that
if $\tau = \Phi_{C}(\sigma)$, then
\begin{equation}
    (31\!-\!2)_k(\sigma) = \nest_k(\tau) \quad \forall k=1,\ldots,n.
\end{equation} 
For $k\in[n]$, the pattern $(2\!-\!31)_k(\sigma)$ (resp. $(31\!-\!2)_k(\sigma)$) indexed by $k$ is called the {\em right embracing number} (resp. {\em  left embracing number}) of $k$ in $\sigma$.
Given a permutation $\sigma$, one constructs two biwords,
$\left(\genfrac{}{}{0pt}{}{f}{f'}\right)$ and $\left(\genfrac{}{}{0pt}{}{g}{g'}\right)$, 
and then forms the biword
\[
\tau =
\left(
  \begin{array}{cc}
    f & g \\
    f' & g'
  \end{array}
\right)
\]
by concatenating $f$ and $g$, and $f'$ and $g'$, respectively. These four words are defined as follows:
\begin{itemize}
    \item The word $f$ (resp. $g$) is defined as the subword of nondescent (resp. descent) tops in $\sigma$,
ordered increasingly; 
    \item The word $f'$ (resp. $g'$) is the permutation of nondescent (resp. descent) bottoms in $\sigma$ obtained from the left-to-right (resp. right-to-left) inversion table (see Definition \ref{inversion-table}) whose  entries are the left embracing numbers of $\sigma$.
\end{itemize}
See the following illustration:
\begin{align}\label{construction2}
    &\left(\genfrac{}{}{0pt}{}{\text{Nondescent tops (with labels)}}{\text{Nondescent bottoms}}\right) \xRightarrow[\text{with entries }(31-2)_i]{\text{left-to-right inversion table}} \left(\genfrac{}{}{0pt}{}{f}{f'}\right);\\[6pt]
    &\left(\genfrac{}{}{0pt}{}{\text{Descent tops (with labels)}}{\text{Descent bottoms}}\right) \xRightarrow[\text{with entries }(31-2)_i]{\text{right-to-left inversion table}}\left(\genfrac{}{}{0pt}{}{g}{g'}\right).
\end{align}

 By rearranging the columns of $\tau$ so that the top row is in increasing order, we obtain the permutation $\Phi_C(\sigma)$ as the bottom row of the  rearranged biword.

 \begin{example}
    Let $\sigma=4\,1\,2\,7\,9\,6\,5\,8\,3\in\S_9$ with the left embracing numbers 
    $$(\lbr_1,\dots,\lbr_8)\,\sigma=(0,1,1,0,0,0,0,1,0).$$
Then
\[
\left(\genfrac{}{}{0pt}{}{f}{f'}\right)
= \left(
\genfrac{}{}{0pt}{}{1_0}{7}\,
\genfrac{}{}{0pt}{}{2_1}{2}\,
\genfrac{}{}{0pt}{}{3_1}{4}\,
\genfrac{}{}{0pt}{}{5_0}{8}\,
\genfrac{}{}{0pt}{}{7_0}{9}
\right),\quad
\left(\genfrac{}{}{0pt}{}{g}{g'}\right)
= \left(
\genfrac{}{}{0pt}{}{4_0}{1}\,
\genfrac{}{}{0pt}{}{6_0}{3}\,
\genfrac{}{}{0pt}{}{8_1}{6}\,
\genfrac{}{}{0pt}{}{9_0}{5}
\right),
\]

\[
\tau'
= \left( 
\genfrac{}{}{0pt}{}{f}{f'} \,
\genfrac{}{}{0pt}{}{g}{g'} 
\right)
= \left(
\genfrac{}{}{0pt}{}{1}{7}\,
\genfrac{}{}{0pt}{}{2}{2}\,
\genfrac{}{}{0pt}{}{3}{4}\,
\genfrac{}{}{0pt}{}{5}{8}\,
\genfrac{}{}{0pt}{}{7}{9}\,
\genfrac{}{}{0pt}{}{4}{1}\,
\genfrac{}{}{0pt}{}{6}{3}\,
\genfrac{}{}{0pt}{}{8}{6}\,
\genfrac{}{}{0pt}{}{9}{5}
\right)
\to
\left(
\genfrac{}{}{0pt}{}{1}{7}\,
\genfrac{}{}{0pt}{}{2}{2}\,
\genfrac{}{}{0pt}{}{3}{4}\,
\genfrac{}{}{0pt}{}{4}{1}\,
\genfrac{}{}{0pt}{}{5}{8}\,
\genfrac{}{}{0pt}{}{6}{3}\,
\genfrac{}{}{0pt}{}{7}{9}\,
\genfrac{}{}{0pt}{}{8}{6}\,
\genfrac{}{}{0pt}{}{9}{5}
\right).
\]

and thus $\Phi_C(\sigma)=\tau = 7\,2\,4\,1\,8\,3\,9\,6\,5$.  
We verify that 
 \begin{itemize}
     \item $\ndes(\sigma)=|\{1,2,3,5,7\}|=5$, $\mathrm{pmin }(\sigma)=|\{2\}|=1$, $(31\!-\!2)(\sigma)=|\{41\!-\!2,41\!-\!3,96\!-\!8\}|=3$, $(2\!-\!31)(\sigma)=|\{7\!-\!96,\;4\!-\!83,\;7\!-\!83,\;6\!-\!83,\;5\!-\!83\}|=5$, $\mathrm{MADL}(\sigma)=15$, and 
    $\NDdif(\sigma)=\NDbot(\sigma)-\NDtop(\sigma)=30-18=12$;
     \item $\wex(\tau)=|\{1,2,3,5,7\}|=5$, $\fix(\tau)=|\{2\}|=1$, $\cros(\tau)=5$, $\nest(\tau)=3$, $\inv(\tau)=15$ and $\depth(\tau)=(7-1)+(4-3)+(8-5)+(9-7)=12$.
 \end{itemize}
\end{example}

\subsection{Bijection \texorpdfstring{$\Phi_{SZ}$}{Phi\_{SZ}}}
We describe  a variant of the  mapping $\Phi_{SZ}:=\widetilde{\Psi}_{FZ}^{-1}\circ\rho\circ\widetilde{\Psi}_{FV}$ 
in \cite{SZ10} as in the following.

%
Define  $\widetilde{\tau} = \Phi_{SZ}(\sigma)$ for $\sigma\in \S_n$ such that
\begin{equation}
    (2\!-\!31)_k(\sigma) = \nest_k(\widetilde{\tau}) \quad \forall k=1,\ldots,n.
\end{equation}
Given a permutation $\sigma\in \S_n$, we first construct two biwords,
$\left(\genfrac{}{}{0pt}{}{f}{\widetilde{f}'}\right)$ and $\left(\genfrac{}{}{0pt}{}{g}{\widetilde{g}'}\right)$, 
and then form the biword
\[
\widetilde{\tau} =
\left(
  \begin{array}{cc}
    f & g \\
    \widetilde{f}' & \widetilde{g}'
  \end{array}
\right)
\]
by concatenating $f$ and $g$, and $\widetilde{f}'$ and $\widetilde{g}'$, respectively. These four words are defined as follows:
\begin{itemize}
    \item The word $f$ (resp. $g$) is defined as the subword of nondescent (resp. descent) tops in $\sigma$,
ordered increasingly;
\item  The word $\widetilde{f}'$ (resp. $\widetilde{g}'$) is the permutation of nondescent (resp. descent) bottoms in $\sigma$ obtained from the left-to-right (resp. right-to-left) inversion table (see Definition \ref{inversion-table}) whose entries are the right embracing numbers of $\sigma$.

\end{itemize}
See the following illustration:
\begin{align}\label{construction1}
    &\left(\genfrac{}{}{0pt}{}{\text{Nondescent tops (with labels)}}{\text{Nondescent bottoms}}\right) \xRightarrow[\text{with entries }(2-31)_i]{\text{left-to-right inversion table}} \left(\genfrac{}{}{0pt}{}{f}{\widetilde{f}'}\right);\\[6pt]
    &\left(\genfrac{}{}{0pt}{}{\text{Descent tops (with labels)}}{\text{Descent bottoms}}\right) \xRightarrow[\text{with entries }(2-31)_i]{\text{right-to-left inversion table}}\left(\genfrac{}{}{0pt}{}{g}{\widetilde{g}'}\right).
\end{align}
By rearranging the columns of $\widetilde{\tau}$ so that the top row is in increasing
order, we obtain the permutation $\Phi_{SZ}(\sigma)$ as the bottom row of the
rearranged biword.

\begin{example}
Let $\sigma=4\,1\,2\,7\,9\,6\,5\,8\,3\in\S_9$ with the right embracing numbers 
$$(\rbr_1,\dots,\rbr_9)\,\sigma=(0,0,0,1,1,1,2,0,0).$$
Then
\[
\left(\genfrac{}{}{0pt}{}{f}{\widetilde{f}'}\right)
= \left(
\genfrac{}{}{0pt}{}{1_0}{2}\,
\genfrac{}{}{0pt}{}{2_0}{4}\,
\genfrac{}{}{0pt}{}{3_0}{9}\,
\genfrac{}{}{0pt}{}{5_1}{8}\,
\genfrac{}{}{0pt}{}{7_2}{7}
\right)\quad
\left(\genfrac{}{}{0pt}{}{g}{\widetilde{g}'}\right)
= \left(
\genfrac{}{}{0pt}{}{4_1}{3}\,
\genfrac{}{}{0pt}{}{6_1}{5}\,
\genfrac{}{}{0pt}{}{8_0}{1}\,
\genfrac{}{}{0pt}{}{9_0}{6}
\right),
\]

\[
\widetilde{\tau}'
= \left( 
\genfrac{}{}{0pt}{}{f}{\widetilde{f}'} \,
\genfrac{}{}{0pt}{}{g}{\widetilde{g}'} 
\right)
= \left(
\genfrac{}{}{0pt}{}{1}{2}\,
\genfrac{}{}{0pt}{}{2}{4}\,
\genfrac{}{}{0pt}{}{3}{9}\,
\genfrac{}{}{0pt}{}{5}{8}\,
\genfrac{}{}{0pt}{}{7}{7}\,
\genfrac{}{}{0pt}{}{4}{3}\,
\genfrac{}{}{0pt}{}{6}{5}\,
\genfrac{}{}{0pt}{}{8}{1}\,
\genfrac{}{}{0pt}{}{9}{6}
\right)
\to
\left(
\genfrac{}{}{0pt}{}{1}{2}\,
\genfrac{}{}{0pt}{}{2}{4}\,
\genfrac{}{}{0pt}{}{3}{9}\,
\genfrac{}{}{0pt}{}{4}{3}\,
\genfrac{}{}{0pt}{}{5}{8}\,
\genfrac{}{}{0pt}{}{6}{5}\,
\genfrac{}{}{0pt}{}{7}{7}\,
\genfrac{}{}{0pt}{}{8}{1}\,
\genfrac{}{}{0pt}{}{9}{9}
\right).
\]

and thus $\Phi_{SZ}(\sigma)=\widetilde{\tau} = 2\,4\,9\,3\,8\,5\,7\,1\,6$. We verify that \begin{itemize}
    \item $\ndes(\sigma)=|\{1,2,3,5,7\}|=5$, $\fmax(\sigma)=|\{7\}|=1$, $(31\!-\!2)(\sigma)=|\{41\!-\!2,41\!-\!3,96\!-\!8\}|=3$, $(2\!-\!31)(\sigma)=|\{7\!-\!96,\;4\!-\!83,\;7\!-\!83,\;6\!-\!83,\;5\!-\!83\}|=5$, $\MAD(\sigma)=17$, and 
    $\NDdif(\sigma)=\NDbot(\sigma)-\NDtop(\sigma)=30-18=12$;
    \item $\wex(\widetilde{\tau})=|\{1,2,3,5,7\}|=5$, $\fix(\widetilde{\tau})=|\{7\}|=1$, $\cros(\widetilde{\tau})=3$, $\nest(\widetilde{\tau})=5$, $\inv(\widetilde{\tau})=17$ and $\depth(\widetilde{\tau})=(2-1)+(4-2)+(9-3)+(8-5)=12$.
\end{itemize}
\end{example}



\subsection{Proof of Theorem~\ref{thm3}} We prove only \eqref{equ:dis-Corteel}; the proof of \eqref{equ:dis} is analogous and omitted.

Given $\sigma\in \S_n$, let $\tau=\Phi_{C}(\sigma)$. 
It is easily checked that every column $\binom{i}{j}$ in $\binom{f}{f'}$
satisfies $i\le j$ and every column $\binom{i}{j}$ in $\binom{g}{g'}$ satisfies $i> j$. It follows that each letter $\sigma_i$ in $f$ corresponds to a nondescent $\sigma_i$ in $\pi$ and also a weak excedance $\sigma_i$ in $\tau$, that is, for all $1\le i \le n$,
\begin{equation}\label{eq:ndes-wex}
\sigma_i<\sigma_{i+1} \quad (1\leq i\leq n-1)\quad\text{or}\quad i=n\iff \sigma_i \le \tau(\sigma_i).
\end{equation}
Hence, we have
$\ndes(\pi) = \wex(\tau)$.
By definition, $(31\!-\!2)(\sigma)$ is the sum of the left embracing numbers in $\pi$ and the sum of the
inversions in the  words $f'$ and $g'$.
If a pair $(i,j)$ is an inversion in $f'$ (resp. $g'$), then we have $j<i\le \tau(i)<\tau(j)$ (resp. $\tau(j)<\tau(i)<j<i$) that is  a nesting in $\tau$. Thus, we have 
\begin{equation}\label{equ:31-2-nest}
    (31\!-\!2)_k(\sigma) = \nest_k(\tau), \quad \forall k\in [n].
\end{equation}

Combining \eqref{equ:h_X-restricted-1} and \eqref{equ:h_X-restricted-2}, we derive
\begin{equation}
    (2\!-\!31)_k(\sigma) = \cros_k(\tau), \quad  \forall k\in [n].
\end{equation}

By the definition of $\mathrm{pmin}$ (see Definition \ref{Def:fmax-pmin}), if $k$ is a postminimum in $\sigma$, then $k$ is a nondescent bottom. Moreover, since $k$ is also a right-to-left minimum, it follows that $k$ is a nondescent top. Consequently, $k$ belongs to both subwords $f$ and $f'$. One can check that $k$ is a postminimum of $\sigma$, if and only if $p_k=h_k$ and the $k$th step is $\blue{\mathrm{E}_2}$ in the 2-restricted Laguerre history $\widetilde{\Psi}_{FV}(\sigma)$. Then, combining \eqref{equ:31-2-nest} and  \eqref{fix:height}, we obtain $\tau(k)=k$.

As shown in \cite{SZ10}, the inversion number can be characterized  in terms of the crossing and nesting numbers by
\begin{equation}\label{def:inv-cros-nest}
    \inv(\sigma)=\drop(\sigma)+\cros(\sigma)+2\,\nest(\sigma).
\end{equation}

Combining the above discussion with Eqs.~\eqref{def:Madl} and ~\eqref{def:inv-cros-nest}, we see that
\begin{align}
    \mathrm{MADL}(\sigma)&=\des(\sigma)+(2\!-\!31)(\sigma)+2(31\!-\!2)(\sigma)\nonumber\\
    &=\drop(\tau)+\cros(\tau)+2\nest(\tau)=\inv(\tau).
\end{align}
If $i\in [n]$ is a weak-excedance of $\tau$, viz,  $\tau(i)\ge i$, we say that 
$\tau(i)$ is a weak-excedance top and $i$  a weak-excedance bottom. By~\eqref{eq:ndes-wex}, 
every nondescent top (resp.\ nondescent bottom) of $\sigma$ corresponds naturally to  a weak-excedance bottom (resp.\ weak-excedance top) of $\tau$.
Therefore, by Definitions~\eqref{def:depth} and~\eqref{def:nondif}, we have $\NDdif(\sigma)=\depth(\tau)$.  

\qed

\section{Applications  to gamma-positivity}\label{sec:gamma}
\subsection{Proof of Theorem~\ref{thm4} and Theorem~\ref{thm5}}
Consider the following specialization of  \eqref{(p,q)-Eulerian-poly}
\begin{equation}\label{equ:specailization}    
\mathbf{u}=(1,xu_2,u_3,xu_4)\quad\textrm{and}\quad \beta=\alpha,
\end{equation} 
we obtain
\begin{multline}\label{specialization-A}  A_n(\mathbf{u},f,g,h,t\,|\,\alpha,\alpha\,;\,p,q)
    =\sum_{\sigma \in \mathfrak{S}_{n+1}}x^{\des(\sigma)}u_2^{{\val}(\sigma)}u_3^{\da(\sigma)}u_4^{\dd(\sigma)}f^{\mathrm{lmaxpk(\sigma)-1}}g^{\bmax(\s)-1}h^{\dis(\sigma)}\\\times
 t^{\lrmaxda(\s)+\rlmaxdd(\s)}\alpha^{{\lmax}(\sigma)+{\rmax}(\sigma)-2}p^{\ldes(\sigma)}q^{\rasc(\sigma)}.
\end{multline}
For $0\le k\le\lfloor n/2\rfloor$, let $\gamma_{n,k}(x,u_3,u_4,f,g,h\,|\,t,\alpha\,|\,p,q)$ be the coefficient of ${u^k_2}$ in the expansion 
\begin{equation}\label{equ:A-coe-u_2}  A_n(\mathbf{u},f,g,h,t\,|\,\alpha,\alpha\,;\,p,q)=\sum_{k=0}^{\lfloor n/2\rfloor}\gamma_{n,k}(x,u_3,u_4,f,g,h\,|\,t,\alpha\,;\,p,q)u_2^{k}, 
\end{equation}
where $\mathbf{u}=(1,xu_2,u_3,xu_4)$.

Define the two special cases of \eqref{specialization-A}:
    \begin{align} 
A_n(x,u_2,u_3,u_4,f,g,h\,|\,p,q):&=A_n(\mathbf{u},f,g,h,1\,|\,1,1\,;\,p,q);\\
B_n(x,u_2,u_3,u_4,f,g,h\,|\,t,\alpha):&=A_n(\mathbf{u},f,g,h,t\,|\,\alpha,\alpha\,;\,1,1).
\end{align}

\begin{lemma}\label{lem:gamma-expansion}
The following expansion formulas hold:
 \begin{equation}\label{A-gamma-expansion2}  A_n(x,u_2,u_3,u_4,f,g,h\,|\,p,q)=\sum_{k=0}^{\lfloor n/2\rfloor}\gamma_{n,k}^a(f,g,h\,|\,p,q)(xu_2)^{k}(u_3+xu_4)^{n-2k}, 
\end{equation}
and
 \begin{equation}\label{A-gamma-expansion1}  B_n(x,u_2,u_3,u_4,f,g,h\,|\,t,\alpha)=\sum_{k=0}^{\lfloor n/2\rfloor}\gamma_{n,k}^b(f,g,h\,|\,t,\alpha)(xu_2)^{k}(u_3+xu_4)^{n-2k},
\end{equation}
where
\begin{equation}\label{gamma-a}
\gamma_{n,k}^a(f,g,h\,|\,p,q)=\sum_{\sigma\in\S_{n+1,\des=k}^{\mathrm{dd=0}}}f^{\mathrm{lmaxpk(\sigma)-1}}g^{\bmax(\s)-1}h^{\dis(\sigma)}p^{\ldes(\sigma)}q^{\rasc(\sigma)},
\end{equation}
and
\begin{equation}\label{gamma-b}   \gamma_{n,k}^b(f,g,h\,|\,t,\alpha)=\sum_{\sigma\in\S_{n+1,\des=k}^{\mathrm{dd=0}}}f^{\mathrm{lmaxpk(\sigma)-1}}g^{\bmax(\s)-1}h^{\dis(\sigma)}
 t^{\lrmaxda(\s)}\alpha^{{\lmax}(\sigma)+{\rmax}(\sigma)-2}
\end{equation} 
with $\S_{n,\des=k}^{\mathrm{dd=0}}=\{\sigma\in\S_n: \dd(\sigma)=0\text{ and } \des(\sigma)=k\}$.
\end{lemma}
\begin{proof}
    Combining~\eqref{equ:A-coe-u_2}, \eqref{specialization-A} and  Theorem~\ref{thm1} with ~\eqref{equ:specailization}, then 
substituting $u_2\to\frac{u_2(u_3+xu_4)^2}{x}$ and $z\to\frac{z}{(u_3+xu_4)}$,  we obtain
\begin{equation}\label{equ:con;gamma}
    \sum_{n\ge 0}\sum_{k=0}^{\lfloor n/2\rfloor}\frac{\gamma_{n,k}(x,u_3, u_4,f,g,h\,|\,t,\alpha\,;\,p,q)}{x^k(u_3+xu_4)^{n-2k}}u_2^{k}z^n=
   \mathcal{J}\left(z; \frac{b_k}{u_3+xu_4},\, \lambda_k\right),
\end{equation}
where
\begin{align}
b_k&=h^k\bigl( u_3\bigl[\,k+1;\alpha t , \,1\,\bigr]_{p,q}+xu_4\bigl[\,k+1;1,\,\alpha t \,\bigr]_{p,q}\bigr),\label{bk}\\
\lambda_{k+1}&=h^{2k+1}u_2\,[\,k+1\,]_{p,q}\,\bigl[\,k+2;\alpha g,\alpha f \,\bigr]_{p,q} \quad (k\geq 0).\label{lamk}
\end{align}

When $\alpha=t=1$, we see that for $k\ge 0$, 
\begin{equation}
    b_k=h^k(u_3+xu_4)[k+1]_{p,q}.
\end{equation}
Thus, the coefficient of $u_2^kz^{n}$ in the left-hand side of~\eqref{equ:con;gamma} is a polynomial in $f,g,h,p$ and $q$ with coefficients independent of $x, u_3$ and $u_4$. Setting 
$x=u_3=1$ and $u_4=0$ in \eqref{equ:con;gamma}, 
we derive that this coefficient is equal to 
$\gamma_{n,k}(1,1, 0,f,g,h\,|\,1,1\,;\,p,q) =\gamma_{n,k}^a(f,g,h\,|\,p,q)$ by  \eqref{gamma-a}. It follows from  \eqref{equ:con;gamma} that 

\begin{equation}
    \frac{\gamma_{n,k}(x,u_3, u_4,f,g,h\,|\,1,1\,;\,p,q)}{x^k(u_3+xu_4)^{n-2k}}=\gamma_{n,k}^a(f,g,h\,|\,p,q),
\end{equation}
which is equivalent to~\eqref{A-gamma-expansion2}. 

When $p=q=1$, we see that for $k\ge 0$,
\begin{equation}
{b_k}=h^k(k+\alpha t)(u_3+xu_4).
\end{equation}
Thus,  the coefficient of $u_2^kz^{n}$ in the left-hand side of~\eqref{equ:con;gamma} is a polynomial in $f,g,h,t$ and $\alpha$ 
with  coefficients independent of $x, u_3$ and $u_4$. 
Setting $x=u_3=1$ and $u_4=0$ in \eqref{equ:con;gamma}, 
we derive that this coefficient is equal to 
$\gamma_{n,k}(1,1, 0,f,g,h\,|\,t,\alpha\,;\,1,1) =\gamma_{n,k}^b(f,g,h\,|\,t,\alpha)$ by  \eqref{gamma-b}. It follows from  \eqref{equ:con;gamma} that 

\[
\frac{\gamma_{n,k}(x,u_3, u_4,f,g,h\,|\,t,\alpha\,;\,1,1)}{x^k(u_3+xu_4)^{n-2k}}
=\gamma_{n,k}^b(f,g,h\,|\,t,\alpha) ,
\]
 which is equivalent to~\eqref{A-gamma-expansion1}. 
\end{proof}
\begin{proof}[\textbf{Proof of Theorem~\ref{thm4}}]
     Setting $u_2=u_3=u_4=1$ in \eqref{A-gamma-expansion2} of Lemma \ref{lem:gamma-expansion} yields \eqref{gamma-des-special2-intro}.
\end{proof}

\begin{proof}[\textbf{Proof of Theorem~\ref{thm5}}]
Setting $u_2=u_3=u_4=1$ in \eqref{A-gamma-expansion1} of Lemma \ref{lem:gamma-expansion} yields \eqref{gamma-des-special1}.
\end{proof}


\subsection{Proof of Theorem~\ref{thm6}}
We consider the following special case  of  \eqref{generalization-C}
\begin{equation}\label{equ:specailization2}    
\mathbf{u}=(x,u_2,xu_3,u_4)\quad\textrm{and}\quad \beta=\alpha
\end{equation} 
that is,
\begin{multline}\label{specail-case-C}   C_n(\mathbf{u},t,s,h\,|\,\alpha, p,q):=\sum_{\sigma\in\S_n}x^{\exc(\sigma)}u_2^{\cval(\sigma)}u_3^{\cda(\sigma)}u_4^{\cdd(\sigma)}t^{\fix(\sigma)}s^{\inv(\sigma)}\\\times h^{\depth(\sigma)}\alpha^{\rlmin(\sigma)+\lmax(\sigma)}p^{\cros(\sigma)}q^{\nest(\sigma)},
\end{multline}
where $\mathbf{u}=(x,u_2,xu_3,u_4)$.
For $0\le k\le \lfloor(n-d)/2\rfloor$, let $c_{n,k,d}(x,u_3,u_4,s,h\,|\,\alpha,p,q)$ be defined by the expansion
\begin{equation}\label{coef-C}
    C_n(\mathbf{u},t,s,h\,|\,\alpha, p,q)=\sum_{d=0}^{n}t^d\sum_{k=0}^{\lfloor(n-d)/2\rfloor}c_{n,k,d}(x,u_3,u_4,s,h,\alpha\,|\,p,q)u_2^k.
\end{equation}

\begin{lemma}\label{lem:C_n}
Let $\mathbf{u}=(x,u_2,xu_3,u_4)$. We have
\begin{equation}\label{gammaa-expansion-cyclic}
    C_n(\mathbf{u},t,s,h\,|\,\alpha, p,q)=\sum_{d=0}^nt^d\sum_{k=0}^{\lfloor(n-d)/2\rfloor}\gamma_{n,k,d}^c(s,h\,|\,\alpha,p,q)(xu_2)^k(u_4+xu_3p)^{n-2k-d},
\end{equation} 
where 
\begin{equation}\label{gamma1-cyclic}   \gamma_{n,k,d}^c(s,h\,|\,\alpha,p,q)=\sum_{\sigma\in\S_{n,\exc=k}^{\mathrm{cda=0},\fix=d}}s^{\inv(\sigma)}h^{\depth(\sigma)}\alpha^{\rlmin(\sigma)+\lmax(\sigma)}p^{\cros(\sigma)}q^{\nest(\sigma)},
\end{equation} 
where $\S_{n,\exc=k}^{\mathrm{cda=0},\fix=d}=\{\sigma\in\S_n: \cda(\sigma)=0, \fix(\sigma)=d \text{ and } \exc(\sigma)=k\}.$
\end{lemma}
\begin{proof}
    Combining \eqref{specail-case-C}, \eqref{coef-C}, and Theorem~\ref{thm2} with  \eqref{equ:specailization2}, then substituting $u_2\to\frac{u_2(u_4+xu_3p)^2}{x}$, $z\to\frac{z}{(u_4+xu_3p)}$, and  $t\to t(u_4+xu_3p)$,  we obtain
    \begin{equation}\label{equ:gamma-frac}
        \sum_{n\ge 0}\sum_{d=0}^{n}\sum_{k=0}^{\lfloor(n-d)/2\rfloor}\frac{c_{n,k,d}(x,u_3,u_4,s,h\,|\,\alpha,p,q)u_2^k}{x^k(u_4+xu_3p)^{n-2k-d}}u_2^kt^dz^n=\mathcal{J}(z;\mathbf{b},\boldsymbol{\lambda}),
    \end{equation}
    where 
    \begin{align}
        b_k=(sh)^k(tq^ks^k+[k;\alpha,1]_{p,qs}),\quad 
        \lambda_{k}=u_2(sh)^{2k+1}[k;\alpha,1]_{p,qs}^2.
    \end{align}
    Thus, the coefficient of $u_2^kt^dz^n$ in the left-hand side of \eqref{equ:gamma-frac} is a polynomial in $s,h,\alpha,p$ and $q$ with coefficients independent of $x,u_3$ and $u_4$. Setting $x=u_4=1$ and $u_3=0$ in \eqref{equ:gamma-frac}, we derive that this coefficient is equal to $c_{n,k,d}(1,0,1,s,h\,|\,\alpha,p,q)=\gamma_{n,k,d}^c(s,h\,|\,\alpha,p,q)$ by \eqref{gamma1-cyclic}. It follows from~\eqref{equ:gamma-frac} that 
    \begin{equation}
        \frac{c_{n,k,d}(x,u_3,u_4,s,h\,|\,\alpha,p,q)u_2^k}{x^k(u_4+xu_3p)^{n-2k-d}}=\gamma_{n,k,d}^c(s,h\,|\,\alpha,p,q),
    \end{equation}
    which is equal to~\eqref{gammaa-expansion-cyclic}.
\end{proof}
\begin{proof}[\textbf{Proof of Theorem \ref{thm6}}]
    Setting $u_2=u_4=1$, $u_3=1/p$ and  $t=t(1+x)$ (resp. $u_2=u_4=p$, $u_3=1$ and $t=t(1+x)$) in   \eqref{gammaa-expansion-cyclic} of Lemma \ref{lem:C_n} yields Eq. \eqref{equ:C_n-gamma} with $i=1,2$.
\end{proof}


\begin{proposition}For $n\ge 1$, we have
    \begin{equation}\label{equ:derangement}
        \sum_{\sigma\in\mathcal{D}_n}(-1)^{\nest(\sigma)}x^{\exc(\sigma)}h^{\depth(\sigma)}=\sum_{j=1}^{\lfloor n/2\rfloor}\binom{n-j-1}{j-1}h^{n-j}x^{j}(1+x)^{n-2j}.
    \end{equation}
\end{proposition}
    \begin{proof}
Specializing Theorem~\ref{thm2} with $u_2=u_3=u_4=\alpha=\beta=s=p=q=1$, $q=-1$, and $t=0$, 
we have 
\begin{equation}
    b_0=0,\;b_1=(1+x)h,\;\lambda_1=xh, \;\text{ and } \lambda_2=0.
\end{equation}
Thus
\begin{equation}\label{equ;genratingfun}  1+\sum_{n\ge1}\left(\sum_{\sigma\in\mathcal{D}_n}(-1)^{\mathrm{nest}(\sigma)}
x^{\mathrm{exc}(\sigma)}h^{\mathrm{depth}(\sigma)}\right)z^n=\frac{1}{1-\dfrac{\lambda_1 z^2}{1-b_1z}}=\frac{1-(1+x)hz}{1-(1+x)hz-xhz^2}.
\end{equation}
For $n\ge1$, one can check that
\begin{equation}\label{equ:coef}
    [z^n]\frac{1-az}{1-az-bz^2}
=\sum_{j=1}^{\lfloor n/2\rfloor}\binom{n-j-1}{j-1}a^{n-2j}b^{j}.
\end{equation}
Substituting $a=(1+x)h$ and $b=xh$ in~\eqref{equ:coef} and combining with~\eqref{equ;genratingfun} yields \eqref{equ:derangement}.
\end{proof}
\begin{remark} As $\inv=\nest+\depth$,
   Eq.~\eqref{equ:derangement} is equivalent to Theorem~5.4 (or Theorem~1.6) in~\cite{Eu26}.
\end{remark}

\subsection{Combinatorics of \texorpdfstring{$\gamma$}{gamma}-coefficient factorization}

\begin{proposition}\label{prop:divisi}
    For $0\le k\le \lfloor n/2\rfloor$, there exist polynomials 
    $E_{n,k}(h\,|\,p,q)\in\N[h,p,q]$ and $E_{n,k}(f,h\,|\,t,\alpha)\in\N[f,h,t,\alpha]$ such that
    \begin{align}\label{divi-gamma-2}
       \gamma^a_{n,k}(1,1,h\,|\,p,q) &= (p+q)^k E_{n,k}(h\,|\,p,q),\\
    \label{divi-gamma-1}         
         \gamma^b_{n,k}(f,f,h\,|\,t,\alpha)&=2^kE_{n,k}(f,h\,|\,t,\alpha).
    \end{align}
\end{proposition}
\begin{proof}
    For all $k\ge 1$,  as
 \begin{equation}\label{equ:p+q}
     (p+q)\mid[k]_{p,q}[k+1]_{p,q},
 \end{equation}
it follows from \eqref{lamk} that $\lambda_k$ must have a factor $(p+q)$ if  $\alpha=f=g=1$. Combining this and~\eqref{A-gamma-expansion2} with $f=g=1$, we prove \eqref{divi-gamma-2}.

    When $p=q=1$ and $g\to f$, it follows from~\eqref{lamk} that
\begin{equation}\label{gamma-divisibility}
    \lambda_k=h^{2k-1}k((k-1)+2\alpha f)u_2=2\cdot h^{2k-1}(k(k-1)/2+k\alpha f)u_2
\end{equation}
for all $k\ge 1$. Combining \eqref{gamma-divisibility} and \eqref{A-gamma-expansion1} with $g=f$, we prove \eqref{divi-gamma-1}.
\end{proof}
\begin{remark}
     Brändén~\cite{Bra08} conjectured  the  divisibility \eqref{divi-gamma-2} for $h=1$; this conjecture was subsequently confirmed in~\cite{SZ12}. 
\end{remark}

\begin{table}[ht]
\centering
\renewcommand{\arraystretch}{1.25}
\begin{tabular}{c|c|c|c}
\hline
$n$ & $k=0$ & $k=1$ & $k=2$\\
\hline
0 & $1$ &  &  \\
1 & $1$ &  &  \\
2 & $1$ & $h(p+q)$ &  \\
3 & $1$ & $h(p+q)\bigl(h(p+q)+2\bigr)$ &  \\
4 & $1$ & $h(p+q)\Bigl(h^{2}(p+q)^{2}+2h(p+q)+3\Bigr)$
  & $h^{2}(p+q)^{2}\Bigl(h^{2}(p^{2}+pq+q^{2})+1\Bigr)$ \\
\hline
\end{tabular}
\caption{The first values of $\gamma^{a}_{n,k}:=\gamma^{a}_{n,k}(1,1,h\,|\, p,q)$ for $0\le 2k\le n\le 4$.}
\end{table}


\begin{table}[ht]
\centering
\renewcommand{\arraystretch}{1.25}
\begin{tabular}{c|c|c|c}
\hline
$n$ & $k=0$ & $k=1$ & $k=2$\\
\hline
0 & $1$ &  &  \\
1 & $\alpha t$ &  &  \\
2 & $\alpha^{2}t^{2}$ & $2\alpha h f$ &  \\
3 & $\alpha^{3}t^{3}$ & $2\alpha h f\bigl(h+\alpha h t+2\alpha t\bigr)$ &  \\
4 & $\alpha^{4}t^{4}$ &
$2\alpha h f\Bigl(\bigl(h+\alpha ht+\alpha t\bigr)^2+2\alpha^2t^2\Bigr)$ &
$4\alpha f h^{2}\Bigl(h^{2}+\alpha f(2h^{2}+1)\Bigr)$ \\
\hline
\end{tabular}
\caption{The first values of $\gamma^{b}_{n,k}:=\gamma^{b}_{n,k}(f,f,h\,|\,t,\alpha)$ for $0\le 2k\le n\le 4$.}
\end{table}

For the rest of this section, we 
focus on the combinatorial interpretations of the  two quotients in Proposition~\ref{prop:divisi}. 
Following 
Foata-Strehl~\cite{FS74} and 
Hetyei-Reiner~\cite{HR98}, we define the André permutations in terms of \emph{$x$-factorization}. 
\begin{definition}\label{x-frac}
Let $\s=\s_1\s_2\dots\s_n\in\S_n$ and let $x=\s_i$ for $i\in [n]$. The $x$-factorization of $\s$ is given by $\s=u\,\lambda(x)\,x\,\rho(x)\,v$, where
\begin{itemize}
\item[(1)] $\lambda(x)=\s_j\cdots\s_{i-1}$ with $\s_l>x$ for $j\le l\le i-1$, $u=\s_1\cdots\s_{j-1}$, and $\s_{j-1}<x$.
\item[(2)] $\rho(x)=\s_{i+1}\cdots\s_{k}$ with $\s_l>x$ for $i+1\le l\le k$, $v=\s_{k+1}\cdots\s_{n}$, and $\s_{k+1}<x$.
\end{itemize}
Here,  any of the words $u,\lambda(x),\rho(x),v$ may be empty. 
\end{definition}

\begin{example} If $\s=8\,10\,1\,4\,5\,11\,3\,6\,12\,2\,7\,9\in \S_{12}$, then the $3$-factorization of $\s$ is 
$$\underbrace{8\,10\,1}_{u}\,\underbrace{4\,5\,11}_{\lambda(3)}\,3\,\underbrace{6\,12}_{\rho(3)
}\,\underbrace{2\,7\,9}_{v}.$$
\end{example}

\begin{definition}[\cite{HR98}]\label{def:x-factorization-André1-2'}
Let $\s=\s_1\cdots\s_n\in\S_n$. 
Say that $\s$ is an \textbf{André permutation of the first kind (resp. second kind)} 
if $\s$ has no double descents, i.e., $\s_{i-1}>\s_{i}>\s_{i+1}$, and
each $x$-factorisation $u\,\lambda(x)\,x\,\rho(x)\,v$ of $\s$ has property
\begin{itemize}
\item $\lambda(x)=\emptyset$ if $\rho(x)=\emptyset$,
\item $\max(\lambda(x))<\max(\rho(x))$  $($resp. $\min(\rho(x))<\min(\lambda(x))$$)$  if $\rho(x)\neq\emptyset$ and $\lambda(x)\neq\emptyset$.
\end{itemize}
\end{definition} 

Let $\A_n^1$ (resp. $\A_n^2$) be the set of André permutations of the first (resp. second) kind  in $\S_n$. It is known that \cite{FS74} the cardinality of 
$\A_n^1$ (resp. $\A_n^2$)  is the  \emph{Euler number} $E_{n}$, which is the $n$th coefficient in  the Taylor expansion of $\sec(z)+\tan(z)$, namely,
\begin{equation}
    \sum_{n\ge 0}E_n\frac{z^n}{n!}=\sec(z)+\tan(z).
\end{equation}

For convenience, we define the subsets of $\A_{n}^1$ and $\A_n^2$, respectively 
\begin{align}
    \A_{n,\des=k}^1:=\{\sigma\in\A_n^1: \des(\sigma)=k\};\\
    \A_{n,\des=k}^2:=\{\sigma\in\A_n^2: \des(\sigma)=k\}.
\end{align}

\begin{theorem}\label{Thm: generalized PZ} For $n\ge 0$ and $0\le k\le \lfloor n/2\rfloor$, we have 
    \begin{equation}\label{inter-d}
    E_{n,k}(h\,|\,p,q)=\sum_{\sigma\in\mathcal{A}_{n+1,\des=k}^2}h^{\dis(\sigma)}p^{\ldes(\sigma)-\des(\sigma)}q^{\rasc(\sigma)}.
\end{equation}
\end{theorem}

By using a group action on permutations without double descents, Pan and Zeng~\cite{PZ21} 
proved \eqref{divi-gamma-2}  and \eqref{inter-d} when $h=1$.  
We shall extend their proof to the general case in Section~\ref{Appendix-A}.

\begin{table}[t]
\centering
\renewcommand{\arraystretch}{1.25}
\begin{tabular}{c|c|c|c}
\hline
$n$ & $k=0$ & $k=1$ & $k=2$\\
\hline
0 & $1$ &  &  \\
1 & $1$ &  &  \\
2 & $1$ & $h$ &  \\
3 & $1$ & $2h+h^2p+h^2q$ &  \\
4 & $1$ & $3h+2h^2(p+q)+h^3(p+q)^2$ & $h^2+h^4(p^2+pq+q^2)$ \\
\hline
\end{tabular}
\caption{The first values of $E_{n,k}(h\,|\,p,q)$ for $0\le 2k\le n\le 4$.}
\end{table}

 For  $\sigma=\sigma_1\ldots \sigma_n\in \S_n$ with $\sigma_0=\sigma_{n+1}=+\infty$, a letter $\s_i\in[n]$ is said to be a
 \begin{itemize}
    \item \emph{left-to-right minimum} (\textbf{lmin}) if $\sigma_j>\sigma_i$ for every $j<i$;
    \item \emph{left-to-right-minimum-valley} (\textbf{lminval})  if 
 $\s_i$ is a left-to-right minimum and also a valley;
 \item  \emph{right-to-left-minimum-valley} (\textbf{rminval})  if $\s_i$ is a  right-to-left minimum and also a valley;
 \item  right-to-left-minimum-double-ascent (\textbf{rlminda}) if it is a double ascent and also a right-to-left minimum.
 \end{itemize}
 \begin{definition}\label{def:dis-variant}
Let $\sigma=\sigma_1\ldots \sigma_n$ be a permutation in $\S_n$ with $\sigma_0=\sigma_{n+1}=+\infty$.  
Denote by $\Val(\sigma)$ and $\Pk(\sigma)$ the sets of valleys and peaks of $\sigma$, respectively.
We define a variant of $\dis$ by
\begin{equation}
\widetilde{\dis}(\sigma)
:=\sum_{i\in \Pk(\sigma)} i \;-\;\sum_{j\in \Val(\sigma)\setminus\{1\}} j .
\end{equation}
\end{definition}

Let $\sigma = \sigma_1 \cdots \sigma_n \in \mathfrak{S}_n$. The \emph{complement} (resp. the \emph{reverse}) of $\sigma$ is defined by, respectively  
\begin{equation}\label{eq:complement}
    \sigma^c = (n+1-\sigma_1)(n+1-\sigma_2)\cdots(n+1-\sigma_n),
\end{equation}
and
\begin{equation}
    \sigma^r = \sigma_n \sigma_{n-1} \cdots \sigma_1.
\end{equation}
 
\begin{example}
   Let $\sigma = 71385624\in\S_8$, then we have 
\[
    \sigma^c = 28614375 \text{ and } \sigma^r=42658317.
\] 
\end{example}

For convenience, let 
\begin{subequations}
    \begin{align}
    \mathrm{lrmin}(\sigma):&=\lrmin(\sigma)+\rlmin(\sigma)-2;\\ \mathrm{lrmival}(\sigma):&=\lmival(\sigma)+\rmival(\sigma)-2;\\
\mathrm{lrmax}(\sigma):&=\lmax(\sigma)+\rmax(\sigma)-2;\\
    \mathrm{lrmaxpk}(\sigma):&=\pmax(\sigma)+\bmax(\sigma)-2.
\end{align}
\end{subequations}

\begin{lemma}\label{lem:mapping-eta}
    The mapping $\eta:=\sigma\mapsto(\sigma^c)^r$ is a bijection on $\S_{n,\des=k}^{dd=0}$ such that 
    \begin{equation}
        (\des,\lrmaxpk,\dis,\lrmaxda,\lrmax)\,\sigma=(\des,\lrmival,\widetilde{\dis},\rlminda,\mathrm{lrmin})\,\eta(\sigma).
    \end{equation}
\end{lemma}

\begin{table}[t]
    \centering
    \renewcommand{\arraystretch}{1.25}
    \begin{tabular}{c|l|l|l}
    \hline
    $n$&$k=0$&$k=1$&$k=2$\\
    \hline
        0 &  $1$&&\\
        1 & $\alpha t$&&\\
        2 & $\alpha^{2}t^{2}$& $\alpha h f$&\\
        3 & $\alpha^{3}t^{3}$&
$\alpha h f\bigl(h+\alpha h t+2\alpha t\bigr)$&\\
        4 & $\alpha^{4}t^{4}$& $\alpha h f\Bigl(\bigl(h+\alpha ht+\alpha t\bigr)^2+2\alpha^2t^2\Bigr)$&$\alpha h^2 f(h^2+2\alpha h^2 f+\alpha  f)$\\
\hline
    \end{tabular}
    \caption{The first values of $E_{n,k}(f,h,t,\alpha)$ for $0\le 2k \le n\le 4$.}
    \label{tab:placeholder}
\end{table}
Lemma~\ref{lem:mapping-eta} yields 
another combinatorial interpretation for  $\gamma_{n,k}^b(f,f,h\,|\,t,\a)$ defined in \eqref{gamma-b}.
\begin{proposition}For $n\ge 0$ and $0\le k\le \lfloor n/2\rfloor$, we have
    \begin{equation}
\gamma_{n,k}^b(f,f,h\,|\,t,\a)=\sum_{\sigma\in\S_{n+1,\des=k}^{\dd=0}}f^{\lrmival(\sigma)}h^{\widetilde{\dis}(\sigma)}t^{\rlminda(\sigma)}\alpha^{\mathrm{lrmin}(\sigma)}
\end{equation}
\end{proposition}


\begin{theorem}\label{thm-andre-minimum}
For $n\ge 0$ and $0\le k\le \lfloor n/2\rfloor$,
    we have
    \begin{equation}\label{equ:andre-minimum}
        E_{n,k}(f,h\,|\,t,\alpha)=\sum_{\sigma\in\mathcal{A}_{n+1,\des=k}^1}f^{\rmival(\sigma)-1}h^{\dis(\sigma)}t^{\rlminda(\sigma)}\alpha^{\mathrm{rmin}(\sigma)-1}.
    \end{equation}
\end{theorem}

By using a group action on permutations without double descents, Dong et al.~\cite{DLP25}
proved \eqref{divi-gamma-1}  and \eqref{equ:andre-minimum} when $f=h=1$.  
We shall extend their proof to the general case in Section~\ref{Appendix-B}.

We now derive  a second combinatorial interpretation of $E_{n,k}(f,h\,|\,t,\alpha)$ 
from Theorem \ref{thm-andre-minimum}.
\begin{theorem}\label{thm-andre-maximum}
For $n\ge 0$ and $0\le k\le \lfloor n/2\rfloor$, we have 
\begin{equation}\label{def:dici-d}
    E_{n,k}(f,h\,|\,t,\alpha)=\sum_{\substack{\sigma\in\A_{n+1,\des=k}^1}}f^{\pmax(\sigma)-1}h^{\dis(\sigma)}t^{\lrmaxda(\sigma)}\alpha^{\lmax(\sigma)-1}.
\end{equation}
\end{theorem}
\begin{proof}
We define a bijection $\rho:\mathcal {A}_{n+1}^1\to\mathcal{A}_{n+1}^1$ satisfying, for all
 \(\sigma\in\mathcal A_{n+1}^1\) 
 \begin{equation}\label{equ:rho}
     (\des,\lrmival,\dis,\rlminda,\rlmin)\sigma
=
(\des,\lrmaxpk,\dis,\lrmaxda,\lmax)\rho(\sigma).
 \end{equation}

Given $\sigma\in\mathcal{A}_{n+1}^1$, we define $\rho(\sigma)$ as follows. First, we decompose \(\sigma\) into cycles 
according to its right-to-left minima. 
Rewrite each cycle by placing its maximum element in the first position, then
reorder the cycles from left to right in increasing order of their maximum elements.
Erasing the cycle notation and concatenating the reordered cycles yields a
permutation \(\rho(\sigma)\), which is readily verified to be an André permutation of the first kind. It is clear that \(\rho\) is a bijection.

In this cycle decomposition, all descents of $\sigma$ occur within cycles. Therefore, the cyclic rotation that moves the maximum element to the first position in each cycle does not change $\des(\sigma)$. Moreover, \(\rho\) maps right-to-left minima to left-to-right maxima and
right-to-left-minimum double ascents to left-to-right-maximum double ascents. This proves \eqref{equ:rho}.

The desired result now follows from~\eqref{equ:andre-minimum} and the bijection $\rho$. 
\end{proof}
\begin{example}
    Let $\sigma=561748239\in\A_{9}^1$.  According to its right-to-left minima, we decompose $\sigma$ into disjoint cycles as  $$\sigma=(561)(7482)(3)(9).$$ Then we have  
    \[\sigma':=\rho(\sigma)=(3)(615)(8274)(9)=361582749.\]
    Moreover, 
    \begin{itemize}
        \item $\des(\sigma)=|\{2,4,6\}|=3$, $\mathrm{rminda}(\sigma)=|\{3\}|=1$, and $\rlmin(\sigma)=|\{1,2,3,9\}|=4$;
        \item $\des(\sigma')=|\{2,5,7\}|=3$, $\mathrm{lmaxda}(\sigma')=|\{3\}|=1$, and $\lmax(\sigma')=|\{3,6,8,9\}|=4.$
    \end{itemize}
\end{example}


Define the following enumerative polynomials 
\begin{subequations}
    \begin{align}
   E_n(x,h\,|\,p,q):&=\sum_{k=0}^{\lfloor n/2\rfloor}E_{n,k}(h\,|\,p,q)x^k\nonumber\\
   &=\sum_{\sigma\in\A_{n+1}^2}x^{\des(\sigma)}h^{\dis(\sigma)}p^{\ldes(\sigma)-\des(\sigma)}q^{\rasc(\sigma)};\label{andre:poly1}\\
   E_n(x,f,h\,|\,t,\alpha):&=\sum_{k=0}^{\lfloor n/2\rfloor}E_{n,k}(f,h,t,\alpha)x^k\nonumber\\
   &=\sum_{\sigma\in\A_{n+1}^1}x^{\des(\sigma)}f^{\pmax(\sigma)-1}h^{\dis(\sigma)}t^{\lrmaxda(\sigma)}\alpha^{\lmax(\sigma)-1}.\label{andre:poly2}
\end{align}
\end{subequations}

\begin{theorem}We have
\begin{enumerate}
    \item 
    \begin{equation}\label{continued-andre}
  \sum_{n=0}^\infty E_n(x,h\,|\,p,q)z^n
 = \mathcal{J}(z;\mathbf{b},\boldsymbol{\lambda}),
\end{equation}
where $b_k=h^k[k+1]_{p,q}$  and 
$\lambda_{k+1}=xh^{2k+1}[k+1]_{p,q}[k]_{p,q}/(p+q)$ for $k\geq 0$.
\item
 \begin{equation}\label{continued-andre2}
  \sum_{n=0}^\infty E_n(x,f,h\,|\,t,\alpha)z^n
 = \mathcal{J}(z;\mathbf{b},\boldsymbol{\lambda}),
\end{equation}
where $b_k=h^{k}(k+\alpha t)$ and  $\lambda_{k+1}=xh^{2k+1}k(k-1+2\alpha f)/2$ for $k\geq 0$.
\end{enumerate}
\end{theorem}
\begin{proof}  
Combining Eqs.~\eqref{andre:poly1}, \eqref{divi-gamma-2}, \eqref{gamma-des-special2-intro} and  
Theorem~\ref{thm1} with
\[(\mathbf{u},f,g,h,t,\alpha,\beta,\,p,q)=(1,x,1,0,1,1,h,1,1,1,\,p,q),
\]
 we obtain the first $J$-fraction \eqref{continued-andre}.

 Similarly, combining Eqs.~\eqref{andre:poly2}, ~\eqref{divi-gamma-1}, \eqref{gamma-des-special1} and  
Theorem~\ref{thm1} with 
\[
(\mathbf{u},f,g,h,t,\alpha,\beta,\,p,q)=(1,x,1,0,f,f,h,t,\alpha,\alpha,\,1,1),
\]
we obtain the second $J$-fraction~\eqref{continued-andre2}. 
\end{proof}
\begin{remark}
Han~\cite{Han20} asked for a combinatorial interpretation of
\eqref{continued-andre} when  $x=h=p=1$, which is a $q$-analogue of 
 the remarkable  $J$-fraction for  Euler numbers \(E_{n+1}\):
\begin{equation}\label{CF:E_n}
\sum_{n=0}^{\infty} E_{n+1} z^n 
= \mathcal{J}\!\left(z; k+1,\binom{k+1}{2}\right).
\end{equation}
Pan and Zeng~\cite{PZ21} established the $h=1$ case
of \eqref{continued-andre} shortly. 
The  $f=h=1$ case of \eqref{continued-andre2} appeared in \cite{XZ24}.
\end{remark}

\section{Factorization of \texorpdfstring{$\gamma$}{gamma}-coefficients via group actions}
\label{sec:proof-andre}

\subsection{Proof of Theorem~\ref{Thm: generalized PZ}}\label{Appendix-A}

We  use Pan-Zeng's group action \cite{PZ21} to prove the following equivalent identity, 
for $0\le k\le \lfloor n-1/2\rfloor$,  
    \begin{equation}\label{p-equ:pz}
\sum_{\sigma\in\S_{n,\des=k}^{\dd=0}}h^{\dis(\sigma)}p^{\ldes(\sigma)}q^{\rasc(\sigma)}=(p+q)^k\sum_{\sigma\in\mathcal{A}_{n,\des=k}^2}h^{\dis(\sigma)}p^{\ldes(\sigma)-k}q^{\rasc(\sigma)}.
    \end{equation}

For convenience, define the subset of $\S_n$
\begin{equation}\label{dd=0-subset}
    \S_n^{\dd=0}=\{\sigma\in\S_n:\dd(\sigma)=0\}.
\end{equation}
For any permutation
$\sigma\in\S_n^{\dd=0}$ and $x\in[n]$, we shall identify $\sigma$ with its $x$-factorization\footnote{For convenience, we adopt a concise form that is equivalent to Definition \ref{x-frac}.}, i.e.,
\[
\sigma=(w_1,w_2,x,w_4,w_5)=w_1\,w_2\,x\,w_4\,w_5,
\]
where $w_2$ (resp., $w_4$) is the maximal contiguous interval (possibly empty) immediately to the left (resp., right) of $x$ whose letters are all greater than $x$. Let $y_1$ (resp. $y_2$) be the smallest element in $w_2$ (resp. $w_4$), i.e.,
 \[y_1:=\min(w_2)\,\text{ and }\, y_2:=\min(w_4).\]
\begin{definition}
A valley $x$ of $\sigma$ is said to be
\begin{itemize}
  \item \emph{good} (resp.\ \emph{bad}) if $y_1>y_2$ (resp.\ $y_1<y_2$);
  \item of \emph{type I} if $\min(y_1,y_2)$ is a peak or double ascent;
  \item of \emph{type II} if $\min(y_1,y_2)$ is a valley.
\end{itemize}
\end{definition}

\begin{proposition}
Let $\sigma\in\S_n^{\dd=0}$ and $x\in \mathrm{Val}(\sigma)$ with $y=\min(y_1,y_2)$.
\begin{itemize}
\item[(i)] If $x$ is of type I and $y$ is a peak, then $w_4=y$ (resp.\ $w_2=y$) if
      $y_1>y_2$ (resp.\ $y_1<y_2$).
\item[(ii)] If $x$ is of type I and $y$ is a double ascent, then $w_4=yw_4''$
      (resp.\ $w_2=yw_2''$) with $w_2'',w_4''\neq\epsilon$ if
      $y_1>y_2$ (resp.\ $y_1<y_2$).
\item[(iii)] If $x$ is of type II, then $w_4=w_4' y w_4''$ (resp.\ $w_2=w_2' y w_2''$)
      with $w_2',w_2'',w_4',w_4''\neq\epsilon$ if
      $y_1>y_2$ (resp.\ $y_1<y_2$).
\end{itemize}
\end{proposition}

\begin{definition}
For $\sigma\in\S_n^{\dd=0}$ and each $x\in \mathrm{Val}(\sigma)$ with $y=\min(y_1,y_2)$,
we define its transform $\varphi(\sigma,x)$ as follows:
\begin{enumerate}
\item If $x$ is of type I and $y$ is a peak, then
\[
\varphi(\sigma,x)=
\begin{cases}
(w_1,y,x,w_2,w_5) & \text{if } y=y_2,\\
(w_1,w_4,x,y,w_5) & \text{if } y=y_1.
\end{cases}
\]
\item If $x$ is of type I and $y$ is a double ascent, then
\[
\varphi(\sigma,x)=
\begin{cases}
(w_1,yw_2,x,w'',w_5) & \text{if } y=y_2 \text{ and } w_4=yw'',\\
(w_1,w'',x,yw_4,w_5) & \text{if } y=y_1 \text{ and } w_2=yw''.
\end{cases}
\]
\item If $x$ is of type II, then
\[
\varphi(\sigma,x)=
\begin{cases}
(w_1,w_2yw',x,w'',w_5) & \text{if } y=y_2 \text{ and } w_4=w'yw'',\\
(w_1,w',x,w''yw_4,w_5) & \text{if } y=y_1 \text{ and } w_2=w'yw''.
\end{cases}
\]
with $w',w''\neq\epsilon$.
\end{enumerate}
\end{definition}

The three cases $(i), (ii), (iii)$ of the transformation with $y=y_2$ are
depicted in Fig.~\ref{fig:pz}.

\begin{figure}[htp]
    \centering
    \includegraphics[width=0.95\linewidth]{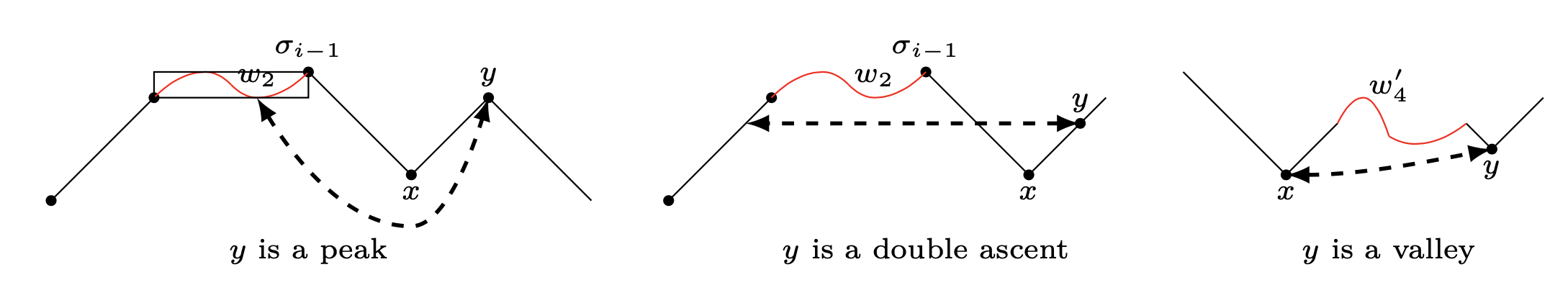}
    \caption{The transform $\varphi(\sigma,x)$ according to the type of $y$.
}
    \label{fig:pz}
\end{figure}
 This transformation satisfies the following proposition.
\begin{proposition}
If $\sigma\in \S_{n,\des=k}^{\dd=0}$ and $x\in \mathrm{Val}(\sigma)$, then $\varphi(\sigma,x)\in \S_{n,\des=k}^{\dd=0}$ and
\begin{align}
\dis(\varphi(\sigma,x)) &= \dis(\sigma),\label{eq:D}\\[4pt]
(2\!-\!13)\,\varphi(\sigma,x)
&=
\begin{cases}
(2\!-\!13)(\sigma)+1 & \text{if $x$ is good},\\
(2\!-\!13)(\sigma)-1 & \text{if $x$ is bad},
\end{cases}\label{2-13}\\[6pt]
(31\!-\!2)\,\varphi(\sigma,x)
&=
\begin{cases}
(31\!-\!2)(\sigma)-1 & \text{if $x$ is good},\\
(31\!-\!2)(\sigma)+1 & \text{if $x$ is bad}.
\end{cases}\label{31-2}
\end{align}

\end{proposition}
\begin{proof}
   Since the transformation $\varphi(\sigma,x)$ preserves the values at the valleys and peaks of $\sigma$, it follows directly from \eqref{eq:distance} that identity \eqref{eq:D} holds. The proofs of \eqref{2-13} and \eqref{31-2} can be found in \cite[Proposition~5]{PZ21}.
\end{proof}

Next, we define the transform $\varphi(\sigma,S)$ for any subset $S$ of $\mathrm{Val}(\sigma)$ with $\sigma\in\S_n^{\dd=0}$.

\begin{definition}
Let $\sigma\in\S_n^{\dd=0}$. For any $S\subseteq \mathrm{Val}(\sigma)$, let $\{S_1,S_2\}$ be the partition of $S$ such that
\begin{enumerate}
\item $S_1$ is the subset of $S$ consisting of valleys of type I, say $i_1,\ldots,i_\ell$;
\item $S_2$ is the subset of $S$ consisting of valleys of type II, say $j_k<\cdots<j_2<j_1$.
\end{enumerate}
Define the transforms
\[
\varphi(\sigma,S_1)=\varphi\bigl(i_\ell,\ldots,\varphi(i_2,\varphi(i_1,\sigma))\bigr),
\]
\[
\varphi(\sigma,S_2)=\varphi\bigl(j_k,\ldots,\varphi(j_2,\varphi(j_1,\sigma))\bigr),
\]
\[
\varphi(\sigma,S)=\varphi\bigl(\varphi(\sigma,S_1),S_2\bigr).
\]
\end{definition}
\begin{remark}
    Pan-Zeng prove that the image
$\varphi(\sigma,S_1)$ is independent of the order of $i_1,\ldots,i_\ell$, while
$\varphi(\sigma,S_2)$ is defined for the elements of $S_2$ in the decreasing order
$j_1>j_2>\cdots>j_k$.
\end{remark}

\begin{proposition}
If $\sigma \in \A_{n,\des=k}^2$ and $S \subseteq \Val(\sigma)$, then $\tau := \varphi(\sigma,S) \in \S_{n,\des=k}^{\dd=0}$ is well defined and
\begin{equation}\label{eq:2.4}
S=\{x\in \Val(\tau)\mid x \text{ is a bad valley}\}.
\end{equation}
\end{proposition}

For any set $S$, we denote by $2^{S}$ the set of all subsets of $S$. In what follows, for
$\sigma \in \A_{n,\des=k}^2$ we will identify $\Val(\sigma)$ with $[k]$ under the map
$a_i \mapsto i$ for $i \in [k]$ if $\Val(\sigma)$ consists of $a_1<a_2<\cdots<a_k$, and
identify any subset $S\in \Val(\sigma)$ with its image $S'\in 2^{[k]}$. Thus, we will use
$2^{[k]}$ instead of $2^{\Val(\sigma)}$.

\begin{proposition}\label{prop:PZ}
The map $\varphi:\A^2_{n,\des=k}\times 2^{[k]}\to \S_{n,\des=k}^{\dd=0}$ is a bijection such that
for $(\sigma,S')\in \A^2_{n,k}\times 2^{[k]}$ we have
\begin{align}
\dis(\sigma)&=\dis(\varphi(\sigma,S));\label{S-D}\\
(2\text{-}13)(\sigma) + |S'| &= (2\text{-}13)\,\varphi(\sigma,S);\label{2-13-D}\\
(31\text{-}2)(\sigma) - |S'| &= (31\text{-}2)\,\varphi(\sigma,S).\label{31-2-D}
\end{align}
\end{proposition}
\begin{proof}
    Eq. \eqref{S-D}  follows directly from \eqref{eq:D}. The proofs of \eqref{2-13-D} and \eqref{31-2-D} can be found in \cite[Proposition~7]{PZ21}.
\end{proof}
\begin{proof}[\textbf{Continuation of Proof of Eq.~\eqref{p-equ:pz}}]
Since $(p+q)^k = \sum_{S\in 2^{[k]}} p^{|S|} q^{k-|S|}$, it follows from Proposition~\ref{prop:PZ} that
\[
\sum_{(\sigma,S)\in \A^2_{n,\des=k}\times 2^{[k]}} h^{\dis(\sigma)}
p^{\ldes(\sigma)-|S|}q^{\rasc(\sigma)+|S|}
=
\sum_{\sigma\in \S_{n,\des=k}^{\dd=0}} h^{\dis(\sigma)} p^{\ldes(\sigma)}\,q^{\rasc(\sigma)},
\]
which is precisely equivalent to \eqref{prop:PZ}.
\end{proof}

\begin{example}
     Let $\tau=11\,\mathbf{2}\,13\,\mathbf{1}\,6\,\mathbf{4}\,5\,\mathbf{3}\,8\,9\,\mathbf{7}\,10\in\S_{13,\des=5}^{\dd=0}$.
\begin{enumerate}
\item[(i)] First, let $(\sigma,S):=(\tau,\emptyset)$. We have $S_2=\{1,3\}$ and $\min(S_2)=1$, so $S=\{1\}$ and
\[
\sigma:=\varphi(\sigma,1)=11\,\mathbf{1}\,12\,13\,\mathbf{2}\,6\,\mathbf{4}\,5\,\mathbf{3}\,8\,9\,\mathbf{7}\,10.
\]
As $S_2=\{3\}$, we have $\min S_2=3$ and $S:=\{1,3\}$, hence
\[
\sigma:=\varphi(\sigma,3)=11\,\mathbf{1}\,12\,13\,\mathbf{2}\,6\,\mathbf{3}\,5\,\mathbf{4}\,8\,9\,\mathbf{7}\,10.
\]

\item[(ii)] Now for $\sigma$, $S_2=\emptyset$ and $S_1=\{4,7\}$, then $S:=\{1,3,4,7\}$ and
\[
\sigma:=\varphi(\sigma,S_1)=11\,\mathbf{1}\,12\,13\,\mathbf{2}\,6\,\mathbf{3}\,10\,\mathbf{7}\,8\,9\,\mathbf{4}\,5\in\mathcal{A}_{13,5}^2.
\]
\end{enumerate}
\end{example}

\subsection{Proof of Theorem~\ref{thm-andre-minimum}}\label{Appendix-B}

We prove the following equivalent identity
\begin{align}\label{equ:DLP}
&\sum_{\sigma\in\S_{n,\des=k}^{\dd=0}}f^{\lrmival(\sigma)}h^{\widetilde{\dis}(\sigma)}t^{\rlminda(\sigma)}\alpha^{\mathrm{lrmin}(\sigma)}\nonumber\\
&=2^k\,\sum_{\sigma\in\mathcal{A}_{n,\des=k}^1}f^{\rmival(\sigma)-1}h^{\dis(\sigma)}t^{\rlminda(\sigma)}\alpha^{\mathrm{rmin}(\sigma)-1}.
\end{align}
Dong–Lin–Pan \cite{DLP25} proved the above identity  in the special case $f=h=1$. We show that their approach extends to  the general case. Recall that for any $x\in[n]$ and $\sigma\in\mathfrak{S}_n$, the $x$-factorization of $\sigma$ is defined by
\[
\sigma = w_1 w_2\, x\, w_4 w_5,
\]
where $w_2$ (resp., $w_4$) is the maximal contiguous interval (possibly empty) immediately to the left (resp., right) of $x$ whose letters are all greater than $x$.  Define the \emph{swap mapping}  $\varphi_x:\S_n\to\S_n$ by
\begin{equation}\label{varphi-x}
    \varphi_x(\s):= w_1w_4\,x\,w_2w_5.
\end{equation}
For instance, with $\sigma=348579162\in\mathfrak{S}_{9}$, in the $5$-factorization of $\sigma$, $w_2=8$ and $w_4=79$ and so $\varphi_5(\sigma)=347958162$.
It is plain to see that $\varphi_x$ is an involution acting on $\mathfrak{S}_n$, and for all $x,y\in[n]$, $\varphi_x$ and $\varphi_y$ commute.
Therefore, for any $S\subseteq[n]$ we can define the function
\[
\varphi_S:\mathfrak{S}_n\to\mathfrak{S}_n
\quad\text{by}\quad
\varphi_S=\prod_{x\in S}\varphi_x,
\]
where multiplication is the composition of mappings.
Hence the group $\mathbb{Z}_2^n$ acts on $\mathfrak{S}_n$ via the function $\varphi_S$.
This action is known as the \emph{Foata--Strehl action} (FS-action for short) on permutations~\cite{FS74}.

\begin{definition}[Bi-basic decomposition]
    For $\sigma\in\S_n$, with $\rlmin(\sigma)-1=k$ and $\lrmin(\sigma)-1=l$, there is a unique decomposition of $\sigma$ as
    \begin{equation}\label{bi-decomp}
        \sigma = \alpha_1\alpha_2\cdots\alpha_k\,1\,\beta_1\beta_2\cdots\beta_l,
    \end{equation}
    where \begin{itemize}
  \item the first letter of each $\alpha_i$ ($1 \le i \le k$), which is smallest inside $\alpha_i$,
  is a left-to-right minimum of $\sigma$;
  \item and the last letter of each $\beta_i$ ($1 \le i \le l$), which is smallest inside $\beta_i$,
  is a right-to-left minimum of $\sigma$.
\end{itemize}
This is the bi-basic decomposition of $\sigma$ and each $\alpha_i$ or $\beta_i$ is called a block of $\sigma$.
\end{definition}
\begin{example}
    Let $\sigma=7\,10\,2\,9\,5\,\mathbf{1}\,8\,11\,3\,12\,6\,4
  \in \mathfrak{S}_{12}$. Then we have 
  \begin{equation}
      \red{7\,10\,|\,2\,9\,5}\,|\,{\mathbf1}\,|\,\blue{8\,11\,3\,|\,12\,6\,4}.
  \end{equation}
\end{example}

For a word $w = w_1w_2\cdots w_n$, let $w^r := w_nw_{n-1}\cdots w_1$ be the reversal of $w$. 

For a given permutation $\sigma \in \S_n^{\dd=0}$  with bi-basic decomposition as in \eqref{bi-decomp} and $x \in [n]$, Dong-Lin-Pan defined the group action 
 $\psi_{x}(\sigma)$ as follows.

\begin{itemize}
    \item[(1)] If $x$ is the first letter of $\alpha_i$  for some $i$, then $x$ is a valley of $\sigma$. Let 
    \begin{equation}
        \widetilde{\alpha}_i:=\varphi_S(\alpha_i),
    \end{equation}
    where $S$ is the set of all double ascents of $\sigma$ inside $\alpha_i$. Define $\psi_{x}(\sigma)$ to be the permutation obtained from $\sigma$ by deleting the block $\alpha_i$ and then inserting $(\widetilde{\alpha}_i)^r$ in some proper gap between blocks so that $x$ becomes an additional right-to-left minimum;
    \item[(2)] If $x$ is the last letter of $\beta_i$ for some $i$, where $\beta_i$ has length at least two~\footnote{In case (1), each block $\alpha_i$ has length at least two, as $\sigma$ has no double descents.},
then $x$ is also a valley of $\sigma$. Let
\[
\widetilde{\beta}_i := \varphi_S(\beta_i),
\]
where $S$ is the set of all double ascents of $\sigma$ inside $\beta_i$. Define $\psi_x(\sigma)$ to be the permutation obtained from $\sigma$ by deleting the block $\beta_i$
and then inserting $(\widetilde{\beta}_i)^r$ in some proper gap between blocks so that $x$ becomes an additional
left-to-right minimum.
\item[(3)] If $x\neq 1$ is a valley of $\sigma$ not in the above two cases, then set $\psi_x(\sigma)=\varphi_x(\sigma)$ defined in \eqref{varphi-x};
\item[(4)] Otherwise, set $\psi_x(\sigma)=\sigma$.
\end{itemize}
As shown by Dong et al. \cite{DLP25}, for any $S \subseteq [n]$,  the restriction $\psi_S$ over $\S_n^{\dd=0}$ is stable and  defines a $\mathbb{Z}_2^n$-action on $\S_n^{\dd=0}$, which is called the \emph{block-Foata--Strehl action} (BFS-action for short).
Let $V(\sigma)$ be the set of valleys other than $1$ in $\sigma \in \mathfrak{S}_n^{\dd=0}$. It is clear that $\psi_x(\sigma) = \sigma$ if and only if
$x \notin V(\sigma)$.

\begin{example}
   If $\sigma=
{\color{blue}3\,4\,8\,5\,7\,10}\,|\,1\,|\,{\color{magenta}6}\,{\color{magenta}2}\,|\,{\color{magenta}9}
\in \S_{10}^{\dd=0}$,
then $V(\sigma)=\{2,3,5\}$. We compute that
\[
\begin{aligned}
\psi_{2}(\sigma) &= 3\,4\,8\,5\,7\,10\,|\,{\color{cyan}2\,6}\,|\,1\,|\,9,\\
\psi_{3}(\sigma) &= 1\,|\,6\,2\,|\,{\color{cyan}4\,7\,10\,5\,8\,3}\,|\,9,\\
\psi_{5}(\sigma) &= 3\,4\,{\color{cyan}7\,10\,5\,8}\,|\,1\,|\,6\,2\,|\,9,\\
\psi_{\{2,3\}}(\sigma) &= {\color{cyan}2\,6}\,|\,1\,|\,{\color{cyan}4\,7\,10\,5\,8\,3}\,|\,9,\\
\psi_{\{2,5\}}(\sigma) &= 3\,4\,{\color{cyan}7\,10\,5\,8}\,|\,{\color{cyan}2\,6}\,|\,1\,|\,9,\\
\psi_{\{3,5\}}(\sigma) &= 1\,|\,6\,2\,|\,{\color{cyan}4\,8\,5\,7\,10\,3}\,|\,9,\\
\psi_{V}(\sigma) &= {\color{cyan}2\,6}\,|\,1\,|\,{\color{cyan}4\,8\,5\,7\,10\,3}\,|\,9.
\end{aligned}
\]
\end{example}

\begin{lemma}\label{lem:BSF}
For any $\sigma\in\S_n^{\dd=0}$ and $x\in[n]$,  the mapping $\psi_x$ satisfies the following identity:
\begin{equation} \label{prop:psi'}(\des,\lrmival,\widetilde{\dis},\rlminda,\mathrm{lrmin})\,\sigma=(\des,\lrmival,\widetilde{\dis},\rlminda,\mathrm{lrmin})\,\psi_x(\sigma).
\end{equation}
\end{lemma}

\begin{proof}

From the bi-basic decomposition of $\sigma$, it is easy to see that 
\begin{itemize}
    \item the first  element of each block $\alpha_i$ is a left-to-right-minimum-valley;
    \item the last element of each block $\beta_i$ of length at least two is a right-to-left-minimum-valley;
    \item the (unique) element in each block $\beta_i$ of length one is a right-to-left-minimum-double-ascent.
\end{itemize}
In the first two cases, we see that mapping $\psi_x$ transfers a left-to-right-minimum-valley to a right-to-left-minimum-valley, or a right-to-left-minimum-valley to a left-to-right-minimum-valley. In the third case, the right-to-left-minimum-double-ascent is fixed under mapping $\psi_x$. Thus, the mapping $\psi_x$ preserves the statistics $\lrmival$, $\mathrm{lrmin}$, and $\rlminda$. Moreover, since the mapping $\psi_x$ does not change the values of the valleys and peaks of $\sigma$, it keeps $\widetilde{\dis}$ unchanged.  The reverse operator and the mapping $\varphi_x$ keep $\des$ unchanged as well. Therefore, \eqref{prop:psi'} follows.
\end{proof}

\begin{definition}[\cite{HJO23}]\label{def:cycle-André'}
Let $A:=\{a_1,\dots, a_k\}$ be a set of $k$ positive integers. 
Let $C=(a_1\cdots a_k)$ be a cycle (cyclic permutation) of $A$  with $a_1=\min\{a_1,\dots, a_k\}$. Then,  cycle  $C$ is called an
\textbf{\emph{André cycle}} if the word 
 $a_2\cdots a_k$ is an André permutation of the first kind.
 We say that a  permutation is a \textbf{cycle André permutation}
if it is a product of disjoint André cycles. 
\end{definition}
Let $\Web_{n}$ be the set of cycle André   permutations  of $[n]$. For example, $(1\;5\;6)(2\;7\;4\;8)(3)$ is a cycle André   permutation.

\begin{lemma}[\cite{HJO23} Theorem~4.1]\label{lem:phi-bi}
    There exists a bijection $\phi: \Web_n\to \A_{n+1}^1$ as follows: let $\s=C_1C_2\ldots C_l\in \Web_n$, where 
  $C_1, \ldots, C_l$ are André cycles ordered from   left to right in increasing order of their smallest  letters, which are  at the end of each cycle;
erasing the parentheses   in $\s$ and appending letter $n+1$ to the end results in $\phi(\s)$. 
\end{lemma}

\begin{proof}[\textbf{Continuation of Proof of Eq.~\eqref{equ:DLP}}]
    For $\sigma\in\S_n^{\dd=0}$, let $\mathrm{Orb}(\sigma)$ be the orbit of $\sigma$ under the BSF-action. As $\psi_S(\sigma)=\sigma$ if and only if $S\cap \mathrm{V}(\sigma)=\emptyset$, we have $|\mathrm{Orb}(\sigma)|=2^{\mathrm{V}(\sigma)}$. Assuming that $\lrmin(\sigma)=l$, then by Definition \ref{def:x-factorization-André1-2'} of André permutations of the first kind, there exists a unique $\hat{\sigma}\in\mathrm{Orb}(\sigma)$ with bi-basic decomposition 
    \begin{equation}\label{equ:form}
        \hat{\sigma}=1\beta_1\beta_2\cdots\beta_l
    \end{equation}
    satisfying that 
    \begin{itemize}
        \item [$(\star)$] each block $\beta_i$ when removing its last letter becomes an André permutation of the first kind.
    \end{itemize}

    By Lemma~\ref{lem:BSF}, we have 
    \begin{align}
        \sum_{\pi\in\mathrm{Orb}(\sigma)}f^{\lrmival(\pi)}h^{\widetilde{\dis}(\pi)}t^{\rlminda(\pi)}\alpha^{\mathrm{lrmin}(\pi)}=2^{\mathrm{V}(\hat{\sigma})}f^{\lrmival(\hat{\sigma})}h^{\widetilde{\dis}(\hat{\sigma})}t^{\rlminda(\hat{\sigma})}\alpha^{\mathrm{lrmin}(\hat{\sigma})}\\
        =2^{\des(\hat{\sigma})}f^{\rmival(\hat{\sigma})-1}h^{\widetilde{\dis}(\hat{\sigma})}t^{\rlminda(\hat{\sigma})}\alpha^{\rlmin(\hat{\sigma})-1},
    \end{align}
    The second identity follows from $|\mathrm{V}(\hat{\sigma})|=\des(\hat{\sigma})$, $\lrmival(\hat{\sigma})=\rmival(\hat{\sigma})-1$ and  $\mathrm{lrmin}(\hat{\sigma})=\rlmin(\hat{\sigma})-1$. Summing over all orbits of $\S_{n,\des=k}^{\dd=0}$ under the BFS-action gives
    \begin{equation}\label{equ:BSF-af}
\sum_{\sigma\in\S_{n,\des=k}^{\dd=0}}f^{\lrmival(\sigma)}h^{\widetilde{\dis}(\sigma)}t^{\rlminda(\sigma)}\alpha^{\mathrm{lrmin}(\sigma)}=2^k\,\sum_{\hat{\sigma}\in\widehat{\S}_{n,\des=k}^{\dd=0}}f^{\rmival(\hat{\sigma})-1}h^{\widetilde{\dis}(\hat{\sigma})}t^{\rlminda(\hat{\sigma})}\alpha^{\mathrm{rmin}(\hat{\sigma})-1},
    \end{equation}
    where $$\widehat{\S}_{n,\des=k}^{\dd=0}:=\{\hat{\sigma}\in\widehat{\S}_{n,\des=k}^{\dd=0}:\text{ bi-basic decomposition of $\hat{\sigma}$ has the form \eqref{equ:form} satisfying $(\star)$}\}.$$  
    For a word $w=w_1w_2\cdots w_n$, let $c(w)$ be the cyclic permutation $(w_1-1,w_2-1,\dots,w_n-1)$.
    There exists a natural one-to-one correspondence
    \begin{equation}\label{corresponding}
\hat{\sigma}=1\beta_1\beta_2\cdots\beta_l\mapsto c(\beta_1)c(\beta_2)\cdots c(\beta_l)=\tau.
    \end{equation}
    between $\widehat{\S}_{n,\des=k}^{\dd=0}$ and $\{\tau\in\Web_{n-1}:\drop(\tau)=k\}$. Then, 
    combining this and  the bijection $\phi$ given in Lemma~\ref{lem:phi-bi}, we see $\phi(\tau):=\tau'\in\mathcal{A}_n$. One can easily check that
    \begin{equation}\label{equ:prop}
        (\des,\rmival-1,\widetilde{\dis},\rlminda,\rlmin)\,\hat{\sigma}=(\des,\rmival,\dis,\rlminda,\rlmin)\,\tau'.
    \end{equation}
   We obtain Eq.~\eqref{equ:DLP} by combining  \eqref{equ:BSF-af} and \eqref{equ:prop}.
\end{proof}

\begin{example}
  Let $\sigma=
{\color{blue}5\,10\,7\,8\,11}\,|\,{\color{blue}4\,6}\,|\,1\,|\,{\color{magenta}3\,2}\,|\,{\color{magenta}9}
\in \S_{11}^{\dd=0}$. Then, we have
\begin{align}
    \psi_5(\sigma)&=\blue{4\,6}\,|\,1\,|\, {\color{magenta}3\,2}\,|\,{\color{magenta}{8\,11\,7\,10\,5}} \,|\,{\color{magenta}9};\nonumber\\
    \psi_{4,5}(\sigma)&=1\,|\,{\color{magenta}3\,2}\,|\,{\color{magenta}6\,4}\,|\,{\color{magenta}8\,11\,7\,10\,5} \,|\,{\color{magenta}9};\nonumber\\
   \hat{\sigma}:= \psi_{4,5,7}(\sigma)&=1\,|\,{\color{magenta}3\,2}\,|\,{\color{magenta}6\,4}\,|\,{\color{magenta}10\,7\,8\,11\,5} \,|\,{\color{magenta}9}.\nonumber
\end{align}
By \eqref{corresponding}, we obtain 
\begin{equation}
    \hat{\sigma}\rightarrow \tau:=(2\,1)(5\,3)(9\,6\,7\,10\,4)(8)\in\Web_{10}.
\end{equation}
Applying the $\phi$, we have
\begin{equation}
    \tau':=\phi(\tau)=2\,1\,5\,3\,9\,6\,7\,10\,4\,8\,11\in\A_{11}^1.
\end{equation} 
We now verify \eqref{equ:prop}.
\begin{itemize}
    \item On one hand, we have $\des(\hat{\sigma})=|\{2,4,6,9\}|=4$, $\rmival(\hat{\sigma})-1=|\{1,2,4,5\}|-1=3$, $\rlminda(\hat{\sigma})=|\{9\}|=1$, $\rlmin(\hat{\sigma})=|\{1,2,4,5,9\}|=5$ and 
    $\widetilde{\dis}(\hat{\sigma})=3+6+10+11-(2+4+5+7)=12$.
    \item On the other hand, we have $\des(\tau')=|\{1,3,5,8\}|=4$, $\rmival(\tau')=|\{1,3,4\}|=3$, $\rlminda(\tau')=|\{8\}|=1$, $\rlmin(\tau')=|\{1,3,4,8,11\}|=5$, and 
    $\dis(\tau')=2+5+9+10-(1+3+4+6)=12$.
\end{itemize}
\end{example}

\section{Application to Total-positivity}\label{sec:total-positivity}
\begin{definition}[Total positivity]
A matrix of real numbers is said to be totally positive (TP in short) if all its minors are nonnegative.
\end{definition}
\begin{definition}[Coefficientwise total positivity]\label{def:coefTP}
We say that a polynomial in $\R[\mathbf{X}]$, where $\mathbf{X}$ is the set of variables, 
is \emph{coefficientwise-positive} if all its coefficients are nonnegative.
A matrix whose entries are polynomials with real coefficients is said to be
\emph{coefficientwise-totally positive} (\emph{coefficientwise-TP}, in short)
if all its minors, which are polynomials themselves, are coefficientwise-positive in $\mathbf{X}$
\end{definition}

Let $\mathbf{a} := (a_n)_{n \ge 0}$ denote a sequence of real numbers (resp. polynomials in $\R[\mathbf{X}]$). To this sequence, we associate the
corresponding \emph{Hankel matrix}.
\begin{equation}\label{Hankel-matrix}
    H_\infty(\mathbf{a}):=(a_{i+j})_{i,j\ge 0}
=
\begin{pmatrix}
a_0 & a_1 & a_2 & \cdots\\
a_1 & a_2 & a_3 & \cdots\\
a_2 & a_3 & a_4 & \cdots\\
\vdots & \vdots & \vdots & \ddots
\end{pmatrix}.
\end{equation}

We say that the sequence $\mathbf{a}$ is \emph{Hankel-totally positive} (resp. \emph{coefficientwise Hankel-totally positive}) if the
Hankel matrix $H_\infty(\mathbf{a})$ is  totally positive (resp. coefficientwise-totally positive in $\mathbf{X}$).

\subsection{S-fraction and coefficientwise-TP}
For a sequence of real numbers $(a_n)_{n \ge 0}$, its Hankel matrix 
$(a_{i+j})_{i,j \ge 0}$ is \emph{totally positive} if and only if there exist 
nonnegative real numbers $\mathbf{\alpha}:=(\alpha_n)_{n\ge 0}$ such that the following 
identity holds (at the level of formal power series):
\begin{equation}
\mathcal{S}(z; \mathbf{\alpha}):
= 
\cfrac{\alpha_0}{1 - \cfrac{\alpha_1 z}{1 - \cfrac{\alpha_2 z}{1 - \cfrac{\alpha_3 z}{\ddots}}}}=\sum_{n=0}^{\infty} a_n z^n  .
\label{eq:Sfraction}
\end{equation}
This result is a combination of results due to Stieltjes~\cite{Sti94} 
and Gantmakher--Krein~\cite{GK37}. 
The continued fraction of the form on the right-hand side of 
equation~\eqref{eq:Sfraction} is called a \emph{Stieltjes-type continued fraction} 
(\emph{$S$-fraction} in short).

If $(a_n)_{n \ge 0}$ is instead a sequence of polynomials in one or several variables,
we no longer have any equivalent condition available to us.
However, if an $S$-fraction such as the one in equation~\eqref{eq:Sfraction} exists,
where the $\alpha_0, \alpha_1, \ldots$ are all polynomials with nonnegative coefficients,
then that is a \emph{sufficient} condition for the Hankel matrix
$(a_{i+j})_{i,j \ge 0}$ to be coefficientwise-totally positive. The following fact is folklore.

\begin{lemma}\label{lem:CoeffHankelTP}
Let $\mathbf{a} = (a_n)_{n \ge 0}$ and $\mathbf{\alpha}=(\alpha_n)_{n \ge 0}$
be two sequences of real coefficient polynomials in $\R[\mathbf{X}]$ satisfying the following power series identity~\eqref{eq:Sfraction}.
If the polynomials $\alpha_n$ ($n\ge 0$) are coefficientwise-positive, 
then $\mathbf{a}$ is \emph{coefficientwise-Hankel totally positive} in $\mathbf{X}$.
\end{lemma}

\noindent
This fact was first stated in this form in~\cite{Sok14}, although
it follows straightforwardly from the work of \cite{Fla80}
and of~\cite[Section~3, Chapter~4]{Vie83}. Note that the converse of Lemma~\ref{lem:CoeffHankelTP} does not hold in general.



A permutation $\sigma=\sigma_1\sigma_2\cdots \sigma_n\in\S_{n}$ is called an \emph{alternating permutation} (or \emph{zigzag permutation}) if $\sigma_1>\sigma_2<\sigma_3>\sigma_4<\cdots$. As proved by André~\cite{And79}, the number of alternating permutations in $\S_n$ is given by the Euler number $E_n$. Let  $\mathcal{Z}_{n}$  denote the set of alternating permutations of length $n$.   A permutation of odd length is alternating if and only if it contains neither a double ascent nor a double descent. 

Define the polynomials
\begin{multline}   Z_n^o(x,f,g,h,\alpha,\beta\,|\,p,q):=\sum_{\sigma\in\mathcal{Z}_{2n+1}}x^{\des(\sigma)}f^{\mathrm{lmaxpk(\sigma)-1}}g^{\bmax(\s)-1}h^{\dis(\sigma)} \\\times\alpha^{{\lmax}(\sigma)-1}\beta^{\rmax(\sigma)-1}p^{\ldes(\sigma)}q^{\rasc(\sigma)}.
\end{multline}   

\begin{theorem}
    The sequence $\mathcal{Z}_o :=(Z_n(x,f,g,h,\alpha,\beta\,|\,p,q))_{n\ge 0}$  is
coeﬃcientwise-Hankel totally positive in $\mathbf{X}:=\{x,f,g,h,\alpha,\beta,p,q\}$.
\end{theorem}
\begin{proof}
    By setting $\mathbf{u}=(1,x,0,0)$ in Theorem~\ref{thm1}, we derive 
    \begin{equation}\label{odd}      \sum_{n=0}^{\infty}Z_n(x,f,g,h,\alpha,\beta\,|\,p,q)z^{n}=\mathcal{J}(\sqrt{z};\mathbf{0},\boldsymbol{\lambda})=\mathcal{S}(z;\boldsymbol{\lambda}),
    \end{equation}
    where $\lambda_{k}=xh^{2k-1}[\,k\,]_{p,q}\,\bigl[\,k+1;\beta g , \alpha f\,\bigr]_{p,q}$.
    
    Hence, every $\lambda_k$ 	
  is coefficientwise-positive in the variables $x,f,g,h,\alpha,\beta,p,q$. By Lemma~\ref{lem:CoeffHankelTP}, the sequence $\mathcal{Z}_o$
  is therefore coefficientwise Hankel-totally positive.
\end{proof}

A permutation $\sigma$ of even length $2n$ is an alternating permutation if and only if $\sigma':=\sigma\star(2n+1)$ is an alternating permutation, where $\star$ denotes the concatenation of words.
\begin{lemma}\label{prop:relation}
     Let $\sigma':=\sigma\star(2n+1)\in\mathcal{Z}_{2n+1}$ with  $\sigma\in\mathcal{Z}_{2n}$. We have 
     \begin{equation}\label{sigma-sigma'}
         (\des,\pmax,\dis+2n-\sigma_{2n},\lmax)\,\sigma=(\des,\pmax-1,\dis,\lmax-1)\,\sigma'.
     \end{equation}     
     For each $x\in[2n]$, we have
    \[(31\!-\!2)_x(\sigma)=(31\!-\!2)_x(\sigma') \text{ and }
    (2\!-\!31)_x(\sigma)=\begin{cases}
        (2\!-\!13)_x(\sigma')\; &\text{if } \sigma^{-1}(x) \text{ is even;}\\[5pt]
        (2\!-\!13)_x(\sigma')-1\; &\text{if } \sigma^{-1}(x) \text{ is odd.} 
    \end{cases}\]
\end{lemma}
\begin{proof}
Eq.~\eqref{sigma-sigma'}  follows directly from the Definitions of these permutation statistics.  
    Let  $\sigma'=\sigma\star(n+1)\in\mathcal{Z}_{2n+1}$, where $\sigma\in\mathcal{Z}_{2n}$. By the $x$-decomposition of $\sigma'$, the first equality is obvious. Suppose that for $x\in[2n]$,  the $x$-decomposition of $\sigma'$ is 
 \[\sigma' = u_1 v_1 \cdots u_k v_k, \quad k \geq 1,\]   
 where $x$ lies in   $v_j$, see \eqref{x-decomposition}.
    It is straightforward to verify the following:
    \begin{itemize}
        \item word $v_k$ has $2n+1$ as its last letter, possibly together with other entries preceding $2n+1$.
        \item If $x$ is a peak of $\sigma'$, then $x$ cannot lie in  $v_k$.   
    \end{itemize}
    When $\sigma^{-1}(x)$ is even, i.e., $x$ is a valley of $\sigma'$, the word $v_j$  must contain elements that are larger than $x$ on both sides. Consequently, each contribution to $(2\!-\!13)_x(\sigma')$  corresponds to a contribution to $(2\!-\!31)_x(\sigma)$.
    When $\sigma^{-1}(x)$ is odd, the letter $x$ is a peak of $\sigma'$, and the word $v_j$ contains only $x$. In this case, one contribution to $(2\!-\!31)_x(\sigma)$ is lost. See the corresponding illustration below.
\begin{center}
    \begin{tikzpicture}[scale=0.65]
\draw[very thick] (-3.4,1) -- (15,1);
\draw[very thick] (-3.4,0)   -- (15,0);
\draw[red, very thick](-3.1,0) .. controls (-2.6,-1) .. (-2.1,0);
\draw[blue, very thick](-1.9,1) .. controls (-1.4,2) .. (-0.9,1);
\draw[red, very thick](-0.7,0) .. controls (-0.2,-1) .. (0.3,0);
\draw[blue, very thick](0.5,1) .. controls (1.0,2) .. (1.5,1);
\draw[red, very thick](1.7,0) .. controls (2.2,-1) .. (2.7,0);

\draw[blue, very thick](3.6,1) .. controls (4.1,2) .. (4.6,1);
\draw[red,  very thick] (4.8,0) .. controls (5.3,-1) .. (5.8,0);
\draw[blue, very thick] (6.0,1) .. controls (6.5, 2) .. (7.0,1);
\draw[red,  very thick] (7.2,0) .. controls (7.7,-1) .. (8.2,0);
\draw[blue, very thick] (8.4,1) .. controls (8.9, 2) .. (9.4,1);

\draw[red,  very thick] (10.1,0) .. controls (10.6,-1) .. (11.1,0);
\draw[blue, very thick] (11.3,1) .. controls (11.8,2) .. (12.3,1);
\draw[red,  very thick] (12.5,0) .. controls (13.0,-1) .. (13.5,0);
\draw[blue, very thick] (13.7,1) .. controls (14.2,2) .. (14.7,1);

\node at (-2.6,-1.1) {$u_1$};
\node at (-0.2,-1.1) {$u_2$};
\node at (2.2,-1.1) {$u_3$};
\node at (5.3,-1.1)   {$u_j$};
\node at (7.7,-1.1)   {$u_{j+1}$};
\node at (10.6,-1.1)  {$u_{k-1}$};
\node at (13.0,-1.1)  {$u_{k}$};

\node at (6.5,0.5)   {$x$};
\node at (3.1,0.5) {$\cdots$};
\node at (10,0.5) {$\cdots$};
\node at (-1.4,2.1) {$v_1$};
\node at (1,2.1) {$v_2$};
\node at (4.1,2.1) {$v_{j-1}$};
\node at (6.5,2.1)   {$v_j$};
\node at (8.9,2.1)   {$v_{j+1}$};
\node at (11.8,2.1)  {$v_{k-1}$};
\node at (14.2,2.1)  {$v_k$};

\draw[thick, decorate, decoration={snake, amplitude=1mm, segment length=5mm}]
      (-0.9,1) -- (-0.7,0);
\draw[thick, decorate, decoration={snake, amplitude=1mm, segment length=5mm}]
      (1.5,1) -- (1.7,0);
\draw[thick, decorate, decoration={snake, amplitude=1mm, segment length=5mm}]
      (4.6,1) -- (4.8,0);


\draw[dashed, thick] (8.2,0) -- (8.4,1);
\draw[dashed, thick] (11.1,0) -- (11.3,1);
\draw[dashed, thick] (13.5,0) -- (13.7,1);

\draw[dashed, thick] (7.0,1) -- (7.2,0);
\draw[dashed, thick] (12.3,1) -- (12.5,0);
\node at (7.7,0.5) {$\leftrightarrow$
};
\node at (13,0.5) {$\leftrightarrow$};
\end{tikzpicture}
\end{center}
This completes the proof.
\end{proof}

 Define the polynomials 
\begin{equation}   Z_n^e(x,f,h,\alpha\,|\,p,q):=\sum_{\sigma\in\mathcal{Z}_{2n}}x^{\des(\sigma)}f^{\mathrm{lmaxpk(\sigma)}}h^{\dis(\sigma)+2n-\sigma_{2n}}\alpha^{{\lmax}(\sigma)}p^{\ldes(\sigma)}q^{2\!-\!31(\sigma)}.
\end{equation}  
\begin{theorem}
    The sequence $\mathcal{Z}_e :=(Z_n^e(x,f,h,\alpha\,|\,p,q))_{n\ge 0}$  is
coeﬃcientwise-Hankel totally positive in $\mathbf{X}:=\{x,f,h,\alpha,p,q\}$.
\end{theorem}

\begin{proof}
    Combining Lemma~\ref{prop:relation} and Theorem~\ref{thm1} with $\mathbf{u}=(1,x,0,0)$ and $\beta=0$ we derive 
    \begin{equation}\label{even}      \sum_{n=0}^{\infty}Z_n^e(x,f,h,\alpha\,|\,p,q)z^{n}=\mathcal{J}(\sqrt{z};\mathbf{0},\boldsymbol{\lambda})=\mathcal{S}(z;\boldsymbol{\lambda}),
    \end{equation}
    where $\lambda_{k}=xh^{2k-1}[\,k\,]_{p,q}\,\bigl[\,k+1;0, \alpha f\,\bigr]_{p,q}$. Hence, every $\lambda_k$ 	
  is coefficientwise-positive in the variables $x,f,h,\alpha,p,q$. By Lemma~\ref{lem:CoeffHankelTP}, the sequence $\mathcal{Z}_e$
  is therefore coefficientwise Hankel-totally positive.
\end{proof}

\subsection{J-fraction and coefficientwise-TP}

Let   $\mathbf{a}:=(a_n)_{n\ge 0}$ be a sequence of polynomials in $\R[\mathbf{X}]$, where $\mathbf{X}$ is the set of variables. Assume that there exists a $J$-fraction for which the following identity is satisfied:
\begin{equation}
    \sum_{n\ge 0}a_nt^n
= \mathcal{J}(t;\mathbf{b},\boldsymbol{\lambda})=\cfrac{1}{1-b_0 t \;-\;
   \cfrac{\lambda_1 t^2}{1-b_1 t \;-\;
   \cfrac{\lambda_2 t^2}{1-b_2 t \;-\; \ddots}}}.
\end{equation}
Define the associated tridiagonal matrix, called \emph{Jacobi matrix} or \emph{production matrix} see~\cite{PSZ23},
\begin{equation}\label{Jacobi-matrix}
    M_{\infty}(\mathbf{b},\boldsymbol{\lambda})
=\begin{pmatrix}
    b_0&1&0&0&\ldots\\
    \lambda_1&b_1&1&0&\ldots\\
    0&\lambda_2&b_2&1&\ldots\\
    0&0&\lambda_3&b_3&\ldots\\
    \vdots&\vdots&\vdots&\vdots&\ddots
\end{pmatrix}.
\end{equation}
The following characterization for coeﬃcientwise-TP is due to Sokal \cite{Sok14}, see also \cite[Theorem~9.7]{PSZ23}.
\begin{theorem}[\cite{Sok14}]\label{thm:Coeﬃcientwise-TP}
Let   $\mathbf{a}:=(a_n)_{n\ge 0}$ be a sequence of polynomials in $\R[\mathbf{X}]$ and  the associated Hankel matrix $H_\infty(\mathbf{a})$.
  If $M_\infty(\mathbf{b},\boldsymbol{\lambda})$ is coefficientwise-TP in $\mathbf{X}$, 
  then so is $H_\infty(\mathbf{a})$.
\end{theorem}
The classical criterion in \cite[Theorem 4.3]{Pin10} for tridiagonal (Jacobi) matrices over 
$\mathbb{R}$ extends verbatim to matrices with entries in an arbitrary partially ordered commutative ring; we shall therefore work in this generality throughout.
\begin{lemma}\label{lem:total-jacobi}
A Jacobi matrix  is coefficientwise-totally positive in $\mathbf{X}$ if and only if all its off-diagonal elements and all its principal minors containing consecutive rows and columns are coefficientwise-positive in $\mathbf{X}$.

\end{lemma}

\begin{proof}
Let $A=M_\infty(\mathbf b,\boldsymbol\lambda)=(a_{i,j})_{i,j\geq 0}$. Then $a_{i,j}=0$ if $|i-j|> 1$.  Fix $r\ge1$ and index sets
$I=(i_1<\cdots<i_r)$ and $J=(j_1<\cdots<j_r)$, and consider the minor
\[
A\binom{I}{J}:=\det\bigl(a_{i_p, j_q}\bigr)_{1\le p,q\le r}.
\]
If $|i_k-j_k|\ge2$ for some $k$, then $A\binom{I}{J}=0$ since $A$ is tridiagonal. Thus assume $|i_k-j_k|\le1$ for all $k$. Whenever $|i_\ell-j_\ell|=1$, the tridiagonal form forces a block-triangular
decomposition of the corresponding submatrix, yielding the factorization
\begin{equation}\label{eq:block-factor}
A\binom{i_1,\dots,i_r}{j_1,\dots,j_r}
=
A\binom{i_1,\dots,i_{\ell-1}}{j_1,\dots,j_{\ell-1}}\;
A\binom{i_\ell}{j_\ell}\;
A\binom{i_{\ell+1},\dots,i_r}{j_{\ell+1},\dots,j_r}.
\end{equation}
Iterating \eqref{eq:block-factor} we can factor any minor of $A$, namely, if
\[
i_1=j_1, \ldots,i_{k_1}=j_{k_1},\quad
i_{k_1+1}\neq j_{k_1+1},\ldots,i_{k_2}\neq j_{k_2},\quad
i_{k_2+1}=j_{k_2+1},\ldots,\ i_{k_3}=j_{k_3},\ldots,
\]
then
\begin{equation}\label{factorization-matrix}
A\binom{i_1,\ldots,i_r}{j_1,\ldots,j_r}
=
A\binom{i_1,\ldots,i_{k_1}}{j_1,\ldots,j_{k_1}}\,
A\binom{i_{k_1+1}}{j_{k_1+1}}\cdots
A\binom{i_{k_2}}{j_{k_2}}\,
A\binom{i_{k_2+1},\ldots,i_{k_3}}{j_{k_2+1},\ldots,j_{k_3}}\cdots, 
\end{equation}
where 
\[
A\binom{i_k}{j_k}=
\begin{cases}
1, & \text{if } j_k=i_k+1,\\[2pt]
\lambda_{i_k}, & \text{if } j_k=i_k-1.
\end{cases}
\]

Finally, if $k_t+1<k_{t+1}$ for some $t\in[r-1]$, then
\begin{equation}\label{eq:principal-split}
A\binom{k_1,\dots,k_s}{k_1,\dots,k_s}
=
A\binom{k_1,\dots,k_t}{k_1,\dots,k_t}\;
A\binom{k_{t+1},\dots,k_s}{k_{t+1},\dots,k_s}.
\end{equation}
Iterating \eqref{eq:principal-split}, every principal minor factors into a product
of consecutive principal minors.
Since coefficientwise-positive polynomials are closed under multiplication, all
factors on the right-hand side of \eqref{factorization-matrix} are coefficientwise-positive,
and therefore $A\binom{I}{J}$ is coefficientwise-positive for all $I,J$.
Hence $A$ is coefficientwise-totally positive.
\end{proof}

Define the  multivariable Eulerian polynomials
\begin{equation}
    A_n(x,h\,|\,p,q):=\sum_{\sigma\in\S_{n+1}}x^{\des(\sigma)}h^{\dis(\sigma)}p^{\ldes(\sigma)}q^{\rasc(\sigma)}.
\end{equation}

\begin{theorem}
    The sequence $\mathcal{A}:=(A_n(x,h\,|\,p,q))_{n\ge 0}$ is coefficientwise-Hankel totally positive in $\mathbf{X}:=\{x,h,p,q\}$.
\end{theorem}
\begin{proof}  
By Theorem \ref{thm1} with $u_2=u_4=x$ and $u_1=u_3=f=g=t=\alpha=\beta=1$, we have
\begin{equation}\label{j-fracion}
    \sum_{n\ge 0}^{\infty}A_n(x,h\,|\,p,q)z^n=\mathcal{J}(z;\mathbf{b},\boldsymbol{\lambda}),
\end{equation}
where $b_k=h^k(1+x)[k+1]_{p,q}$ and $\lambda_{k}=xh^{2k-1}[k]_{p,q}[k+1]_{p,q}$.

Let $A=M_\infty(\mathbf b,\boldsymbol\lambda)=(a_{i,j})_{i,j\geq 0}$. For $k\ge 0$ and $n\ge 1$, define the contiguous principal minors
\begin{equation}
    D_{k,n} :=A\binom{I}{J}\, \text{ with } I=J=\{k,\dots ,k+n-1\}.
\end{equation}
Then we have  the standard continuant recurrence
\begin{equation}\label{recurrence-D}
    D_{k,n}=b_{k+n-1}D_{k,n-1}-\lambda_{k+n-1}D_{k,n-2},
\end{equation}
where $D_{k,0}=1$ and $D_{k,1}=b_k$.

We prove by induction on $n\geq 1$ that 
\begin{equation}\label{cal:det}
    D_{k,n}
=
h^{\,nk+\binom{n}{2}}(1+x+x^2+\cdots+x^n)
\left(\prod_{j=1}^{n}[k+j]_{p,q}\right).
\end{equation}
Set $[n]_x:=1+x+\cdots+x^{n-1}$ and
\[
F_{k,n}:=h^{\,nk+\binom{n}{2}}[n+1]_x\bigg(\prod_{j=1}^n [k+j]_{p,q}\bigg).
\]
Clearly $F_{k,0}=1=D_{k,0}$ and $F_{k,1}=(1+x)h^k[k+1]_{p,q}=b_k=D_{k,1}$.
 By \eqref{j-fracion}, for $n\ge2$, we have 
\begin{align*}
    b_{k+n-1}&=(1+x)h^{k+n-1}[k+n]_{p,q},\\
\lambda_{k+n-1}&=x\,h^{2k+2n-3}[k+n-1]_{p,q}[k+n]_{p,q},
\end{align*}
and a direct computation gives
\[
b_{k+n-1}F_{k,n-1}-\lambda_{k+n-1}F_{k,n-2}
=
h^{\,nk+\binom{n}{2}}
\bigl((1+x)[n]_x-x[n-1]_x
\bigr)\biggl(\prod_{j=1}^n [k+j]_{p,q}\biggr).
\]
Since $(1+x)[n]_x-x[n-1]_x=[n+1]_x$, we obtain
\[
b_{k+n-1}F_{k,n-1}-\lambda_{k+n-1}F_{k,n-2}=F_{k,n}.
\]
Thus $F_{k,n}$ satisfies the same recurrence and initial conditions as $D_{k,n}$ for all $k\ge0$ and $n\ge1$. This establishes \eqref{cal:det}.

By Lemma \ref{lem:total-jacobi} and Theorem \ref{thm:Coeﬃcientwise-TP}, the sequence $\mathcal{A}$ is therefore coefficientwise Hankel-totally positive.

\end{proof}

\section{Concluding remarks}\label{sec:special-cases}

The following identities are direct consequences of the continued fraction: specializing one of the parameters to $-1$ makes the expression collapse into a simple closed form.

 \begin{proposition}\label{prop:simple-formula}  The following identities hold.
      \begin{subequations}
      \begin{align}
        &\sum_{\sigma\in\S_n}(-1)^{\inv(\sigma)}x^{\exc(\sigma)}h^{\depth(\sigma)}=(1-xh)^{n-1}\label{special-case1}\\
        &\sum_{\sigma\in\S_n}(-1)^{\nest(\sigma)}x^{\exc(\sigma)}h^{\depth(\sigma)}=(1+xh)^{n-1}\label{special-case2}\\
        &\sum_{\sigma\in\S_n}(-1)^{\rlmin(\sigma)}x^{\exc(\sigma)}h^{\depth(\sigma)}=(-1)^n(1-xh)^{n-1}\label{special-case3}\\
        &\sum_{\sigma\in\S_n}(-1)^{\lmax(\sigma)}h^{\depth(\sigma)}=(-1)^n(1-h)^{n-1}.
        \end{align}
 \end{subequations}   
  \end{proposition}
     \begin{proof}
     If $\lambda_2=0$, the $J$-fraction~\eqref{J-frac}  reduces to the rational fraction:
\begin{equation}\label{equ:truncate}
  \mathcal{J}(z;\mathbf{b},\boldsymbol{\lambda})=\frac{1}{1-b_0z-\dfrac{\lambda_1 z^2}{1-b_1z}}
=\frac{1-b_1z}{(1-b_0z)(1-b_1z)-\lambda_1 z^2}.  
\end{equation}
   Substituting  $\mathbf{u}=(x,1,x,1)$, $p=q=t=\alpha=\beta=1$, and  $s=-1$ in  Theorem~\ref{thm2}, we obtain 
   $b_0=1,\; b_1=-xh,\; \lambda_1=-xh,\; \lambda_2=0,$
and 
\begin{equation}   
\sum_{n\geq 0}\sum_{\sigma\in\S_n}(-1)^{\inv(\sigma)}x^{\exc(\sigma)}h^{\depth(\sigma)}z^n
=1+\frac{z}{1-(1-xh)z},
\end{equation}
which implies~\eqref{special-case1}. The remaining identities follow similarly.
\end{proof}
 
\begin{remark}
Eqs.~\eqref{special-case1}--\eqref{special-case3} were proved  by Eu et al.~\cite{Eu26} without using continued fractions.
\end{remark}

Eu et al.~\cite{Eu26} considered a variant of the notions of crossing and nesting for permutations as follows:
\begin{subequations}
    \begin{align}
\cros^*(\sigma)
= \#\{(i,j)\in [n]\times[n] : (i<j< \sigma_i<\sigma_j)\ \vee\ (\sigma_i<\sigma_j\le i<j)\};\\
\nest^*(\sigma)
= \#\{(i,j)\in [n]\times[n] : (i<j< \sigma_j<\sigma_i)\ \vee\ (\sigma_j<\sigma_i\le i<j)\}.
\end{align}
\end{subequations}
We note that these statistics are closely related to $\cros$ and $\nest$ (see~\eqref{def:cros-nest}) by the inverse mapping. More precisely, we have the following result.
\begin{figure}[htp]
\centering
\begin{picture}(400,80)(0,-3)
\setlength{\unitlength}{1.5mm}
\linethickness{.5mm}

\put(2,10){\line(1,0){28}}
\put(5,10){\circle*{1.3}}\put(5,10){\makebox(0,-6)[c]{\small $l$}}
\put(12,10){\circle*{1.3}}\put(12,10){\makebox(0,-6)[c]{\small $k$}}
\put(19,10){\circle*{1.3}}\put(19,10){\makebox(0,-6)[c]{\small $\sigma_{l}$}}
\put(26,10){\circle*{1.3}}\put(26,10){\makebox(0,-6)[c]{\small $\sigma_k$}}
\red{\qbezier(5,10)(12,20)(19,10)}
\blue{\qbezier(11,10)(18,20)(25,10)}
\put(16,0){\makebox(0,0)[c]{\small Upper $\mathrm{crossing}^*$}}

\put(36,10){\line(1,0){28}}
\put(39,10){\circle*{1.3}}\put(39,10){\makebox(0,-6)[c]{\small $l$}}
\put(50,10){\circle*{1.3}}\put(50,10){\makebox(0,6)[c]{\small $k=\sigma_{l}$}}
\put(60,10){\circle*{1.3}}\put(60,10){\makebox(0,-6)[c]{\small $\sigma_k$}}
\red{\qbezier(39,10)(44,0)(50,10)}
\blue{\qbezier(49.5,10)(53.5,0)(59,10)}

\put(70,10){\line(1,0){28}}
\put(74,10){\circle*{1.3}}\put(74,10){\makebox(0,-6)[c]{\small $\sigma_k$}}
\put(81,10){\circle*{1.3}}\put(81,10){\makebox(0,-6)[c]{\small $\sigma_{l}$}}
\put(88,10){\circle*{1.3}}\put(88,10){\makebox(0,-6)[c]{\small $k$}}
\put(95,10){\circle*{1.3}}\put(95,10){\makebox(0,-6)[c]{\small $l$}}
\red{\qbezier(74,10)(81,0)(88,10)}
\blue{\qbezier(80,10)(87,0)(94,10)}
\put(64,0){\makebox(0,0)[c]{\small Lower $\mathrm{crossing}^*$ $k\le \sigma_l$}}

\end{picture}
\caption{Upper $\mathrm{crossing}^*$ and lower $\mathrm{crossing}^*$ indexed at $k$.}\label{fig:cros^*}
\end{figure}

\begin{figure}[htp]
\centering
\begin{picture}(400,80)(0,-3)
\setlength{\unitlength}{1.5mm}
\linethickness{.5mm}

\put(2,10){\line(1,0){28}}
\put(5,10){\circle*{1.3}}\put(5,10){\makebox(0,-6)[c]{\small $l$}}
\put(12,10){\circle*{1.3}}\put(12,10){\makebox(0,-6)[c]{\small $k$}}
\put(19,10){\circle*{1.3}}\put(19,10){\makebox(0,-6)[c]{\small $\sigma_k$}}
\put(26,10){\circle*{1.3}}\put(26,10){\makebox(0,-6)[c]{\small $\sigma_l$}}
\red{\qbezier(5,10)(15.5,20)(26,10)}
\blue{\qbezier(11,10)(14.5,15)(18,10)}
\put(16,0){\makebox(0,0)[c]{\small Upper $\mathrm{nesting}^*$}}

\put(36,10){\line(1,0){28}}
\put(39,10){\circle*{1.3}}\put(39,10){\makebox(0,-6)[c]{\small $l$}}
\put(49,10){\circle*{1.3}}\put(49,10){\makebox(0,6)[c]{\small $k=\sigma_k$}}
\put(60,10){\circle*{1.3}}\put(60,10){\makebox(0,-6)[c]{\small $\sigma_l$}}
\red{\qbezier(39,10)(49.5,0)(60,10)}
\blue{%
  \qbezier(48.3,10)(50.5,7)(47.5,7)
  \qbezier(48.3,10)(46,9)(47.5,7)
}

\put(70,10){\line(1,0){28}}
\put(74,10){\circle*{1.3}}\put(74,10){\makebox(0,-6)[c]{\small $\sigma_l$}}
\put(81,10){\circle*{1.3}}\put(81,10){\makebox(0,-6)[c]{\small $\sigma_k$}}
\put(88,10){\circle*{1.3}}\put(88,10){\makebox(0,-6)[c]{\small $k$}}
\put(95,10){\circle*{1.3}}\put(95,10){\makebox(0,-6)[c]{\small $l$}}
\red{\qbezier(74,10)(84,1)(95,10)}
\blue{\qbezier(80,10)(83.5,6)(87.5,10)}
\put(64,0){\makebox(0,0)[c]{\small Lower $\mathrm{nesting}^*$ $k\le \sigma_k$}}

\end{picture}
\caption{Upper $\mathrm{nesting}^*$ and lower $\mathrm{nesting}^*$ indexed at $k$}\label{fig:nest^*}
\end{figure}

\begin{lemma}\label{final lemma}
    The inverse mapping $\iota: \S_n\to \S_n,\; \iota(\sigma)=\sigma^{-1}$  is a bijection   and satisfies
    \begin{equation}\label{property}(\exc,\drop,\fix,\depth,\cros,\nest)\,\sigma=(\drop,\exc,\fix,\depth,\cros^*,\nest^*)\,
    \iota(\sigma).
    \end{equation}
\end{lemma}
\begin{proof}
We  only  verify the statistic $\depth$; the remaining correspondences are straightforward. 
By Definition \ref{def:depth}, we see that
    \[
\depth(\sigma^{-1})
=\sum_{\sigma^{-1}(k)>k}\big(\sigma^{-1}(k)-k\big)=\sum_{i>\sigma(i)}\big(i-\sigma(i)\big),
\]
where the second equality follows by setting \(i=\sigma^{-1}(k)\). 
Moreover,
\[
0=\sum_{i=1}^n(\sigma(i)-i)=\sum_{\sigma(i)>i}\big(\sigma(i)-i\big)+\sum_{i>\sigma(i)}\big(i-\sigma(i)\big).
\]
Combining the above two identities we obtain  \(\depth(\sigma^{-1})=\depth(\sigma)\).

\end{proof}

    Substituting $\mathbf{u}=(x,1,x,1)$ and $s=\alpha=\beta=1$ in Theorem~\ref{thm2} we obtain
    \begin{equation}\label{equ:continud-fractions-depth}
        1+\sum_{n\ge 1}\bigg(\sum_{\sigma\in\S_n}x^{\exc(\sigma)}t^{\fix(\sigma)}h^{\depth(\sigma)}p^{\cros(\sigma)}q^{\nest(\sigma)}\bigg)z^n=\mathcal{J}(z;\mathbf{b}, \boldsymbol{\lambda}),
    \end{equation}
    where
    $b_k=h^k\big(tq^k+(1+xp)[k]_{p,q}\big)$ and $\lambda_{k}=xh^{2k-1}[k]_{p,q}^2$.

 By Lemma~\ref{final lemma}, since  $n=\exc(\sigma)+\fix(\sigma)+\drop(\sigma)$ for $\sigma\in \S_n$,  substituting 
 $(x,t)\to (\frac{1}{x},\frac{t}{x})$ and then $z\to xz$ in \eqref{equ:continud-fractions-depth},
 we obtain 
\begin{equation}\label{T-Eu-equ}
    \sum_{n\ge 0}\bigg(\sum_{\sigma\in\S_n}x^{\exc(\sigma)}t^{\fix(\sigma)}p^{\cros^*(\sigma)}q^{\nest^*(\sigma)}h^{\depth(\sigma)}\bigg)z^n=\mathcal{J}(z;\mathbf{b},\boldsymbol{\lambda}),
\end{equation}   
where $b_k=h^k(x[k]_{p,q}+tq^k+p[k]_{p,q})$ and $\lambda_{k}=xh^{2k-1}[k]^2_{p,q}$.

When $t=1$, the above $J$-fraction reduces to
\begin{equation}\label{Eu-equ}
\sum_{n\ge 0}\bigg(\sum_{\sigma\in\S_n}
x^{\exc(\sigma)}
p^{\cros^*(\sigma)}
q^{\nest^*(\sigma)}
h^{\depth(\sigma)}\bigg)z^n
=\mathcal{J}(z;\mathbf{b},\boldsymbol{\lambda}),
\end{equation}
where
$b_k=h^k\!\left(x[k]_{p,q}+[k+1]_{p,q}\right)$ and
$\lambda_k=xh^{2k-1}[k]_{p,q}^2$.
This $J$-fraction was obtained by Eu et al.~\cite[Theorem~1.2]{Eu26}.
They further observed that it exhibits a symmetry
$\nest^*\leftrightarrow\cros^*$ for fixed values of $\exc$ and $\depth$,
and asked for a bijective proof of this symmetry.
The following theorem provides such a bijection.
\begin{theorem}
Let $\iota: \S_n\to \S_n,\; \iota(\sigma)=\sigma^{-1}$ and $\Phi^*=\Phi_{SZ}\circ\Phi_{C}^{-1}$. Then the  mapping  $\Theta:=\iota\circ\Phi^*\circ\iota:\mathfrak{S}_n\to\mathfrak{S}_n$ is a bijection such that for any $\sigma\in \S_n$,
 \begin{align}\label{Bijection}    (\exc,\nest^*,\cros^*,\depth)\,\sigma=(\exc,\cros^*,\nest^*,\depth)\,\Theta(\sigma)
\end{align}
\end{theorem}
\begin{proof}
By \eqref{equ:p_i+q_i-2'} of Theorem \ref{thm3} and $n=\wex+\drop$, the bijection $\Phi^*$ satisfies
\begin{equation}\label{equ:p_i+q_i-2}(\drop,\nest,\cros,\depth)\,\sigma=(\drop,\cros,\nest,\depth)\,
\Phi^*(\sigma).
    \end{equation}
Combining \eqref{property} and \eqref{equ:p_i+q_i-2}, we obtain \eqref{Bijection}.
\end{proof}

    From the construction of the bijections $\Phi_{SZ}$ and $\Phi_{C}$, we see that the two bijections satisfy the  following identity 
\begin{align}     \Phi_{SZ}\circ\Phi_{C}^{-1}&=\widetilde{\Psi}_{FZ}^{-1}\circ\rho\circ\widetilde{\Psi}_{FV}\circ(\widetilde{\Psi}_{FZ}^{-1}\circ\widetilde{\Psi}_{FV})^{-1}\nonumber\\
     &=\widetilde{\Psi}_{FZ}^{-1}\circ\rho\circ\widetilde{\Psi}_{FZ},
\end{align}
  where $\widetilde{\Psi}_{FZ}$ and $\rho$ are defined in \eqref{equ:corresponding-FZ*} and \eqref{def:rho}. 
  
  Besides the bijections $\Phi_{C}$ and $\Phi_{SZ}$, Steingrímsson and Williams~\cite{SW07} constructed another  variant  $\Phi_{SW}$ of $\Phi_{CSZ}$ satisfying 
\begin{equation}
  (\ndes,\ldes,2\!-\!31)\,\sigma=(\drop+1,\nest,\cros)\,\Phi_{SW}(\sigma)\quad \text{for}\quad \sigma\in \S_n.
\end{equation}

Let $\mathrm{Ddif}(\sigma):=\mathrm{Destop}(\sigma)-\mathrm{Desbot}(\sigma)$ for $\sigma\in \S_n$.   
It is known \cite[Proposition~3]{CSZ97}  that 
the mapping $\Phi_{CSZ}$ implies that the two bi-statistics
\[
(\des,\mathrm{Ddif})\; \textrm{ and }\; (\exc,\depth)
\]
are equidistributed on $\S_n$. Two alternative proofs of this result were given  by 
Reifegerste \cite{Rei02} 
and Petersen and Tenner~\cite{PT15} using variants of the \emph{fundamental transformation}.


 
\subsection*{Acknowledgement}
 The first author was supported by the China Scholarship Council (No. 202206220034).


\bibliographystyle{plain}
\bibliography{refs}
\end{document}